\documentclass[a4paper,10pt]{article}
\usepackage{amssymb,amscd,amsfonts,amsbsy}
\usepackage{enumerate}
\usepackage{amsmath,amsthm,amssymb,amsfonts}
\usepackage{mathrsfs}
\usepackage{epsf,epsfig}

\input amssymb.sty
\input euscript.sty

\newtheorem{proposition}{Proposition}[section]
\newtheorem{theorem}[proposition]{Theorem}
\newtheorem{lemma}[proposition]{Lemma}
\newtheorem{corollary}[proposition]{Corollary}

\newtheorem{remark}[proposition]{Remark}

\newtheorem{definition}[proposition]{Definition}

\begin{document}

\title{{The Integrability of Negative Powers of the Solution of 
the Saint Venant Problem}
\thanks{2000 {\it Math Subject Classification.} Primary: 31B05, 31B25, 
35B09, 35B33, 35B50; Secondary 31A05, 33A15, 33A50, 35B40, 35J25.
\newline
{\it Key words}: superharmonic functions, Saint Venant problem, 
integrability, Maximum Principle, barrier function, nonsmooth domains, 
sublevel set estimates}}

\author{Anthony Carbery, Vladimir Maz'ya, Marius Mitrea and David
  J. Rule \\
To appear,  Annali Della Scuola Normale Superiore, Part 2, Vol. XIII, 
series V, June 2014.}

\date{23rd June 2012}

\maketitle

\begin{abstract}
We initiate the study of the finiteness condition 
$\int_{\Omega}u(x)^{-\beta}\,dx\leq C(\Omega,\beta)<+\infty$ where 
$\Omega\subseteq{\mathbb{R}}^n$ is an open set and $u$ is the solution 
of the Saint Venant problem $\Delta u=-1$ in $\Omega$, $u=0$ on 
$\partial\Omega$. The central issue which we address is that of 
determining the range of values of the parameter $\beta>0$ for which 
the aforementioned condition holds under various hypotheses on the 
smoothness of $\Omega$ and demands on the nature of the constant 
$C(\Omega,\beta)$. Classes of domains for which our analysis applies 
include bounded piecewise $C^1$ domains 
in ${\mathbb{R}}^n$, $n\geq 2$, with conical singularities (in particular
polygonal domains in the plane), polyhedra in ${\mathbb{R}}^3$, and 
bounded domains which are locally of class $C^2$ and which 
have (finitely many) outwardly pointing cusps. For example, we show that
if $u_N$ is the solution of the Saint Venant problem in the regular polygon 
$\Omega_N$ with $N$ sides circumscribed by the unit disc in the plane,
then for each $\beta\in(0,1)$ the following asymptotic formula holds:
\begin{eqnarray*}
\int_{\Omega_N}u_N(x)^{-\beta}\,dx=\frac{4^\beta\pi}{1-\beta}
+{\mathcal{O}}(N^{\beta-1})\quad\mbox{as }\,\,N\to\infty.
\end{eqnarray*}
One of the original motivations for addressing the aforementioned issues 
was the study of sublevel set estimates for functions $v$ 
satisfying $v(0)=0$, $\nabla v(0)=0$ and $\Delta v\geq c>0$.  
\end{abstract}

\section{Introduction}
\label{sect:1}
\setcounter{equation}{0}

\subsection{Background}

Suppose that $u$ is a positive superharmonic function defined in an open, 
bounded subset $\Omega$ of ${\mathbb{R}}^n$, i.e. 
\begin{eqnarray}\label{In-T1}
\Delta u\leq 0\quad\mbox{ and }\quad u>0\quad\mbox{ in }\quad\Omega. 
\end{eqnarray}
Two issues which have received a considerable amount of attention in the 
literature are:
\begin{enumerate}
\item[(i)] proving lower pointwise bounds 
for $u$ in terms of powers of the distance function to the boundary, and 
\item[(ii)] establishing the membership of $u$ to the Lebesgue scale 
$L^p(\Omega)$, $0<p\leq\infty$. 
\end{enumerate}
See, for example, \cite{Arm1}, \cite{Arm2}, \cite{Ai2}, \cite{AE}, 
\cite{Ku}, \cite{KS}, \cite{MS}, \cite{Mas}, \cite{SU} and the 
references therein. We wish to highlight two aspects of 
the philosophy that has emerged from these studies. 
First, granted a certain degree of reasonableness of the 
underlying domain, for superharmonic functions, positivity always 
entails a quantitative version of itself, in the form of the estimate
\begin{eqnarray}\label{In-T2}
u(x)\geq C(\Omega,u)\,\delta_{\Omega}(x)^{\alpha},\qquad
\mbox{for all }\,\,x\in\Omega,
\end{eqnarray}
where $C(\Omega,u)>0$ is a constant depending on $u$ and $\Omega$,
for some exponent $\alpha=\alpha(\Omega)\geq 1$ independent of $u$. 
Here and elsewhere, for an arbitrary set $\Omega\subseteq{\mathbb{R}}^n$, 
we have denoted by $\delta_{\Omega}$ the (Euclidean) distance 
to its boundary, i.e., 
\begin{eqnarray}\label{PKc-1.C}
\delta_{\Omega}(x):={\rm dist}\,(x,\partial\Omega),
\qquad\forall\,x\in{\mathbb{R}}^n.
\end{eqnarray}

The second aspect alluded to above is that there is a common integrability 
threshold for the entire class of positive 
superharmonic functions in the sense that 
\begin{eqnarray}\label{In-T3}
\int_{\Omega}u(x)^{p}\,dx<+\infty,
\end{eqnarray}
for some integrability exponent $p=p(\Omega)>0$ independent of the 
positive superharmonic function $u$ in $\Omega$. 

The specific nature of the exponents $\alpha(\Omega)$ and $p(\Omega)$
is dictated by the degree of regularity exhibited by $\Omega$. 
For example, \eqref{In-T2} has been proved for $\alpha=1$ in a suitable 
subclass of the class of domains satisfying a uniform interior ball
condition which, in turn, contains the class of bounded $C^2$ domains, 
by Kuran in \cite{Ku}, and for bounded planar Jordan domains 
with a Dini-continuous boundary by Kuran and Schiff 
in \cite{KS}. On the other hand, the lower bounds for the Green function 
established in \cite{MS} for Lipschitz domains also lead to estimates of 
the type \eqref{In-T2}, typically for exponents larger than one. 

As far as \eqref{In-T3} is concerned, in the case when 
$\Omega\subseteq{\mathbb{R}}^n$ is a bounded $C^\infty$ domain, 
Armitage \cite{Arm1}, \cite{Arm2} has proved that \eqref{In-T3} 
holds for any positive superharmonic function $u$ in $\Omega$, 
granted that $0<p<n/(n-1)$. This result has been subsequently extended 
by Maeda and Suzuki in \cite{MS} to the class of bounded Lipschitz 
domains for a range of $p$'s which depends on the Lipschitz constant 
of the domain in question, in such a way that $p\nearrow n/(n-1)$ as 
the domain is progressively closer and closer to being of class $C^1$ 
(i.e., as the Lipschitz constant approaches zero). Further refinements 
of this result, in the class of John domains and H\"older domains 
(in which scenario $p$ is typically small), have been studied, 
respectively by Aikawa in \cite{Ai2} and by Stegenga and 
Ullrich in \cite{SU}.

\subsection{Overview and motivation}

In this paper we are concerned with the validity of \eqref{In-T3} 
for {\em{negative}} values of the integrability exponents, in the case when 
$u$ is a positive function with $\Delta u<0$ in $\Omega$. 
A case in point is the solution of the Saint Venant
problem\footnote{For definiteness, a unique solution will exist, say,
  in $W^{1,2}_0(\Omega)$ when $\Omega$ is a bounded open set (see
  \S\ref{sect:2}). Much of our analysis will apply in the more general
setting where we assume that $\Delta u \leq - C_n\,\,\mbox{ in }\,\,\Omega$, with $C_n$
being a positive dimensional constant.} 
(cf., e.g., \cite{Av}, \cite{CJP}, \cite{DE})
\begin{eqnarray}\label{PKxc-1}
\left\{
\begin{array}{l}
\Delta u=-1\,\,\mbox{ in }\,\,\Omega,
\\[4pt]
u=0\,\,\,\,\mbox{ on }\,\,\partial\Omega,
\end{array}
\right.
\end{eqnarray}
and the question which makes the object of our study is that 
of determining the range of values of the parameter $\beta>0$ 
for which an estimate of the form 
\begin{eqnarray}\label{PKc-4}
\int_{\Omega}u(x)^{-\beta}\,dx\leq C(\Omega,\beta)<+\infty
\end{eqnarray}
holds, under various conditions on $\Omega$ and demands on the nature of 
the constant $C(\Omega,\beta)$. Cases of special interest include the
class of nontangentially accessible domains satisfying an inner cone 
condition (which includes the class of Lipschitz domain) in ${\mathbb{R}}^n$, 
polygonal domains in ${\mathbb{R}}^2$, polyhedral domains in ${\mathbb{R}}^3$, 
as well as piecewise smooth domains with conical and cuspidal singularities.
Since $\int_{\Omega}u(x)^{-\beta}\,dx$ is entirely determined by the domain 
$\Omega$ and the parameter $\beta$, we shall occasionally refer to this 
number as the ``$\beta$-integral of $\Omega$.''

Aside from its relevance in potential theory, the problem \eqref{PKxc-1} plays a significant role in elasticity theory.
For example, the torsional rigidity coefficient of $\Omega$, originally 
defined as 
\begin{eqnarray}\label{Hgva}
P(\Omega):=\sup_{0\not=w\in C^\infty_0(\Omega)}
\Bigl(\int_{\Omega}|w|\,dx\Bigr)^2
\Bigl(\int_{\Omega}|\nabla w|^2\,dx\Bigr)^{-1}
\end{eqnarray}
turns out to be 
\begin{eqnarray}\label{Hgvb}
P(\Omega)=\int_{\Omega}u\,dx=\int_{\Omega}|\nabla u|^2\,dx,
\end{eqnarray}
where $u$ is the solution of \eqref{PKxc-1} (cf.~the discussion in 
\cite{Av}, \cite{CF}, \cite{KM}). 

Our interest in the estimate \eqref{PKc-4} was originally motivated 
by problems in harmonic analysis concerning sublevel set estimates 
for a real-valued, strictly convex function of class $C^2$ defined in an 
open, convex set $\Omega\subseteq{\mathbb{R}}^n$. (It is thus also related 
to the behaviour of oscillatory integrals; cf.~\cite{Ca}, \cite{CCW}, \cite{S}.)
It this vein, we recall that it has been 
shown in \cite{Ca} that there exists a finite dimensional constant $C=C_n>0$ 
with the property that, with $|E|$ denoting the Lebesgue 
measure\footnote{Later on, we shall also occasionally
use the notation ${\mathcal{L}}^n(E)$ in place of $|E|$.} of a 
Lebesgue measurable set $E$,  
\begin{eqnarray}\label{sub-1.m}
|\Omega|\leq C\|v\|^{n/2}_{L^\infty(\Omega)},
\end{eqnarray}
provided that, in addition to the already mentioned properties, 
the Hessian of the function $v$ satisfies 
\begin{eqnarray}\label{vab-3}
\det\Bigl[\Bigl(\frac{\partial^2v}{\partial x_{i}\partial x_{j}}\Bigr)
_{1\leq i,j\leq n}\Bigr]\geq 1\quad\mbox{ on }\,\,\Omega.
\end{eqnarray}
As noted in \cite{Ca}, if $v$ is also nonnegative, then by applying 
\eqref{sub-1.m} with $\{x\in\Omega:\,v(x)<t\}$, $t>0$, in place of $\Omega$ 
we obtain the sublevel set estimate 
\begin{eqnarray}\label{vab-4}
\bigl|\{x\in\Omega:\,v(x)<t\}\bigr|\leq C\,t^{n/2},\qquad t>0.
\end{eqnarray}
On the other hand, granted \eqref{vab-3}, the arithmetic-geometric mean 
inequality gives  
\begin{eqnarray}\label{PKc-1a4}
\Delta v/n\geq\bigl(\det((\partial_{ij}v)_{1\leq i,j\leq n})\bigr)^{1/n}\geq 1.
\end{eqnarray}
Hence, it is natural to ask, what happens with \eqref{sub-1.m}
if we only knew $\Delta v\geq n$? Is it reasonable to expect to still have 
such an estimate which, by the same procedure as above, would then lead to a  
sub-level set estimate similar to \eqref{vab-4}? If so, what is the nature 
of the constant $C$ in \eqref{vab-4} in this more general situation?

We wish to elaborate on this point and, in particular, make it more 
transparent how condition \eqref{PKc-4} for the solution of \eqref{PKxc-1} 
comes into play. To set the stage, assume that $v$ is a real-valued, 
strictly convex function $v$ of class $C^2$ in a neighbourhood of 
the origin in ${\mathbb{R}}^n$ and which is normalised so that 
\begin{eqnarray}\label{sub-1}
v(0)=0,\qquad \nabla v(0)=0.
\end{eqnarray}
Next, fix a (small) threshold $t>0$, define 
\begin{eqnarray}\label{MMM.s}
\Omega:=\{x:\,v(x)<t\}\subseteq{\mathbb{R}}^n, 
\end{eqnarray}
and, from now on, restrict $v$ to the open convex set $\Omega$. 
To continue, denote by ${\mathcal{G}}$ the region of space in 
${\mathbb{R}}^{n+1}$ lying directly above the graph of the function 
$v$ and below the $n$-dimensional horizontal plane $x_{n+1}=t$, i.e., 
\begin{eqnarray}\label{sub-2}
{\mathcal{G}}:=\{(x,x_{n+1})\in{\mathbb{R}}^n\times{\mathbb{R}}:\,
x\in\Omega\mbox{ and } v(x) <x_{n+1}< t\}.
\end{eqnarray}
In order to estimate $|{\mathcal{G}}|$, 
the $(n+1)$-dimensional Lebesgue measure of ${\mathcal{G}}$, let $u$ solve 
the auxiliary problem \eqref{PKxc-1}. We then have 
\begin{eqnarray}\label{sub-4}
|{\mathcal{G}}|=\int_{\Omega}(t-v(x))\,dx
=\int_{\Omega}(v(x)-t)(\Delta u)(x)\,dx
=\int_{\Omega}(\Delta v)(x)u(x)\,dx, 
\end{eqnarray}
after integrating by parts and using the fact that both $u$ and 
$v-t$ vanish on $\partial\Omega$. Using this formula, for given 
$\gamma\in(0,1/2)$ we may then compute (making use of the 
obvious inequality $|{\mathcal{G}}|\leq t\,|\Omega|$)
\begin{eqnarray}\label{sub-5}
\int_{\Omega}(\Delta v)^\gamma\,dx &=& 
\int_{\Omega}\bigl((\Delta v)u\bigr)^\gamma u^{-\gamma}\,dx 
\leq\Bigl(\int_{\Omega}(\Delta v)u\,dx\Bigr)^\gamma
\Bigl(\int_{\Omega}u^{-\gamma/(1-\gamma)}\,dx\Bigr)^{1-\gamma}
\nonumber\\[4pt]
&=& |{\mathcal{G}}|^\gamma\Bigl(\int_{\Omega}
u^{-\gamma/(1-\gamma)}\,dx\Bigr)^{1-\gamma}
\leq t^\gamma|\Omega|^{\gamma}
\Bigl(\int_{\Omega}u^{-\gamma/(1-\gamma)}\,dx\Bigr)^{1-\gamma},
\end{eqnarray}
hence
\begin{eqnarray}\label{sub-5CX}
\int_{\Omega}(\Delta v)^\gamma\,dx 
\leq \|v\|_{L^\infty(\Omega)}^{\gamma}|\Omega|^{\gamma}
\Bigl(\int_{\Omega}u^{-\beta}\,dx\Bigr)^{1-\gamma},
\end{eqnarray}
where have set $\beta:=\frac{\gamma}{1-\gamma}\in(0,1)$. 
Note that in the case in which 
\begin{eqnarray}\label{sub-8}
\int_{\Omega}u(x)^{-\beta}\,dx\leq C_{\beta}\,|\Omega|^{1-2\beta/n},
\end{eqnarray}
this analysis gives 
\begin{eqnarray}\label{sub-10}
\int_{\Omega}(\Delta v)^{\gamma}\leq C_{\gamma}
\|v\|_{L^\infty(\Omega)}^{\gamma}\,|\Omega|^{1-\frac{2\gamma}{n}}.
\end{eqnarray}
The upshot of this analysis is that by using the weaker condition 
$\Delta v\geq n$ (in place of the quantitative non-degeneracy of the Hessian matrix 
for $v$, as in \eqref{vab-3}), one deduces from \eqref{sub-10} that
\begin{eqnarray}\label{Paaa.2}
n^{\gamma}|\Omega|\leq C_{\gamma}\,t^{\gamma}|\Omega|^{1-2\gamma/n},
\end{eqnarray}
which leads to 
\begin{eqnarray}\label{Paaa.3}
|\{x:\,v(x) < t\}|\leq D \, t^{n/2},\qquad t>0,
\end{eqnarray}
with $D$ depending only on the dimension $n$. This is of course contingent upon \eqref{sub-10}
holding for some $\gamma \in (0, 1/2)$ (possibly depending on $n$) with the constant 
$C_{\gamma}$ (which is related to $C_\beta$ from \eqref{sub-8} via 
$C_{\gamma}=(C_\beta)^{1-\gamma}$) being independent of the parameter $t$. 

However, we cannot expect an inequality such as \eqref{Paaa.3} to hold
for an arbitrary strictly convex $v$, defined on a convex domain
containing $0$, which satisfies \eqref{sub-1}. For example consider,
for small $\epsilon$, the function $v_\epsilon(x) = x_1^2 + \epsilon x_2^2$ 
defined on $\mathbb{R}^2$, for which  \eqref{Paaa.3} is easily seen to fail. 
Upon reflection, this is related to the fact that inequality \eqref{sub-10}, 
considered for arbitrary convex domains $\Omega$ and strictly convex $v$ 
defined on $\Omega$, is dilation invariant, but, unlike its counterpart for 
the Hessian problem, is not affine invariant. Thus we cannot expect inequality 
\eqref{sub-8} to hold uniformly over all convex domains $\Omega$, and 
indeed at the end of Section~\ref{sect:3} we demonstrate this explicitly. 
On the other hand, Proposition \ref{PKc-1.F} below shows that for convex 
sets $\Omega$ containing the unit ball and contained in some dimensional
multiple of the unit ball, \eqref{sub-8} does hold for $\beta<1/2$.
While it is not clear whether \eqref{sub-8} holds for such sets $\Omega$ 
for all $\beta<1$ with a constant depending only on $\beta$ and $n$, 
related results (in the two-dimensional setting) are obtained in 
Theorem~\ref{Te-3A.u} and Proposition~\ref{taTH} below.

\subsection{Description of results and layout of the paper}

The discussion in \S\,1.1-\S\,1.2 highlights the significance of the problem 
\eqref{PKxc-1} as well as the relevance of the finiteness 
condition \eqref{PKc-4}. Note that the solution $u$ of 
\eqref{PKxc-1} satisfies  $1/u\in L^\infty_{loc}(\Omega)$, so the finiteness
condition in \eqref{PKc-4} is related to the rate at which $u$ vanishes 
on the boundary. While, from this point of view, a pointwise lower bound 
such as \eqref{In-T2} provides, in principle, a venue for deducing an estimate 
of the form \eqref{PKc-4}, the range of negative integrability exponents 
obtained by such a method is typically far from optimal, so a number of new 
ideas are required. A succinct summary of our main results is as follows: 

\begin{theorem}\label{KamT}
The $\beta$-integral associated with a bounded domain 
$\Omega\subseteq{\mathbb{R}}^n$ is finite in any of the following situations:
\begin{enumerate}
\item[(i)] $\beta\in(0,1)$ and $\Omega$ is a bounded piecewise $C^1$ domain 
in ${\mathbb{R}}^n$, $n\geq 2$, with conical singularities; 
\item[(ii)] $\beta\in(0,1)$ and $\Omega$ is a polyhedron in ${\mathbb{R}}^3$;
\item[(iii)] $\beta\in(0,1)$ and $\Omega\subseteq{\mathbb{R}}^n$, $n\geq 2$, 
is a bounded domain, locally of class $C^2$ and which has an outwardly pointing 
cusp at $0\in\partial\Omega$. Specifically, it is assumed that there exists 
a small number $\varepsilon>0$ and a function ${\mathcal{F}}\in C^2([0,1])$ 
with ${\mathcal{F}}(0) = 0$, ${\mathcal{F}}>0$ on $(0,1]$ and ${\mathcal{F}}'(0)=0$, 
for which $\{x\in\Omega:\,x_n\leq 1\}$ coincides with the cuspidal set
$\{x=(x',x_n):\,0<x_n\leq 1,\,\,|x'|<\varepsilon{\mathcal{F}}(x_n)\}$. 
In the case when $n=2$ and $\beta\in(1/2,1)$, the following (necessary) 
finiteness condition is also assumed: 
\begin{eqnarray}\label{ddXCT.i}
\int_0^1{\mathcal{F}}(\tau)^{1-2\beta}\,d\tau<+\infty.
\end{eqnarray}
\end{enumerate}
\end{theorem}

Of course, in part (iii) of Theorem~\ref{KamT}, the same type of result holds 
for any bounded piecewise $C^2$ domain with (finitely many) exterior cusps.  
We also wish to emphasise that part (i) of Theorem~\ref{KamT} covers, 
in particular, the case of polygons in the plane. One special case 
is treated in Proposition~\ref{taTH} where it is shown that, for each fixed 
$\beta\in(0,1)$, the $\beta$-integral of a regular polygon with $N$ sides which
is circumscribed by the unit disc has following asymptotic 
\begin{eqnarray}\label{HGba-1Xi}
\frac{4^\beta\pi}{1-\beta}
+{\mathcal{O}}(N^{\beta-1})\qquad\mbox{as }\,\,N\to\infty.
\end{eqnarray}

In the case of a bounded piecewise $C^1$ domain
$\Omega$ with conical singularities, our approach is to estimate 
the contribution from individual conical points by carefully devising 
appropriate barrier functions which compare favourably with 
the solution of \eqref{PKxc-1}. The contribution from the region 
$\Omega$ away from the boundary singularities is then estimated 
separately, by relying on results valid on smooth domains. 
See Theorem~\ref{Te-3A} and Theorem~\ref{Te-3A.u} which are the main 
results phrased in the two-dimensional setting, as well as 
Theorem~\ref{Te-3AH} which contains an extension to the higher 
dimensional case. 

Let us now review the content of the various sections of this paper. 
In Section~\ref{sect:2} we derive estimates for the solution of 
the Saint Venant problem 
in rather general domains, satisfying weak regularity properties,  
described in terms of basic geometric measure theoretic conditions. 
This portion of our analysis points to the value $\beta=1/2$ as the
natural critical exponent for the condition \eqref{PKc-4} in this
degree of generality for the underlying domain $\Omega$. 
Improvements of this result in the case when $\Omega$ satisfies an inner
cone condition (cf.~Definition~\ref{Hga.4}) are subsequently discussed in 
Section~\ref{sect:3}. As a preamble, here we briefly review the 
construction and properties of classical barrier functions in cones.
We then derive a lower pointwise bound for the Green function 
(akin to work in \cite{MS}) which is then used to prove Theorem~\ref{TL-AS},
the main result in this section. A consequence of this theorem 
is that \eqref{PKc-4} holds for any $\beta\in(0,1)$ 
in the case when $\Omega$ is a bounded $C^1$ domain. 

Sections~\ref{sect:4}-\ref{sect:5} are devoted to studying the 
class of bounded piecewise $C^1$ domains with conical singularities, 
and the main results here are Theorem~\ref{Te-3A} and Theorem~\ref{Te-3AH}. 
The new phenomenon which we discover is that, 
much as for bounded $C^1$ domains, \eqref{PKc-4} continues to hold 
for every $\beta\in(0,1)$ in the aforementioned class of piecewise 
$C^1$ domains with conical singularities in ${\mathbb{R}}^n$.
Our approach is based on the realisation that, for the type 
of domains considered in these sections, the size (smallness, in fact)
of the solution of the Saint Venant problem \eqref{PKxc-1} is, 
in the process of taking an integral average, better controlled 
than pointwise estimates from below in terms of powers of the 
distance function to the boundary might originally seem to indicate.

Finally, in Section~\ref{sect:6}, we study the veracity of \eqref{PKc-4}
for other classes of domains with isolated singularities, such as 
polyhedra and piecewise $C^1$ domains with outwardly pointing cuspidal 
singularities.

\vskip 0.08in
Throughout, we employ the customary convention of using the same letter
for denoting constants whose values may change from line to line. 
Whenever the dependence of the constants in question on certain parameters 
is important, we indicate this as such. Also, $F\approx G$ means that there 
exist $C_1,C_2>0$ which are independent of the relevant parameters entering 
the expressions $F,G$ with the property that $C_1F\leq G\leq C_2F$.

\vskip 0.08in
\noindent{\it Acknowledgments.} Portions of this work have been undertaken 
while the third-named author was visiting the Centre for Analysis and 
Nonlinear PDE (CANPDE) at the University of Edinburgh, while on research 
leave from University of Missouri. He gratefully acknowledges the support 
received from these institutions and the US National Science Foundation 
grant DMS-0653180. The second author was partially supported by the UK
Engineering and Physical Sciences Research Council grant EP/F005563/1.
The fourth author would also like to acknowledge support from the CANPDE.

Last, though not least, the authors are grateful to the anonymous referees for 
their diligent reading of the manuscript and for making a number of insightful 
comments which have led to the current version. 

\section{Estimates for the Saint Venant problem in rough domains}
\label{sect:2}
\setcounter{equation}{0}

Let $\Omega$ be a bounded, open subset in ${\mathbb{R}}^n$, and denote by 
$W^{1,p}(\Omega)$ the classical $L^p$-based Sobolev space of order one in 
$\Omega$, where $1\leq p\leq\infty$. Furthermore, we shall use 
$W^{1,p}_0(\Omega)$ to denote the closure of $C^\infty_0(\Omega)$ in 
$W^{1,p}(\Omega)$. A standard application of the Lax-Milgram lemma shows that 
the Saint Venant problem \eqref{PKxc-1} has a unique solution in the energy 
space $W^{1,2}_0(\Omega)$, i.e., 
\begin{eqnarray}\label{PKc-1}
u\in W^{1,2}_0(\Omega),\qquad \Delta u=-1\,\,\mbox{ in }\,\,\,\Omega,
\end{eqnarray}
is always well-posed. In fact, the solution $u$ of \eqref{PKc-1} can be 
expressed as 
\begin{eqnarray}\label{PKc-2}
u(x)=\int_{\Omega}G(x,y)\,dy,\qquad x\in\Omega, 
\end{eqnarray}
where $G(\cdot,\cdot)$ is the Green function for the Dirichlet Laplacian 
in $\Omega$. The latter is the unique function 
$G:\Omega\times\Omega\to [0,+\infty]$ satisfying  
\begin{eqnarray}\label{Do-1}
G(\cdot,y)\in W^{1,2}(\Omega\setminus B(y,r))
\cap W^{1,1}_0(\Omega),\qquad\forall\,y\in\Omega,\,\,\,
\forall\,r>0, 
\end{eqnarray}
\noindent and 
\begin{eqnarray}\label{con1}
\int_\Omega\langle \nabla_x G(x,y),\nabla\varphi(x)\rangle\,dx=\varphi(y),
\quad\forall\,\varphi\in C^\infty_0(\Omega). 
\end{eqnarray}
See, e.g., \cite{GW} and \cite{Ke} for the proof of the existence 
and uniqueness of the Green function; a number of other useful 
properties of the Green function can be found in these works, such as
the fact that the Green function is symmetric (i.e., $G(x,y)=G(y,x)$ 
for all $x,y\in\Omega$) and satisfies the estimates (valid for $n\geq 3$) 
\begin{eqnarray}\label{TF-1X}
&& G(x,y)\leq C_n|x-y|^{2-n}\quad\mbox{ for all $x,y\in\Omega$},
\\[4pt]
&& G(x,y)\geq C_n|x-y|^{2-n}\quad\mbox{ for $x,y\in\Omega$ with  
$|x-y|\leq {\textstyle\frac12}\,\delta_{\Omega}(x)$},
\label{TF-1}
\end{eqnarray}
where the constants depend only on the dimension. The replacement for
$|x-y|^{2-n}$ in the case when $n=2$ is $\log({\rm diam}\,(\Omega)/|x-y|)$.
Hence, as a consequence of \eqref{PKc-2} and \eqref{TF-1X}, 
\begin{eqnarray}\label{PKc-3}
0<u(x)\leq C_n\,[{\rm diam}\,(\Omega)]^2,\,\,\,\,\mbox{ for each }\,\,
x\in\Omega.  
\end{eqnarray}

\begin{remark}\label{gaba}
Let $\Omega\subseteq{\mathbb{R}}^n$ be an arbitrary open set. 
Then for every $\beta>0$ the solution of \eqref{PKc-1} satisfies 
the bound from below
\begin{eqnarray}\label{hanan}
C(n,\beta)|\Omega|\,[{\rm diam}\,(\Omega)]^{-2\beta}
\leq\int_{\Omega}u(x)^{-\beta}\,dx. 
\end{eqnarray} 
Indeed, $|\Omega|=\int_{\Omega}u(x)^{\beta}u(x)^{-\beta}\,dx\leq C_n^\beta
[{\rm diam}\,(\Omega)]^{2\beta}\int_{\Omega}u(x)^{-\beta}\,dx$, by \eqref{PKc-3}.
In particular, if $\Omega$ has the property that 
$B(0,1)\subseteq\Omega\subseteq B(0,C_n)$, then for every $\beta>0$ there holds
\begin{eqnarray}\label{haTGn}
C(n,\beta) \leq\int_{\Omega}u(x)^{-\beta}\,dx.
\end{eqnarray}
\end{remark}

We wish to point out that in the case when $\Omega$ is regular for the 
Dirichlet problem (i.e., the classical Dirichlet problem is well-posed in the 
class of continuous functions), one actually has $u\in C^0(\overline{\Omega})$. 
Necessary and sufficient criteria for regularity are well-known. 
For example, any bounded open set $\Omega\subseteq{\mathbb{R}}^n$ 
is regular for the Dirichlet problem if it satisfies an exterior 
corkscrew condition \cite[Lemma 1.2.4]{Ke}. The latter piece of terminology is clarified in the 
definition below.  
\begin{definition}\label{TFG-54}
We say that $\Omega\subset{\mathbb{R}}^n$ satisfies an interior 
corkscrew condition if there are constants $M>1$ and $R>0$ such 
that for each $x\in\partial\Omega$ and $r\in (0,R)$ there exists 
\begin{eqnarray}\label{Cr-Pt}
\begin{array}{l}
\mbox{$A_r(x)\in\Omega$, called corkscrew point relative to $x$}, 
\\[4pt]
\mbox{so that $|x-A_r(x)|<r$ and 
${\rm dist}(A_r(x),\partial\Omega)>M^{-1}r$}.
\end{array}
\end{eqnarray}
Also, $\Omega\subset{\mathbb{R}}^n$ satisfies the an exterior 
corkscrew condition if $\Omega^c:={\mathbb{R}}^n\setminus\Omega$ 
satisfies an interior corkscrew condition. 
\end{definition}

As explained in \S\,\ref{sect:1}, the central issue in this paper is that 
of determining the ``largest" value of the parameter $\beta>0$ for which an 
estimate of the form \eqref{PKc-4} holds, under various geometrical conditions 
on $\Omega$. Elucidating the nature of the constant $C(\Omega,\beta)$ 
appearing in \eqref{PKc-4} is also of interest. A basic tool systematically
employed throughout the paper is the Maximum Principle. In order to state
a version of this result valid for functions in the Sobolev space 
$W^{1,2}(\Omega)$ we first recall the following definitions.

\begin{definition}\label{subharmonic}
Let $\Omega\subseteq{\mathbb{R}}^n$ be a bounded open set. 
Given $u\in W^{1,2}(\Omega)$, we say that 
$u$ is subharmonic if
\[
\int_\Omega\langle \nabla u(x),\nabla\varphi(x)\rangle\,dx \leq 0
\quad \mbox{for all nonnegative } \varphi\in C^\infty_0(\Omega). 
\]
We say that $u$ is superharmonic if $-u$ is subharmonic. 
\end{definition} 

\begin{definition}\label{Thab}
Let $\Omega\subseteq{\mathbb{R}}^n$ be a bounded open set and assume that
$E\subseteq\overline{\Omega}$. Given $u\in W^{1,2}(\Omega)$, we say that 
$u\geq 0$ on $E$ in the sense of $W^{1,2}(\Omega)$ if there exists a 
sequence $u_j\in C^\infty(\Omega)\cap W^{1,2}(\Omega)$, $j\in{\mathbb{N}}$, 
which converges to $u$ in $W^{1,2}(\Omega)$ and such that, for each 
$j\in{\mathbb{N}}$, there exists an open neighbourhood $U_j$ of $E$ in 
${\mathbb{R}}^n$ with the property that $u_j>0$ in $U_j\cap\Omega$. 
\end{definition}

As is well-known, if $E\subseteq\Omega$ and $u\geq 0$ on $E$ in the 
sense of $W^{1,2}(\Omega)$ then $u\geq 0$ a.e. on $E$. Furthermore, if 
$u\in W^{1,2}(\Omega)$ satisfies $u\geq 0$ a.e. in $\Omega$ then 
$u\geq 0$ in $\Omega$ in the sense of $W^{1,2}(\Omega)$. Let us also point 
out here that if $u\in W^{1,2}(\Omega)\cap C^0(\overline{\Omega})$ 
satisfies $u|_{\partial\Omega}\geq 0$ then $u\geq 0$ on $\partial\Omega$ 
in the sense of $W^{1,2}(\Omega)$ (cf.~\cite{Ke}). 

Analogously to Definition~\ref{Thab}, one can define $u\leq 0$ and $u=0$ 
on $E\subseteq\overline{\Omega}$ in the sense of $W^{1,2}(\Omega)$. 
In particular, this allows one to compare any two functions 
$u,v\in W^{1,2}(\Omega)$ on $E\subseteq\overline{\Omega}$ 
in the sense of $W^{1,2}(\Omega)$, and also to define the supremum and 
infimum of a function $u\in W^{1,2}(\Omega)$ on $E\subseteq\overline{\Omega}$ 
in the sense of $W^{1,2}(\Omega)$. In this context, the following 
version of the Maximum Principle then holds 
(cf.~\cite[Lemma~1.1.17]{Ke}):

\begin{proposition}\label{Max-PP}
Let $\Omega\subseteq{\mathbb{R}}^n$ be a bounded open set and assume 
that $u\in W^{1,2}(\Omega)$ is a subharmonic function in $\Omega$. 
Then 
\begin{eqnarray}\label{Pgav}
\sup_{\Omega}u\leq \sup_{\partial\Omega}u\quad\mbox{ in the sense of }\,\,
W^{1,2}(\Omega).
\end{eqnarray}
\end{proposition}

Returning to the main topic of interest for us here, we continue by making a 
series of simple yet significant remarks. 

\begin{remark}\label{Rf-1.H}
{\rm The case of a ball in ${\mathbb{R}}^n$, i.e., when $\Omega=B(0,R)$, $R>0$, 
in which scenario \eqref{PKc-1} has the explicit solution} 
\begin{eqnarray}\label{PKc-1.A}
u(x)=\frac{1}{2n}(R^2-|x|^2),\qquad x\in B(0,R),
\end{eqnarray}
{\rm shows that we must necessarily have $\beta<1$ and that the critical 
value $\beta=1$ is unattainable. Indeed, the function in \eqref{PKc-1.A}
satisfies} 
\begin{eqnarray}\label{PKc-1.B}
\frac{R}{2n}\,\delta_{B(0,R)}(x)\leq 
u(x)\leq\frac{R}{n}\,\delta_{B(0,R)}(x),\qquad\forall\,x\in B(0,R),
\end{eqnarray}
{\rm Hence, in this case,} 
\begin{eqnarray}\label{PKc-1.E}
\int_{B(0,R)}u(x)^{-1}\,dx=+\infty.
\end{eqnarray}
{\rm In fact, it can be seen that this is typical of any sufficiently 
smooth domain (in fact, Theorem~\ref{TL-AS}, stated later, shows 
that any domain of class $C^1$ will do), namely any $\beta<1$ will 
work in \eqref{PKc-4}.}
\end{remark}

\begin{remark}\label{RG-2.H}
{\rm Regarding the issue whether the $\beta$-integral 
diverges when $\beta=1$, we shall show that this is always the case when 
the underlying domain satisfies the following condition:}
{\it Given $\Omega\subseteq{\mathbb{R}}^n$ and $x_*\in\partial\Omega$, 
we say that $\Omega$ satisfies an enveloping ball condition of radius $R>0$ 
near $x_*$ if there exists $\rho>0$ with the property that 
for every $x\in B(x_*,\rho)\cap\partial\Omega$ there exists a ball of 
radius $R$ which contains $\Omega$ and whose boundary contains $x$.}

{\rm The relevance of this piece of terminology is apparent from the 
following result:} 
{\it Let $\Omega\subseteq{\mathbb{R}}^n$ be a bounded open set 
which satisfies an enveloping ball condition of radius $R>0$ near a point 
$x_*\in\partial\Omega$. Then, if $u$ denotes the solution of the Saint Venant 
boundary value problem \eqref{PKc-1},}
\begin{eqnarray}\label{Tgaab}
u(x)\leq n^{-1}R\,\delta_{\Omega}(x)
\quad\mbox{ for every $x\in\Omega$ near $x_*$}, 
\end{eqnarray}
{\it so that, in particular,} 
\begin{eqnarray}\label{Tgaab.2}
\int_{\Omega} u(x)^{-1}\,dx=+\infty. 
\end{eqnarray}
{\rm To prove the above bound on $u$ in terms of the distance to the 
boundary, consider an arbitrary point $x_0\in B(x_*,\rho/2)\cap\Omega$ and 
denote by $x_1\in\partial\Omega$ a point for which 
$r:=\delta_{\Omega}(x_0)=|x_1-x_0|$. Then, necessarily, 
$x_1\in B(x_*,\rho)\cap\partial\Omega$.
Consider now a ball $B=B(x_2,R)$ which contains $\Omega$ and such that 
$x_1\in\partial B$. Note that $B(x_0,r)\subseteq\Omega\subseteq B$, so that 
the balls $B(x_0,r)$ and $B(x_2,R)$ are tangent at $x_1$. This implies that 
the points $x_1,x_0,x_2$ are collinear hence, further, $R-r=|x_0-x_2|$.
Next, use the Maximum Principle to deduce that}
\begin{eqnarray}\label{Tgaab.3}
u(x)\leq (2n)^{-1}(R^2-|x-x_2|^2)\quad\mbox{\rm for every }\,\,x\in\Omega,
\end{eqnarray}
{\rm which, when specialised to $x=x_0$, gives}
\begin{eqnarray}\label{Tgaab.4}
u(x_0)\leq n^{-1}R(R-|x_0-x_2|)=n^{-1}Rr=n^{-1}R\,\delta_{\Omega}(x_0).
\end{eqnarray}
{\rm Since $x_0\in B(x_*,\rho/2)\cap\Omega$ was arbitrary, \eqref{Tgaab} 
follows. As far as \eqref{Tgaab.2} is concerned, we first note that 
$B(x_*,\rho/2)\cap\Omega$ is convex, hence Lipschitz (a formal definition 
is given later, in \eqref{lipdom}). In turn, this and \eqref{Tgaab} give that} 
\begin{eqnarray}\label{Tgaab.5}
\int_{\Omega} u(x)^{-1}\,dx\geq\int_{B(x_*,\rho/2)\cap\Omega} u(x)^{-1}\,dx
\geq nR^{-1}\int_{B(x_*,\rho/2)\cap\Omega}\delta_{\Omega}(x)^{-1}\,dx=+\infty,
\end{eqnarray}
{\rm where the last step is a simple consequence of the fact that 
$B(x_*,\rho/2)\cap\Omega$ is Lipschitz (a more general result of this 
nature is discussed later, in Remark~\ref{RR.AA}).} 
\end{remark}

\begin{remark}\label{RG-2.Haaa}
{\rm It is useful to note that, given any bounded open set 
$\Omega\subseteq{\mathbb{R}}^n$, the solution of the problem \eqref{PKc-1}
is bounded from below by a multiple (depending only on the dimension $n$) 
of the square of the distance function to the boundary. This property can be 
established in a variety of ways. 
One such approach involves $\delta_{\Omega,{\rm reg}}$, the regularised
distance function to $\partial\Omega$ (in the sense of Theorem~2, 
p.\,171 in \cite{St}). Recall that this is a $C^\infty$ function 
in ${\mathbb{R}}^n$ satisfying  
$\delta_{\Omega,{\rm reg}}\approx\delta_{\Omega}$ and 
which has the property that for each multi-index $\alpha$ there exists
$C_\alpha>0$ such that
\begin{eqnarray}\label{rho-reg23} 
|\partial^{\alpha}\delta_{\Omega,{\rm reg}}(x)|\leq C_{\alpha}
\,\delta_{\Omega}(x)^{1-|\alpha|},\quad\forall\,x\in\Omega. 
\end{eqnarray}
In particular, there exists a finite dimensional constant $C>0$ 
with the property that $|\Delta(\delta_{\Omega,{\rm reg}}^2)(x)|\leq C$ 
for all $x\in\Omega$. This implies that $u-C^{-1}\delta_{\Omega,{\rm reg}}^2$
is superharmonic in $\Omega$, continuous on $\overline{\Omega}$, and 
vanishes on $\partial\Omega$. Hence, by the Maximum Principle, 
$u(x)\geq C_n\,\delta_{\Omega}(x)^2$ for all $x\in\Omega$.
However, the sharp version of this estimate is 
\begin{eqnarray}\label{TF-2.D}
u(x)\geq (2n)^{-1}\delta_{\Omega}(x)^2,\qquad
\mbox{for every $x\in\Omega$}, 
\end{eqnarray}
and this is established as follows. 
Fix an arbitrary point $x_0\in\Omega$ and abbreviate 
$r:=\delta_{\Omega}(x_0)$. Then for every $\varepsilon\in(0,r)$
we have that $\overline{B(x_0,r-\varepsilon)}\subseteq\Omega$ and  
$u\in C^\infty(\overline{B(x_0,r-\varepsilon)})$. Next, consider
the standard  barrier
\begin{eqnarray}\label{T.Dsa} 
v(x):=(2n)^{-1}\bigl((r-\varepsilon)^2 -|x-x_0|^2\bigr),\qquad 
x\in \overline{B(x_0,r-\varepsilon)},
\end{eqnarray}
and note that, by the Maximum Principle and the properties of $u$, 
we have $u\geq v$ in $\overline{B(x_0,r-\varepsilon)}$. In particular,
$u(x_0)\geq v(x_0)$ which gives $u(x_0)\geq (2n)^{-1}(r-\varepsilon)^2$.
Hence, after sending $\varepsilon$ to zero we obtain
$u(x_0)\geq (2n)^{-1}r^2=(2n)^{-1}\delta_{\Omega}(x_0)^2$. 
Given that $x_0\in\Omega$ has been chosen arbitrarily, \eqref{TF-2.D}
is proved.}
\end{remark}

A more refined analysis proves Proposition~\ref{PKc-1.F}, stated below. 
As a preamble, we first recall some definitions, as well as several 
results of independent interest. 

\begin{definition}\label{fUG}
Given an open set $\Omega\subseteq{\mathbb{R}}^n$, the upper and lower 
$\gamma$-dimensional Minkowski contents of $\partial\Omega$ with respect to 
$\Omega$ are defined as 
\begin{eqnarray}\label{MiN-1}
M^{\ast}_\gamma(\partial\Omega)
:=\limsup_{r\to 0^{+}}\omega_{n-\gamma}(r),\qquad
M_{\ast,\gamma}(\partial\Omega):=\liminf_{r\to 0^{+}}\omega_{n-\gamma}(r),
\end{eqnarray}
where, for every $\alpha\in{\mathbb{R}}$, we have set  
\begin{eqnarray}\label{MiN-2}
\omega_{\alpha}(r):=\frac{|\{x\in\Omega:\,\delta_{\Omega}(x)<r\}|}{r^{\alpha}}.
\end{eqnarray}
The upper and lower Minkowski dimensions of $\partial\Omega$ with respect to 
$\Omega$ are then given by 
\begin{eqnarray}\label{MiN-3}
&& \hskip -0.50in
{\rm dim}^{\ast}_{Minkowski}(\partial\Omega)
:=\inf\{\gamma\geq 0:\, M^{\ast}_\gamma(\partial\Omega)<+\infty\}
=\sup\{\gamma\geq 0:\, M^{\ast}_\gamma(\partial\Omega)=+\infty\}
\nonumber\\[4pt]
&& \hskip 0.30in
=\inf\{\gamma\geq 0:\, M^{\ast}_\gamma(\partial\Omega)=0\}
=\sup\{\gamma\geq 0:\, M^{\ast}_\gamma(\partial\Omega)>0\},
\\[4pt]
&& \hskip -0.50in
{\rm dim}_{\ast,Minkowski}(\partial\Omega)
:=\sup\{\gamma\geq 0:\, M_{\ast,\gamma}(\partial\Omega)>0\}
=\inf\{\gamma\geq 0:\, M_{\ast,\gamma}(\partial\Omega)=0\}
\nonumber\\[4pt]
&& \hskip 0.30in
=\inf\{\gamma\geq 0:\, M_{\ast,\gamma}(\partial\Omega)<+\infty\}
=\sup\{\gamma\geq 0:\, M_{\ast,\gamma}(\partial\Omega)=+\infty\},
\label{MiN-4}
\end{eqnarray}
convening that $\inf\emptyset:=+\infty$ and $\sup\emptyset:=-\infty$.
When ${\rm dim}^{\ast}_{Minkowski}(\partial\Omega)=
{\rm dim}_{\ast,Minkowski}(\partial\Omega)$, the common value is referred to 
as the Minkowski dimension of $\partial\Omega$ with respect to $\Omega$, 
and is denoted by ${\rm dim}_{Minkowski}(\partial\Omega)$.
\end{definition}

Next, we recall the Coarea Formula (see, e.g., \cite{Fe}, 
\cite[Theorem~2, p.\,117]{EG}).
Given a fixed number $n\in{\mathbb{N}}$, denote by ${\mathcal{L}}^n$ 
the $n$-dimensional Lebesgue measure in ${\mathbb{R}}^n$ (occasionally we 
shall use the notation ${\mathcal{L}}^n(E)$ in place of $|E|$) and, 
for each $k\in{\mathbb{N}}$, $k\leq n$, let ${\mathcal{H}}^{k}$ stand 
for the $k$-dimensional Hausdorff measure in ${\mathbb{R}}^n$. 

\begin{proposition}\label{co-area}
Assume that $n\geq m$ and that $f:{\mathbb{R}}^n\to{\mathbb{R}}^m$ 
is a given Lipschitz function. Then, for any $A\subseteq{\mathbb{R}}^n$ 
which is ${\mathcal{L}}^n$-measurable and $g\in L^1(A)$, 
\begin{eqnarray}\label{ufg}
g\Bigl|_{A\cap f^{-1}(\{y\})}\,\,\mbox{ is ${\mathcal{H}}^{n-m}$-summable 
for ${\mathcal{L}}^m$-a.e. $y\in{\mathbb{R}}^m$}
\end{eqnarray}
and it holds that
\begin{eqnarray}\label{ufb}
\int_A g(x)|(Jf)(x)|\,d{\mathcal{L}}^{n}(x)=\int_{{\mathbb{R}}^m}
\Bigl(\int_{A\cap f^{-1}(\{y\})}g\,d{\mathcal{H}}^{n-m}\Bigr)
\,d{\mathcal{L}}^{m}(y),
\end{eqnarray}
\noindent where $Jf=\sqrt{{\rm det}\,[(Df)(Df)^{\top}]}$ 
is the Jacobian of $f$. 
\end{proposition}
 
Recall the definition of $\omega_{\alpha}$ from \eqref{MiN-2}. 
\begin{lemma}\label{bLL.A}
Let $\Omega\subseteq{\mathbb{R}}^n$ be a bounded open set which is 
Jordan measurable (i.e., a bounded open set whose boundary has Lebesgue
measure zero) and set $\Omega_{t}:=\{x\in\Omega:\,\delta_{\Omega}(x)\geq t\}$. Suppose that $\alpha>0$ is such that 
$\omega_{\alpha}$ vanishes at the origin and satisfies a Dini 
integrability condition, i.e., 
\begin{eqnarray}\label{bS-X}
\lim_{r\to 0^{+}}\omega_{\alpha}(r)=0\quad\mbox{ and }\quad 
\int_0\frac{\omega_{\alpha}(r)}{r}\,dr<+\infty.  
\end{eqnarray}
Then for every $t>0$ one has 
\begin{eqnarray}\label{bS-XX}
\int_{\Omega\setminus\Omega_t}\delta_{\Omega}(x)^{-\alpha}\,dx
=\omega_{\alpha}(t)+\alpha\int_0^t\frac{\omega_{\alpha}(r)}{r}\,dr. 
\end{eqnarray}
In particular, for every $t>0$ there holds
\begin{eqnarray}\label{bS-XXX}
\int_{\Omega}\delta_{\Omega}(x)^{-\alpha}\,dx
\leq t^{-\alpha}{\mathcal{L}}^n(\Omega)
+\alpha\int_0^t\frac{\omega_{\alpha}(r)}{r}\,dr<+\infty. 
\end{eqnarray}
\end{lemma}

\begin{proof}
Given that we are assuming that the bounded open set 
$\Omega\subseteq{\mathbb{R}}^n$ is Jordan measurable it follows that 
\begin{eqnarray}\label{bS-1}
{\mathcal{L}}^n(\partial\Omega)=0 
\end{eqnarray}
For each $t>0$, apply the coarea formula \eqref{ufb} with $A:=\Omega\setminus\Omega_t$, 
$g\in L^1(\Omega\setminus\Omega_t)$ arbitrary, and 
$f:{\mathbb{R}}^n\to{\mathbb{R}}$ given by $f(x):=\delta_{\Omega}(x)$ 
for each $x\in{\mathbb{R}}^n$. 
Then $\partial\Omega_t=\Omega\cap\delta_{\Omega}^{-1}(\{t\})$ for every 
$t>0$ and 
\begin{eqnarray}\label{bS-2}
Jf(x)=\left\{
\begin{array}{l}
1\,\,\mbox{ for ${\mathcal{L}}^n$-a.e. $x\in\Omega$},
\\[4pt]
0\,\,\mbox{ for ${\mathcal{L}}^n$-a.e. $x\in{\mathbb{R}}^n\setminus
\overline{\Omega}$},
\end{array}
\right.
\end{eqnarray}
hence for every $t>0$ we have (making use of \eqref{bS-1})
\begin{eqnarray}\label{bS-3}
\int_{\Omega\setminus\Omega_t}g(x)\,dx
=\int_0^t\Bigl(\int_{\partial\Omega_r}g\,d{\mathcal{H}}^{n-1}\Bigr)\,dr.
\end{eqnarray}
In particular, 
\begin{eqnarray}\label{bS-4}
\frac{d}{dt}\Bigl(\int_{\Omega\setminus\Omega_t}g(x)\,dx
\Bigr)=\int_{\partial\Omega_t}g\,d{\mathcal{H}}^{n-1}\quad
\mbox{ for ${\mathcal{H}}^1$-a.e. $t>0$}, 
\end{eqnarray}
which, in the case when $g=1$ yields 
\begin{eqnarray}\label{bS-5}
{\mathcal{H}}^{n-1}(\partial\Omega_t)=
\frac{d}{dt}\Bigl({\mathcal{L}}^{n}(\Omega\setminus\Omega_t)\Bigr)
\quad\mbox{ for ${\mathcal{H}}^1$-a.e. $t>0$}.
\end{eqnarray}
Specialise now \eqref{bS-3} to the case when, for some fixed 
$\alpha>0$ and $M>0$, we take 
\begin{eqnarray}\label{bS-6}
g(x):=\min\Bigl\{\delta_{\Omega}(x)^{-\alpha}\,,\,M\Bigr\},\qquad 
\forall\,x\in\Omega\setminus\Omega_t.
\end{eqnarray}
Then $g\in L^1(\Omega\setminus\Omega_t)$ so this choice yields 
\begin{eqnarray}\label{bS-7}
\int_{\Omega\setminus\Omega_t}
\min\Bigl\{\delta_{\Omega}(x)^{-\alpha}\,,\,M\Bigr\}\,dx
=\int_0^t\min\Bigl\{r^{-\alpha}\,,\,M\Bigr\}
{\mathcal{H}}^{n-1}(\partial\Omega_r)\,dr, 
\end{eqnarray}
hence, ultimately, 
\begin{eqnarray}\label{bS-8}
\int_{\Omega\setminus\Omega_t}\delta_{\Omega}(x)^{-\alpha}\,dx
=\int_0^t r^{-\alpha}{\mathcal{H}}^{n-1}(\partial\Omega_r)\,dr, 
\end{eqnarray}
after letting $M\nearrow +\infty$ and invoking Lebesgue's Monotone Convergence
Theorem. Thus, from \eqref{bS-5} and \eqref{bS-8} we obtain 
\begin{eqnarray}\label{bS-9}
\int_{\Omega\setminus\Omega_t}\delta_{\Omega}(x)^{-\alpha}\,dx
=\int_0^t r^{-\alpha}
\frac{d}{dr}\Bigl({\mathcal{L}}^{n}(\Omega\setminus\Omega_r)\Bigr)\,dr.
\end{eqnarray}
Integrating by parts in the right-hand side of \eqref{bS-9} then gives
\begin{eqnarray}\label{bS-10}
\int_{\Omega\setminus\Omega_t}\delta_{\Omega}(x)^{-\alpha}\,dx
&=& t^{-\alpha}\,{\mathcal{L}}^{n}(\Omega\setminus\Omega_t)
-\lim_{r\to 0^{+}}\Bigl(
r^{-\alpha}\,{\mathcal{L}}^{n}(\Omega\setminus\Omega_r)\Bigr)
\nonumber\\[4pt]
&& +\alpha\int_0^t 
r^{-\alpha-1}{\mathcal{L}}^{n}(\Omega\setminus\Omega_r)\,dr, 
\end{eqnarray}
so that 
\begin{eqnarray}\label{bS-11}
\int_{\Omega\setminus\Omega_t}\delta_{\Omega}(x)^{-\alpha}\,dx
=\omega_{\alpha}(t)-\lim_{r\to 0^{+}}\omega_{\alpha}(r)
+\alpha\int_0^t\frac{\omega_{\alpha}(r)}{r}\,dr. 
\end{eqnarray}
Now \eqref{bS-XX} readily follows from this, granted \eqref{bS-X}.
Finally, \eqref{bS-XXX} is an immediate consequence of \eqref{bS-XX}, 
the crude estimate $\int_{\Omega_t}\delta_{\Omega}(x)^{-\alpha}\,dx\leq
t^{-\alpha}{\mathcal{L}}^{n}(\Omega_t)$, and the fact that 
$t^{-\alpha}{\mathcal{L}}^{n}(\Omega_t)+\omega_{\alpha}(t)
=t^{-\alpha}{\mathcal{L}}^n(\Omega)$. 
\end{proof}

\begin{remark}\label{RR.AA}
Assume that $\Omega\subseteq{\mathbb{R}}^n$ is a Jordan measurable, 
bounded open set. Then an inspection of the proof of Lemma~\ref{bLL.A} reveals
that 
\begin{eqnarray}\label{bcAAS}
\int_{\Omega}\delta_{\Omega}(x)^{-\alpha}\,dx<+\infty\qquad
\mbox{whenever }\,\,{\rm dim}^{\ast}_{Minkowski}(\partial\Omega)<n-\alpha, 
\end{eqnarray}
and 
\begin{eqnarray}\label{bcAAS.2}
\int_{\Omega}\delta_{\Omega}(x)^{-\alpha}\,dx=+\infty\quad
\mbox{if }\,
\sup\{\gamma\geq 0:\, M^{\ast}_\gamma(\partial\Omega)<+\infty,\,\,
M_{\ast,\gamma}(\partial\Omega)>0\}>n-\alpha.
\end{eqnarray}
\end{remark}

\begin{definition}\label{fdr9n}
The set $\Sigma\subseteq{\mathbb{R}}^n$ is said to be 
Ahlfors regular if there exist finite constants 
$C_0,C_1>0$ as well as a number $R>0$ such that
\begin{eqnarray}\label{gvry}
C_0\,r^{n-1}\leq 
{\mathcal{H}}^{n-1}(B(x,r)\cap\Sigma)\leq C_1\,r^{n-1},\qquad
\forall x\in\Sigma,\,\,\,\forall\,r\in(0,R), 
\end{eqnarray}
The triplet $C_0,C_1,R$ makes up what will henceforth be referred to 
as the Ahlfors character of $\Sigma$.
\end{definition}

\begin{lemma}\label{bvv55}
Let $\Omega\subseteq{\mathbb{R}}^n$ be a bounded open set whose 
boundary is Ahlfors regular. Then, for any exponent $\alpha\in[0,1)$ 
there exists a finite constant $C>0$, which depends only $n$, $\alpha$ and 
the Ahlfors character of $\partial\Omega$, such that 
\begin{eqnarray}\label{bc-W1}
\int_{\Omega}\delta_{\Omega}(x)^{-\alpha}\,dx
\leq C\,\bigl[{\mathcal{L}}^{n}(\Omega)\bigr]^{1-\alpha}
\bigl[{\mathcal{H}}^{n-1}(\partial\Omega)\bigr]^{\alpha}.
\end{eqnarray}
This implies that the following generalised isoperimetric inequality holds: 
\begin{eqnarray}\label{bc-Wss}
\int_{\Omega}\delta_{\Omega}(x)^{-\alpha}\,dx\leq C\,
\bigl[{\mathcal{H}}^{n-1}(\partial\Omega)\bigr]^{\frac{n-\alpha}{n-1}}.
\end{eqnarray}
In particular, under the same hypotheses,  
\begin{eqnarray}\label{bca84r}
\int_{\Omega}\delta_{\Omega}(x)^{-\alpha}\,dx
\leq C\,{\rm diam}\,(\Omega)^{n-\alpha}.
\end{eqnarray}
\end{lemma}

\begin{proof}
The version of the isoperimetric inequality proved by H.\,Federer 
(cf.~3.2.43-3.2.44 on p.\,278 of \cite{Fe}) reads
\begin{eqnarray}\label{Fed-H1}
E\subset{\mathbb{R}}^n\,\,\mbox{ with }\,\,{\mathcal{L}}^n(\overline{E})
<+\infty\,\Longrightarrow\,{\mathcal{L}}^n(\overline{E})
\leq\frac{1}{n(\omega_{n-1})^{1/(n-1)}}
\bigl[{\mathcal{H}}^{n-1}(\partial E)\bigr]^{\frac{n}{n-1}},
\end{eqnarray}
where $\omega_{n-1}$ denotes the surface area of $S^{n-1}$. 
Of course, \eqref{Fed-H1} covers the case $\alpha=0$ of \eqref{bc-W1}, 
so we will assume in what follows that $0<\alpha<1$. 

To proceed, we note two consequences of the assumption that that 
$\Omega\subseteq{\mathbb{R}}^n$ is a bounded open set whose boundary is 
Ahlfors regular. First, it is clear that \eqref{bS-1} holds and, 
hence, $\Omega$ is Jordan measurable. Second, it has been proved in 
\cite{HMT} that 
\begin{eqnarray}\label{baD-ii}
{\mathcal{L}}^n(\Omega\setminus\Omega_r)
\leq Cr\,{\mathcal{H}}^{n-1}(\partial\Omega),\qquad\forall r>0,
\end{eqnarray}
where $C>0$ depends only on the Ahlfors character of $\partial\Omega$ 
and, as before, for each $r>0$ we have set 
$\Omega_{r}=\{x\in\Omega:\,\delta_{\Omega}(x)\geq r\}$. 
(Parenthetically, we wish to point out that this estimate implies that 
${\rm dim}^{\ast}_{Minkowski}(\partial\Omega)\leq n-1$ whenever 
$\Omega\subseteq{\mathbb{R}}^n$ is a bounded open set whose 
boundary is Ahlfors regular.) In particular, \eqref{baD-ii} entails 
\begin{eqnarray}\label{baD-ii2}
\omega_{\alpha}(r)\leq Cr^{1-\alpha}
\,{\mathcal{H}}^{n-1}(\partial\Omega),\qquad\forall r>0,
\end{eqnarray}
and, given that $\alpha\in(0,1)$, it follows that the conditions in 
\eqref{bS-X} are satisfied. On the basis of this discussion, \eqref{bS-XXX}
then gives
\begin{eqnarray}\label{bS-XXX.i}
\int_{\Omega}\delta_{\Omega}(x)^{-\alpha}\,dx
&\leq & t^{-\alpha}{\mathcal{L}}^n(\Omega)
+C\,\alpha\Bigl(\int_0^t r^{-\alpha}\,dr\Bigr)
\,{\mathcal{H}}^{n-1}(\partial\Omega)
\nonumber\\[4pt]
&=& t^{-\alpha}{\mathcal{L}}^n(\Omega)
+\frac{C\,\alpha}{1-\alpha}t^{1-\alpha}
\,{\mathcal{H}}^{n-1}(\partial\Omega),
\end{eqnarray}
for every $t>0$. Choosing $t:={\mathcal{L}}^n(\Omega)\bigl/
{\mathcal{H}}^{n-1}(\partial\Omega)$ then readily yields \eqref{bc-W1}. 
Having justified \eqref{bc-W1}, then \eqref{bc-Wss} follows 
from this after observing that \eqref{Fed-H1} implies 
\begin{eqnarray}\label{TFC.a3}
\bigl[{\mathcal{L}}^n(\Omega)\bigr]^{1-\alpha}\leq C_{n,\alpha}
\bigl[{\mathcal{H}}^{n-1}(\partial\Omega)\bigr]^{\frac{n(1-\alpha)}{n-1}}.
\end{eqnarray}

As regards \eqref{bca84r}, this is going to be a consequence of 
\eqref{bc-Wss} and the fact that for any set $E\subseteq{\mathbb{R}}^n$ 
whose boundary is Ahlfors regular there holds 
\begin{eqnarray}\label{AVatar}
{\mathcal{H}}^{n-1}(\partial E)\leq C\,\bigl[{\rm diam}\,(E)\bigr]^{n-1}, 
\end{eqnarray}
where $C>0$ depends only on the Ahlfors character of $\partial E$
(in fact only the upper estimate in the Ahlfors regularity condition 
is really needed for this purpose). At this stage there remains to 
prove \eqref{AVatar} and, given the dilation and translation invariant 
nature of this estimate, there is no loss of generality in assuming that 
${\rm diam}\,(E)=1$ and that, in fact, $\overline{E}\subseteq (-1,1)^n$. 
Partition the cube $(-1,1)^n$ into a grid of congruent subcubes, call them 
$\{Q\}_{Q\in{\mathcal{J}}}$, of side-length $R/(2\sqrt{n})$, where $R\in(0,1)$ 
is such that there exists $C>0$ for which 
\begin{eqnarray}\label{gvDA}
{\mathcal{H}}^{n-1}(B(x,r)\cap\partial E)\leq C\,r^{n-1},\qquad
\forall x\in\partial E,\,\,\,\forall\,r\in(0,R). 
\end{eqnarray}
Consider ${\mathcal{J}}_{\ast}:=\{Q\in{\mathcal{J}}:\,Q\cap\partial E
\not=\emptyset\}$ and, for each $Q\in {\mathcal{J}}_{\ast}$, select 
$x_Q\in Q\cap\partial E$. Then, clearly, 
\begin{eqnarray}\label{gvDB}
\partial E\subseteq\bigcup_{Q\in{\mathcal{J}}_{\ast}}B(x_Q,R/2) 
\end{eqnarray}
which, when used in conjunction with \eqref{gvDA} and the fact that 
${\mathcal{H}}^{n-1}$ is an outer measure, gives 
\begin{eqnarray}\label{gvDC}
{\mathcal{H}}^{n-1}(\partial E) &\leq & 
\sum_{Q\in{\mathcal{J}}_{\ast}}{\mathcal{H}}^{n-1}(\partial E\cap B(x_Q,R/2))
\nonumber\\[4pt]
&\leq & C(R/2)^{n-1}\cdot\#\,{\mathcal{J}}_{\ast}
\leq 2\,C\,n^{n/2}R^{-1}. 
\end{eqnarray}
This, of course, suits our purposes, so the proof of \eqref{AVatar} is 
complete. 
\end{proof}

Here is the proposition alluded to a while ago. 

\begin{proposition}\label{PKc-1.F} 
(i) If $\Omega\subseteq{\mathbb{R}}^n$ is a bounded domain whose boundary has 
a finite upper $\gamma$-dimensional Minkowski content, where $\gamma<n$, 
then \eqref{PKc-4} holds for any $\beta\in(0,(n-\gamma)/2)$.

(ii) If $\Omega\subseteq{\mathbb{R}}^n$ is a bounded domain whose boundary 
is Ahlfors regular, then \eqref{PKc-4} holds for any $\beta\in(0,1/2)$ 
and moreover
\begin{eqnarray}\label{either}
\int_{\Omega}u(x)^{-\beta}\,dx
\leq C\,\bigl[{\mathcal{L}}^{n}(\Omega)\bigr]^{1-2\beta}
\bigl[{\mathcal{H}}^{n-1}(\partial\Omega)\bigr]^{2\beta}
\leq C\,{\rm diam}\,(\Omega)^{n-2\beta}
\quad\mbox{if }\,\,\beta<1/2,
\end{eqnarray}
where $C$ depends only on the Ahlfors character of $\partial\Omega$, 
$n$ and $\beta$.

(iii) As far as \eqref{either} is concerned, the critical value $\beta=1/2$ is 
in the nature of best possible in the sense that for every $\beta\in(1/2,1)$ 
there exists a bounded domain $\Omega_{\beta}\subseteq{\mathbb{R}}^2$ which 
is regular for the Dirichlet problem and has an Ahlfors regular
boundary and with the property that if $u$ solves \eqref{PKc-1} then 
\begin{eqnarray}\label{KbvR}
\int_{\Omega_{\beta}}u(x)^{-\beta}\,dx=+\infty.
\end{eqnarray}
\end{proposition}
\begin{proof}
The claims in (i)-(ii) follow from Lemma~\ref{bLL.A} and Lemma~\ref{bvv55},
respectively, with the help of \eqref{TF-2.D} (satisfied by the 
solution $u$ of problem \eqref{PKc-1}).
Concerning (iii), the task is to construct a counterexample to the statement 
\eqref{PKc-4} in the case when $\beta\in(1/2,1)$ in the class of bounded 
domains which are regular for the Dirichlet problem and have Ahlfors regular 
boundaries. To this end, fix $\beta\in(1/2,1)$ and consider the curvilinear 
triangle $\Omega_{\beta}$ in ${\mathbb{R}}^2$ given by 
\begin{eqnarray}\label{JHca-1}
\Omega_{\beta}:=\bigl\{(x,y)\in{\mathbb{R}}^2:\,0<x<1\,\,\mbox{ and }\,\,
0<y<\varepsilon\,x^{1/(2\beta-1)}\bigr\},  
\end{eqnarray}
where $\varepsilon=\varepsilon(\beta)$ is a sufficiently small 
positive constant, to be specified momentarily. Clearly, the function
\begin{eqnarray}\label{JHca-2}
v(x,y):=y\bigl(\varepsilon x^{1/(2\beta-1)}-y\bigr),\qquad 
\forall\,(x,y)\in\Omega_{\beta},
\end{eqnarray}
is positive in the domain $\Omega_{\beta}$ and is nonnegative 
on its boundary. In addition, for all $(x,y)\in\Omega_{\beta}$ we have
\begin{eqnarray}\label{JHca-3}
-(\Delta v)(x,y) &=& 2-\frac{2\varepsilon(1-\beta)}{(2\beta-1)^2}\,
x^{(3-4\beta)/(2\beta-1)}\,y
\nonumber\\[4pt]
&\geq & 2-\frac{2\varepsilon^2(1-\beta)}{(2\beta-1)^2}\,
x^{(4-4\beta)/(2\beta-1)}
\nonumber\\[4pt]
&\geq & 2-2\varepsilon^2(1-\beta)(2\beta-1)^{-2},
\end{eqnarray}
where the last step makes essential use of the fact that 
$\beta\in(1/2,1)$. At this stage, pick $\varepsilon>0$ sufficiently 
small so that the last expression in \eqref{JHca-3} is $\geq 1$. 
Such a choice forces $u-v$ to be subharmonic in $\Omega_{\beta}$, 
if $u$ is the solution of the Saint Venant problem in $\Omega_{\beta}$ 
(cf.~\eqref{PKc-1}). In addition, $u-v\leq 0$ on $\partial\Omega_{\beta}$ 
by design. The Maximum Principle then gives that $u\leq v$ 
in $\Omega_{\beta}$. Consequently, we may estimate
\begin{eqnarray}\label{JHca-4}
\int_{\Omega_{\beta}}u(x,y)^{-\beta}\,dxdy & \geq & 
\int_{\Omega_{\beta}}v(x,y)^{-\beta}\,dxdy
\nonumber\\[4pt]
&=& \int^1_0\Bigl(\int_0^{\varepsilon x^{1/(2\beta-1)}}y^{-\beta}
\bigl(\varepsilon x^{1/(2\beta-1)}- y\bigr)^{-\beta}\,dy\Bigr)dx
\nonumber\\[4pt]
&=& \varepsilon^{1-2\beta}\Bigl(\int_0^1 x^{-1}\,dx\Bigr) 
\Bigl(\int_0^1 t^{-\beta}(1-t)^{-\beta}\,dt\Bigr)=+\infty,
\end{eqnarray}
after making the change of variables $y=\varepsilon x^{1/(2\beta-1)}\,t$
in the inner integral in the second line. This completes the proof 
of the proposition.
\end{proof}

\section{Barrier functions and domains satisfying a cone condition}
\label{sect:3}
\setcounter{equation}{0}

Here and elsewhere $S^{n-1}$ stands for the unit sphere in ${\mathbb{R}}^n$.
We denote by $\Gamma^R_\theta(x_0,\eta)$ the open, one-component
circular cone in ${\mathbb{R}}^n$ with vertex at $x_0\in{\mathbb{R}}^n$, 
half-aperture $\theta\in(0,\pi)$, axis along $\eta\in S^{n-1}$, 
and (roundly) truncated at $R>0$, i.e., 
\begin{eqnarray}\label{Cg-77}
\Gamma^R_\theta(x_0,\eta)
:=\bigl\{x\in{\mathbb{R}}^n:\,(x-x_0)\cdot\eta>|x-x_0|\,\cos\theta
\,\,\mbox{ and }\,\,|x-x_0|<R\bigr\},
\end{eqnarray}
When $R=+\infty$ (that is, the cone is infinite) we agree to simply write
$\Gamma_\theta(x_0,\eta)$. Furthermore, we use the abbreviation 
$\Gamma^R_\theta$ (respectively, $\Gamma_\theta$, when $R=+\infty$)
whenever $x_0=0\in{\mathbb{R}}^n$ and $\eta=e_n:=(0,...,0,1)\in S^{n-1}$. 

Of course, $\Gamma_{\theta}\cap S^{n-1}$ is the spherical cap with 
centre at the north pole and (spherical) radius $\theta$. More generally, 
given an open, connected subset ${\mathfrak{G}}$ of $S^{n-1}$, we denote by 
$\Gamma_{\mathfrak{G}}$ the open cone in ${\mathbb{R}}^n$ with vertex at the 
origin and shape ${\mathfrak{G}}$, i.e., 
\begin{eqnarray}\label{HBV-rt}
\Gamma_{\mathfrak{G}}:=\{\rho\,\omega:\,\rho>0\,\,
\mbox{ and }\,\,\omega\in {\mathfrak{G}}\}.
\end{eqnarray}

Going further, we let $\Delta_{S^{n-1}}$ stand for the Laplace-Beltrami 
operator on $S^{n-1}$ and fix an open, connected subset ${\mathfrak{G}}$ 
of $S^{n-1}$ with the property that $\partial_{S^{n-1}}{\mathfrak{G}}$, 
the boundary of ${\mathfrak{G}}$ relative to $S^{n-1}$, is sufficiently 
regular. In this setting, we let $\Lambda_{\mathfrak{G}}>0$ be the first 
positive eigenvalue of the nonnegative operator $-\Delta_{S^{n-1}}$ 
equipped with (homogeneous) Dirichlet boundary condition on ${\mathfrak{G}}$ 
and denote by $\phi_{{\mathfrak{G}}}$ an eigenfunction corresponding to 
the eigenvalue $\Lambda_{\mathfrak{G}}$. Hence, 
\begin{eqnarray}\label{vTR1.G}
-\Delta_{S^{n-1}}\phi_{{\mathfrak{G}}}
=\Lambda_{{\mathfrak{G}}}\,\phi_{{\mathfrak{G}}}
\,\,\mbox{ in }\,\,\,{\mathfrak{G}},\,\,\,\mbox{ and }\,\,
\phi_{{\mathfrak{G}}}=0\,\,\mbox{ on }\,\,\partial_{S^{n-1}}{\mathfrak{G}}.
\end{eqnarray}
Recall that any eigenfunction corresponding to $\Lambda_{\mathfrak{G}}$ 
does not change sign in ${\mathfrak{G}}$ (see, e.g., the discussion on 
p.\,42-43 in \cite{Ch} in the case of a spherical cap),
Since $\phi_{\mathfrak{G}}$ is uniquely determined only up to a 
re-normalisation, it follows that there is no loss of generality 
in assuming that 
\begin{eqnarray}\label{Mc-T2A.G}
\phi_{{\mathfrak{G}}}>0\,\,\mbox{ in }\,\,{\mathfrak{G}},
\quad\mbox{and}\,\,\sup_{{\mathfrak{G}}}\phi_{{\mathfrak{G}}}=1.
\end{eqnarray}
For further reference let us also record here that, granted sufficient 
regularity for $\partial_{S^{n-1}}{\mathfrak{G}}$, the function 
$\phi_{{\mathfrak{G}}}$ behaves essentially like the distance to the 
boundary of ${\mathfrak{G}}$. More precisely, if 
$\partial_{S^{n-1}}{\mathfrak{G}}$ is of class 
$C^{1,\alpha}$, for some $\alpha\in(0,1)$, then the following 
estimate (which is going to be useful in \S\,\ref{sect:3} and \S\,\ref{sect:5}) 
holds
\begin{eqnarray}\label{GFsA}
\phi_{{\mathfrak{G}}}(\omega)\approx
{\rm dist}_{S^{n-1}}\bigl(\omega,\partial_{S^{n-1}}{\mathfrak{G}}\bigr),\quad
\mbox{ uniformly for }\,\,\omega\in {\mathfrak{G}},
\end{eqnarray}
where ${\rm dist}_{S^{n-1}}(\omega,\omega'):={\rm arccos}\,(\omega\cdot\omega')$,
for $\omega,\omega'\in S^{n-1}$, denotes the geodesic distance on
$S^{n-1}$. This property is proved later (in the Appendix), as to avoid disrupting 
the flow of the presentation\footnote{We are grateful to a referee for 
questioning an inaccurate claim we made in an earlier version.}.

Corresponding to the case when ${\mathfrak{G}}$ is a 
spherical cap, say ${\mathfrak{G}}=S^{n-1}\cap\Gamma_{\theta}$ 
for some $\theta\in(0,\pi)$, we agree to write $\phi_\theta$ 
and $\Lambda_{\theta}$ in place of $\phi_{S^{n-1}\cap\Gamma_{\theta}}$
and $\Lambda_{S^{n-1}\cap\Gamma_{\theta}}$, respectively. Hence, in 
particular, 
\begin{eqnarray}\label{vTR1}
\begin{array}{l}
-\Delta_{S^{n-1}}\phi_{\theta}=\Lambda_{\theta}\,\phi_{\theta}
\quad\mbox{ in }\,\,\,S^{n-1}\cap\Gamma_\theta
\,\,\mbox{ and }\,\,\phi_{\theta}=0
\,\,\mbox{ on }\,\,S^{n-1}\cap\partial\Gamma_\theta.
\\[4pt]
\phi_{\theta}>0\,\,\mbox{ in }\,\,S^{n-1}\cap\Gamma_\theta
\quad\mbox{and}\quad\sup_{S^{n-1}\cap\Gamma_\theta}\phi_{\theta}=1
\quad\mbox{for each }\,\,\theta\in(0,\pi).
\end{array}
\end{eqnarray}

\begin{definition}\label{defalpha} 
Given an open, connected subset ${\mathfrak{G}}$ of $S^{n-1}$, 
with a sufficiently regular boundary (relative to $S^{n-1}$), 
we associate the index $\alpha_{\mathfrak{G}}$ defined by 
\begin{equation}\label{formula.G}
\alpha_{\mathfrak{G}}:=-{\textstyle{\frac{n-2}{2}}}
+\sqrt{\textstyle{\frac{(n-2)^2}{4}}+\Lambda_{\mathfrak{G}}}, 
\end{equation}
\noindent that is, the unique positive root of the equation 
\begin{eqnarray}\label{critic.G}
\alpha_{\mathfrak{G}}(\alpha_{\mathfrak{G}}+n-2)=\Lambda_{\mathfrak{G}}.
\end{eqnarray}
(Note that since $\Lambda_{\mathfrak{G}}>0$, these considerations 
are meaningful.) Finally, for each $\theta\in(0,\pi)$, we abbreviate 
$\alpha_{S^{n-1}\cap\Gamma_{\theta}}$ by $\alpha_\theta$. 
Hence, in this notation, 
\begin{eqnarray}\label{formula}
\alpha_\theta=-{\textstyle{\frac{n-2}{2}}}
+\sqrt{\textstyle{\frac{(n-2)^2}{4}}+\Lambda_\theta}\quad\mbox{and}\quad
\Lambda_\theta=\alpha_\theta(\alpha_\theta+n-2).
\end{eqnarray}
for any $\theta\in(0,\pi)$. 
\end{definition}
\noindent The index $\alpha_{\mathfrak{G}}$ has been studied by many
authors; see in particular \cite{FH}. The format of \eqref{critic.G} is suggested by the formula 
for the Euclidean Laplacian in spherical polar coordinates
$x=\rho\,\omega\in{\mathbb{R}}^n\setminus\{0\}$, 
with $\rho:=|x|>0$ and $\omega:=x/|x|\in S^{n-1}$, i.e., 
\begin{eqnarray}\label{La-LB}
\Delta f=\rho^{1-n}\partial_\rho\bigl(\rho^{n-1}\partial_\rho f\bigr) 
+\rho^{-2}\Delta_{S^{n-1}}f=\partial^2_\rho f
+(n-1)\rho^{-1}\partial_\rho f+\rho^{-2}\Delta_{S^{n-1}}f.
\end{eqnarray}
Indeed, introducing the barrier function 
$v_{{\mathfrak{G}}}:\Gamma_{\mathfrak{G}}\to{\mathbb{R}}$ by setting 
\begin{eqnarray}\label{vTR}
v_{{\mathfrak{G}}}(x):=\rho^{\alpha_{\mathfrak{G}}}\phi(\omega)
=|x|^{\alpha_{\mathfrak{G}}}\phi_{{\mathfrak{G}}}
\left({\textstyle{\frac{x}{|x|}}}\right)
\quad\mbox{for }\,\,\omega=\frac{x}{|x|}\in {\mathfrak{G}}\subseteq S^{n-1}
\,\,\mbox{ and }\,\rho=|x|>0, 
\end{eqnarray}
it follows that, for each $\omega\in {\mathfrak{G}}$ and $\rho>0$, 
\begin{eqnarray}\label{vTR.GG}
(\Delta v_{{\mathfrak{G}}})(\rho\,\omega)
=[\alpha_{\mathfrak{G}}(\alpha_{\mathfrak{G}}+n-2)
-\Lambda_{\mathfrak{G}}]\rho^{\alpha_{\mathfrak{G}}-2}\phi_{{\mathfrak{G}}}(\omega)=0
\end{eqnarray}
precisely for the choice \eqref{critic.G}. This ensures that the function 
$v_{{\mathfrak{G}}}$ is harmonic in the cone $\Gamma_{{\mathfrak{G}}}$. 
In summary, taking $\alpha_{{\mathfrak{G}}}$ as in \eqref{formula.G} 
ensures that 
\begin{eqnarray}\label{JGD}
\Delta v_{{\mathfrak{G}}}=0\,\mbox{ in }\,\Gamma_{{\mathfrak{G}}},\quad
v_{{\mathfrak{G}}}=0\,\mbox{ on }\,\partial\Gamma_{{\mathfrak{G}}},\quad
v_{\theta}>0\,\mbox{ in }\,\Gamma_{{\mathfrak{G}}}. 
\end{eqnarray}

In the axially symmetric case, i.e., when 
${\mathfrak{G}}=S^{n-1}\cap\Gamma_{\theta}$ for some $\theta\in(0,\pi)$, 
a good deal is known about the properties enjoyed by the exponent 
$\alpha_\theta$ introduced in \eqref{formula} (see, e.g., Theorem~3 on p.\,44, 
Theorem~6 on p.\,50 in \cite{Ch} and the discussion on p.\,112 of \cite{Ai1}). 
Specifically, for each $n\geq 2$ one has 
\begin{eqnarray}\label{alphamap}
&& (0,\pi)\ni\theta\mapsto\alpha_\theta\in(0,+\infty)
\quad\mbox{ is strictly decreasing and continuous},
\\[4pt]
&& \alpha_{\pi/2}=1\quad\mbox{ and }\quad
 \lim_{\theta\searrow 0}\alpha_\theta=+\infty,
\label{alphamap2}
\\[4pt]
&& \lim_{\theta\nearrow\pi}\alpha_\theta=0\quad\mbox{if }\,\,n\geq 3,
\label{lim}
\\[4pt]
&& \alpha_\theta=\frac{\pi}{2\theta}\quad\mbox{if }\,\,n=2,
\quad\mbox{ and }\,\,
\alpha_\theta=\frac{\pi}{\theta}-1\quad\mbox{if }\,\,n=4,
\label{alphan2}
\\[4pt]
&& \alpha_\theta\in({\textstyle\frac12},+\infty)\quad\mbox{and}\quad 
\lim_{\theta\nearrow\pi}\alpha_\theta=\textstyle\frac12
\quad\mbox{if }\,\,n=2,
\label{limn2}
\\[4pt]
&& \alpha_\theta=2\Longleftrightarrow\theta={\rm arccos}\,(1/\sqrt{n}).
\label{limn2F}
\end{eqnarray}
\noindent The computations in the case $n=2$ are particularly simple.  
Indeed, the eigenvalue problem for the Dirichlet-Laplacian
on the one-dimensional arc $\{e^{i\omega}:\,-\theta<\omega<\theta\}$ 
in the unit circle becomes (with `prime'
denoting the angular derivative $d/d\omega$) 
$\phi''(\omega)+\Lambda\phi(\omega)=0$ for $-\theta<\omega<\theta$,  
$\phi(-\theta)=\phi(\theta)=0$. The smallest positive eigenvalue is then
$\Lambda=\Lambda_\theta=\frac{\pi^2}{(2\theta)^2}$ which, in light
of \eqref{formula}, gives the first formula in \eqref{alphan2}.
In the higher dimensional setting, the eigenvalue problem on a 
spherical cap leads to a less transparent equation. 
To describe this, recall that the so-called Gegenbauer 
functions, ${\mathcal{C}}_\alpha^\nu(z)$, are the solutions of 
Gegenbauer's differential equation
\begin{equation}\label{diffeqg}
(z^2-1)\frac{d^2g}{dz^2}+(2\nu+1)z\frac{dg}{dz}
-\alpha(\alpha+2\nu)g=0,\qquad z,\nu,\alpha\in{\mathbb{C}}.
\end{equation}
\noindent When considered with the variable $z$ restricted to the 
interval $(-1,1)$ on the real axis, the above second-order ODE is 
endowed with the initial conditions 
\begin{eqnarray}\label{TGdf}
g(-1)=1\quad\mbox{and}\quad
\frac{dg}{dz}(-1)=-\frac{\alpha(\alpha+2\nu)}{2\nu+1}.
\end{eqnarray}
For more details on this subject see, e.g., \cite{MOS}.
In the present context, the key feature of the Gegenbauer 
functions is that the exponent $\alpha_\theta$ from \eqref{formula}
coincides with the first positive zero of the mapping 
$\alpha\mapsto{\mathcal{C}}^{\frac{n-2}{2}}_\alpha(-\cos\theta)$;
compare with Lemma~6.6.3 in \cite{KMR}. For example, 
the continuity of \eqref{alphamap} follows from this representation 
and classical results on the dependence of the solution of ODE's 
on parameters. For related material see also \cite{Miller}
(especially Theorem~2, p.\,308), and \cite{MS} (where, in lieu of 
\eqref{diffeqg}, the authors work with an ODE for 
$f_{n,\alpha}(\theta):={\mathcal{C}}^{\frac{n-2}{2}}_\alpha(-\cos\theta)$).

We continue by recording the definition of the 
class of nontangentially accessible domains (introduced by 
Jerison and Kenig in \cite{JK}), and by making a couple of remarks. 

\begin{definition}\label{Def-NTA}
A nonempty, proper open subset $\Omega$ of ${\mathbb{R}}^{n}$ is called 
an NTA domain provided $\Omega$ satisfies 
both an interior and an exterior corkscrew condition (with constants 
$M$, $r_\ast$ as in Definition~\ref{TFG-54}) and $\Omega$ satisfies 
a Harnack chain condition, defined as follows (with reference to 
$M$ as above). 

If $x_1,x_2\in\Omega$ and $k\in{\mathbb{N}}$ are such 
that $\delta_{\Omega}(x_i)\geq\varepsilon$ 
for $i=1,2$, and $|x_1-x_2|\leq 2^k\varepsilon$, for some $\varepsilon>0$, 
then there exist $Mk$ balls $B_j\subseteq\Omega$, $1\leq j\leq Mk$, such that 
\begin{enumerate}
\item[(i)] $x_1\in B_1$, $x_2\in B_{Mk}$ and 
$B_j\cap B_{j+1}\neq\emptyset$ for $1\leq j\leq Mk-1$; 
\item[(ii)] each ball $B_j$ has a radius $r_j$ satisfying 
\begin{eqnarray}\label{gabab}
M^{-1}r_j\leq {\rm dist}(B_j,\partial\Omega)\leq Mr_j\,\,\mbox{ and }\,\,
r_j\geq M^{-1}\,{\rm min}\bigl\{\delta_{\Omega}(x_1),
\delta_{\Omega}(x_2)\bigr\}.
\end{eqnarray}
\end{enumerate}
\end{definition}

Two comments are going to be of importance for us later on. 
First, the relevance of the Harnack chain condition is that, thanks to 
Harnack's inequality, if $w$ is a positive harmonic function in $\Omega$
then, in the context of the second part of Definition~\ref{Def-NTA}, 
\begin{eqnarray}\label{RfD.6}
 M^{-k}w(x_1)\leq w(x_2)\leq M^k w(x_1).
\end{eqnarray}
Second, any bounded NTA domain is regular for the Dirichlet problem 
(it suffices to recall that any such domain satisfies an exterior corkscrew 
condition). 

Moving on, a bounded domain $\Omega$ in ${\mathbb{R}}^n$
is called a Lipschitz domain provided $\Omega$ and its boundary 
$\partial\Omega$ locally coincide with, respectively, the upper-graph 
and the graph of a Lipschitz function. In this vein, recall that 
a function $f:D\to{\mathbb{R}}$ where, say, $D\subseteq{\mathbb{R}}^{n-1}$, 
is called Lipschitz provided there exists $M>0$ so that 
$|f(x)-f(y)|\leq M|x-y|$ for any $x,y\in D$. 
A formal definition is given below.

\begin{definition}\label{lipdom} 
A bounded domain $\Omega\subset{\mathbb{R}}^n$ is called Lipschitz if for any 
$x_0\in\partial\Omega$ there exist $r,\,h>0$
and a coordinate system $\{x_1,\dots,x_n\}$ in ${\mathbb{R}}^n$
(isometric to the canonical one) with origin at $x_0$ along with a 
function $\varphi:{\mathbb{R}}^{n-1}\to{\mathbb{R}}$ which is Lipschitz 
and for which the following property holds. If $C(r,h)$ denotes the open cylinder 
$\{x=(x',x_n):\,|x'|<r\mbox{ and }-h<x_n<h\}\subset{\mathbb{R}}^n$, then
\begin{equation}\label{cylinders}
\begin{array}{ll}\vspace{2mm}
\partial\Omega\cap C(r,h)=\{x=(x',x_n):\,\,|x'|<r
\,\,\,\mbox{and}\,\,\,x_n=\varphi(x_1,\dots,x_{n-1})\},
\\[4pt]
\Omega\cap C(r,h)=\{x=(x',x_n):\,\,|x'|<r\,\,\,\mbox{and}\,\,\,
\varphi(x_1,\dots,x_{n-1})<x_n<h\}.
\end{array}
\end{equation}
Fix an atlas for $\partial\Omega$, i.e. a finite collection of cylinders 
$\{C_k(r_k,h_k)\}_{1\leq k\leq N}$ (with associated Lipschitz maps 
$\{\varphi_k\}_{1\leq k\leq N}$) covering $\partial\Omega$. 
The Lipschitz constant of $\Omega$, denoted in what follows
by $\kappa_{\Omega}$, is defined as the infimum of 
${\rm max}\,\{\|\nabla\varphi_k\|_{L^\infty}:\,1\leq k\leq N\}$
taken over all possible atlases of $\partial\Omega$. 

Finally, domains of class $C^k$ for some $k\in{\mathbb{N}}\cup\{0\}$
(or $C^{k,\alpha}$ domains for $k\in{\mathbb{N}}\cup\{0\}$ and $\alpha\in(0,1]$, 
respectively) are defined analogously, by requiring that all functions $\varphi:{\mathbb{R}}^{n-1}\to{\mathbb{R}}$ considered above 
are of class $C^k$ (or class $C^{k,\alpha}$, respectively).
\end{definition}

Clearly, any bounded Lipschitz domain is NTA (hence regular for 
the Dirichlet problem), and has an Ahlfors regular boundary.
For further reference, let us also remark here that 
\begin{eqnarray}\label{RFDS}
\mbox{$\Omega\subseteq{\mathbb{R}}^n$ bounded $C^1$ domain}
\,\Longrightarrow\,\kappa_{\Omega}=0.
\end{eqnarray}

For an open set $\Omega\subseteq{\mathbb{R}}^n$ and a number $R>0$, 
define $\Omega_R$ as the collection of points in $\Omega$ at distance 
at least $R$ from the boundary, i.e., 
\begin{eqnarray}\label{RF65}
\Omega_R:=\bigl\{x\in\Omega:\,\delta_{\Omega}(x)>R\bigr\}.
\end{eqnarray}

\begin{definition}\label{Hga.4}
We say that an open set $\Omega\subseteq{\mathbb{R}}^n$ satisfies 
an (axially symmetric) inner cone condition with half-aperture 
$\theta\in(0,\pi/2)$ provided there exists $R\in(0,{\rm diam}\,(\Omega))$ 
with the property that
\begin{eqnarray}\label{RF66}
\forall\,x\in\overline\Omega\setminus\Omega_R\quad
\exists\,\eta\in S^{n-1}\,\,\mbox{ such that }\,\,
\Gamma_\theta^R(x,\eta)\subseteq\Omega.
\end{eqnarray}
More generally, given an open connected $C^{1,\alpha}$ subdomain ${\mathfrak{G}}$ 
of $S^{n-1}$, with $\alpha\in(0,1)$, we say that $\Omega\subseteq{\mathbb{R}}^n$ 
satisfies an inner cone condition with smooth profile ${\mathfrak{G}}$ 
provided there exists $R\in(0,{\rm diam}\,(\Omega))$ so that 
\begin{eqnarray}\label{RF66.X}
\begin{array}{c}
\forall\,x\in\overline\Omega\setminus\Omega_R\,\,\,
\exists\,U\mbox{ isometry of ${\mathbb{R}}^n$ for which} 
\\[4pt]
\mbox{$U(0)=x$ and }\,U(\Gamma_{\mathfrak{G}}\cap B(0,R))\subseteq\Omega.
\end{array}
\end{eqnarray}
\end{definition}

\begin{definition}\label{Vbab}
Given a bounded Lipschitz domain $\Omega\subseteq{\mathbb{R}}^n$, 
with Lipschitz constant $\kappa_{\Omega}\in[0,+\infty)$, 
define $\alpha_{\Omega}$ to be the index associated as 
in Definition~\ref{defalpha} for the angle 
\begin{eqnarray}\label{MS-G3}
\theta=\theta_{\Omega}:=\arctan\Bigl(\frac{1}{\kappa_{\Omega}}\Bigr)
\in (0,{\textstyle\frac{\pi}{2}}),
\end{eqnarray}
that is, $\alpha_{\Omega}=\alpha_{\theta_{\Omega}}$.
\end{definition}

In the context of the above definitions, it is illuminating 
to point out that, in the class of bounded Lipschitz domains,  
\begin{eqnarray}\label{RFDE}
\mbox{$\alpha_{\Omega}\geq 1$ and, in fact, 
$\alpha_{\Omega}\searrow 1$ as $\kappa_{\Omega}\searrow 0$}.
\end{eqnarray}
Indeed, this follows readily from \eqref{alphamap} and 
(the first formula in) \eqref{alphamap2}. On the other hand, 
by the second formula in \eqref{alphamap2}, 
\begin{eqnarray}\label{RFDE.2}
\alpha_{\Omega}\nearrow +\infty\,\,\mbox{ as }\,\,
\kappa_{\Omega}\nearrow +\infty.
\end{eqnarray}
It is also straightforward to check that 
\begin{eqnarray}\label{MS-G2f}
\begin{array}{c}
\mbox{every bounded Lipschitz domain 
$\Omega\subseteq{\mathbb{R}}^n$ satisfies an}
\\[4pt] 
\mbox{inner cone condition with half-aperture $\theta$, 
for any $\theta\in(0,\theta_{\Omega})$.}
\end{array}
\end{eqnarray}
In the opposite direction we note that there exist bounded NTA domains 
which satisfy an inner cone condition but which are not necessarily 
Lipschitz (take, for example, the set-theoretic difference of an open truncated circular cone and a closed truncated circular subcone with smaller aperture which have a common vertex).

We now proceed to discuss a useful bound from below for the Green 
function associated with the Dirichlet Laplacian in bounded NTA domains
satisfying an inner cone condition. It should be noted that in the class 
of bounded Lipschitz domains and for a more restrictive concept of cone 
condition, \cite[Proposition~2, p.\,272]{MS} contains such an
estimate. Similar estimates have also been proved in \cite{Ai3} in the
setting of a uniform domain, which is a more general notion than that
of an NTA domain. However, the estimates in \cite{Ai3} and \cite{MS} 
are given in a form which is not sufficiently explicit for 
our purposes. Here, we largely follow the approach in \cite{Su} with the 
goal of monitoring how the geometrical characteristics of $\Omega$ enter 
the final estimate. 

\begin{proposition}\label{MS-G}
Assume that $\Omega\subseteq{\mathbb{R}}^n$ is a bounded NTA 
domain which satisfies an inner cone condition with smooth profile 
${\mathfrak{G}}\subseteq S^{n-1}$. Let $\alpha_{{\mathfrak{G}}}$ be 
the index associated with the subdomain ${\mathfrak{G}}$ of $S^{n-1}$ 
as in Definition~\ref{defalpha}, and fix $R\in(0,{\rm diam}\,(\Omega)/4)$ 
such that \eqref{RF66.X} holds.  

Then, if $n\geq 3$, there exists a finite constant $c=c(n,{\mathfrak{G}})>0$ 
with the property that the Green function $G(\cdot,\cdot)$ for the 
Dirichlet Laplacian in $\Omega$ satisfies the dilation invariant estimate
\begin{eqnarray}\label{MS-G2}
\hskip -0.20in
G(x,y)\geq c(n,{\mathfrak{G}})
\Bigl(\frac{\delta_{\Omega}(x)}{{\rm diam}\,(\Omega)}
\Bigr)^{\alpha_{{\mathfrak{G}}}}\Bigl(\frac{R}{{\rm diam}\,(\Omega)}
\Bigr)^{m} R^{2-n},\,\,\mbox{ for every $x\in\Omega$ and $y\in\Omega_R$},
\end{eqnarray}
where $m>0$ depends only on the NTA constants of $\Omega$. 
Furthermore, a similar conclusion holds in the case when $n=2$ 
provided the factor $R^{2-n}$ in the right-hand side of 
\eqref{MS-G2} is replaced by $\log\bigl({\rm diam}\,(\Omega)/R\bigr)$. 
\end{proposition}

\begin{proof}
We shall only consider the case $n\geq 3$, since the two-dimensional 
case is treated analogously. The proof is divided into several steps, 
starting with 

\vskip 0.08in
\noindent{\tt Step~1.} Assume that ${\mathfrak{G}}$ is a connected, 
subdomain of class $C^{1,\alpha}$, for some $\alpha\in(0,1)$, of $S^{n-1}$ 
and recall the barrier function 
$v_{\mathfrak{G}}$ from \eqref{vTR}. Also, fix $z\in {\mathfrak{G}}$. 
Then there exists a finite constant $C({\mathfrak{G}},z)>0$ with the 
property that for every $r>0$ one has 
\begin{eqnarray}\label{Gba-1}
w(rz)\,v_{\mathfrak{G}}(x)
\leq C({\mathfrak{G}},z)\,r^{\alpha_{\mathfrak{G}}}w(x),
\qquad\forall\,x\in\Gamma_{\mathfrak{G}}\cap B(0,r),
\end{eqnarray}
for every function 
\begin{eqnarray}\label{Gba-2}
w\in C^0(\overline{\Gamma_{\mathfrak{G}}\cap B(0,2r)})
\,\,\mbox{ satisfying }\,\,w>0\,\,\mbox{ and }\,\,\Delta w=0
\,\,\mbox{ in }\,\,\Gamma_{\mathfrak{G}}\cap B(0,2r).
\end{eqnarray}

It suffices to establish the above claim in the case when $r=1$, 
since then \eqref{Gba-1} follows by rescaling. If this is the case, 
by considering $x\mapsto w(x)/w(z)$ in place of $w(x)$, there
is also no loss of generality in assuming that $w(z)=1$. 
In this scenario, the desired conclusion follows from the Maximum 
Principle as soon as we show that there exists some finite constant 
$C({\mathfrak{G}},z)>0$ such that 
\begin{eqnarray}\label{Gba-3}
v_{\mathfrak{G}}(x)\leq C({\mathfrak{G}},z)\,w(x),\qquad
\forall\,x\in\Gamma_{\mathfrak{G}}\cap\partial B(0,1),
\end{eqnarray}
for every positive function 
$w\in C^0(\overline{\Gamma_{\mathfrak{G}}\cap B(0,2)})$ which is 
harmonic in $\Gamma_{\mathfrak{G}}\cap B(0,2)$ and satisfies $w(z)=1$.
With this goal in mind, we then observe that, by Harnack's inequality 
and the smoothness of ${\mathfrak{G}}$, there exists 
$C=C({\mathfrak{G}},z)>0$ with the property that 
\begin{eqnarray}\label{Gba-4}
w(\omega)\geq C\,{\rm dist}_{S^{n-1}}
\bigl(\omega,\partial_{S^{n-1}}{\mathfrak{G}}\bigr),
\qquad\forall\,\omega\in {\mathfrak{G}},
\end{eqnarray}
whereas, by virtue of \eqref{GFsA},  
\begin{eqnarray}\label{Gba-5}
v_{{\mathfrak{G}}}(\omega)\approx
{\rm dist}_{S^{n-1}}\bigl(\omega,\partial_{S^{n-1}}{\mathfrak{G}}\bigr),\quad
\mbox{ uniformly for }\,\,\omega\in {\mathfrak{G}}. 
\end{eqnarray}
In concert, \eqref{Gba-4} and \eqref{Gba-5} establish estimate \eqref{Gba-3}, 
thus concluding the proof of the claim in Step~1. 

\vskip 0.08in
\noindent{\tt Step~2.} 
Suppose that $\Omega\subseteq{\mathbb{R}}^n$ is a bounded NTA domain. 
Then there exists a dimensional constant $C_n>0$ and some $m>0$ which 
depends only on the NTA constants of $\Omega$ with the property that 
for each $R\in(0,{\rm diam}\,(\Omega))$ the Green function associated 
with the Dirichlet Laplacian in $\Omega$ satisfies 
\begin{eqnarray}\label{Gba-6}
G(x,y)\geq C_n\Bigl(\frac{R}{{\rm diam}\,(\Omega)}\Bigr)^m R^{2-n},
\quad\mbox{ for every }\,\,x,y\in\Omega_R.
\end{eqnarray}

To justify this claim, recall the constant $M$ from Definition~\ref{Def-NTA} 
and pick $m>0$ such that $M=2^m$. Going further, fix $x,y\in\Omega_R$ 
and select a point $y_o\in B(y,R/2)\setminus B(y,R/4)$. Consider now 
a Harnack chain of balls joining $x$ and $y_o$ in $\Omega$. 
More specifically, pick a natural number $k\sim\log_2(|x-y|/R)$ and suppose 
$B_j\subseteq\Omega$, $1\leq j\leq Mk$, is a family of balls such that 
$x$ is the centre of $B_1$, $y_o$ is the centre of $B_{Mk}$, 
$B_j\cap B_{j+1}\neq\emptyset$ for $1\leq j\leq Mk-1$, 
each ball $B_j$ has a radius $r_j$ satisfying 
$M^{-1}r_j\leq {\rm dist}(B_j,\partial\Omega)\leq Mr_j$, as well as
$r_j\geq M^{-1}\,{\rm min}\bigl\{\delta_{\Omega}(x),
\delta_{\Omega}(y_o)\bigr\}$. Then, by repeated applications 
of Harnack's inequality (compare with \eqref{RfD.6}), we obtain 
\begin{eqnarray}\label{Gba-7}
G(x,y)\geq C_nM^{-k}G(y_o,y)\geq 
C_n\Bigl(\frac{R}{|x-y_o|}\Bigr)^m R^{2-n},
\end{eqnarray}
by the choice of $m$, $k$, and $y_o$, and thanks to \eqref{TF-1X}. 
Since $|x-y_o|\leq{\rm diam}\,(\Omega)$, \eqref{Gba-6} follows. 

A moment's reflection shows that \eqref{Gba-6} implies 
\eqref{MS-G2} in the case when $x,y\in\Omega_R$. We continue with:

\vskip 0.08in
\noindent{\tt Step~3.} 
Here we prove the inequality stated in \eqref{MS-G2} in the case when 
$0<R<{\rm diam}\,(\Omega)/4$ and when $y\in\Omega_{4R}$ and 
$x\in\Omega\setminus(\Omega_R\cup B(y,2R))$. Assuming that two such 
points have been fixed, pick $x_\ast\in\partial\Omega$ such that 
$\delta_{\Omega}(x)=|x-x_{\ast}|$, and introduce $x_0:=\frac12(x+x_{\ast})$. 
Also, choose an isometry $U$ of ${\mathbb{R}}^n$ with $U(0)=x_0$ and 
$U(\Gamma_{{\mathfrak{G}}}\cap B(0,2R))\subseteq\Omega$. 
It follows that if $z\in {\mathfrak{G}}$ is fixed, then there exists 
a finite constant $C=C({\mathfrak{G}},z)\geq 1$ such that 
$R U(z)\in\Omega_{R/C}$. For the reader's convenience, the special 
case when ${\mathfrak{G}}$ is a spherical cap on $S^{n-1}$ with 
half-angle $\theta\in(0,\pi/2)$ and when 
$U(\Gamma_{{\mathfrak{G}}}\cap B(0,2R))=\Gamma^{2R}_{\theta}(x_0,\eta)$
for some $\eta\in S^{n-1}$ is sketched in the picture below:  

\begin{center}
\includegraphics[scale=0.65]{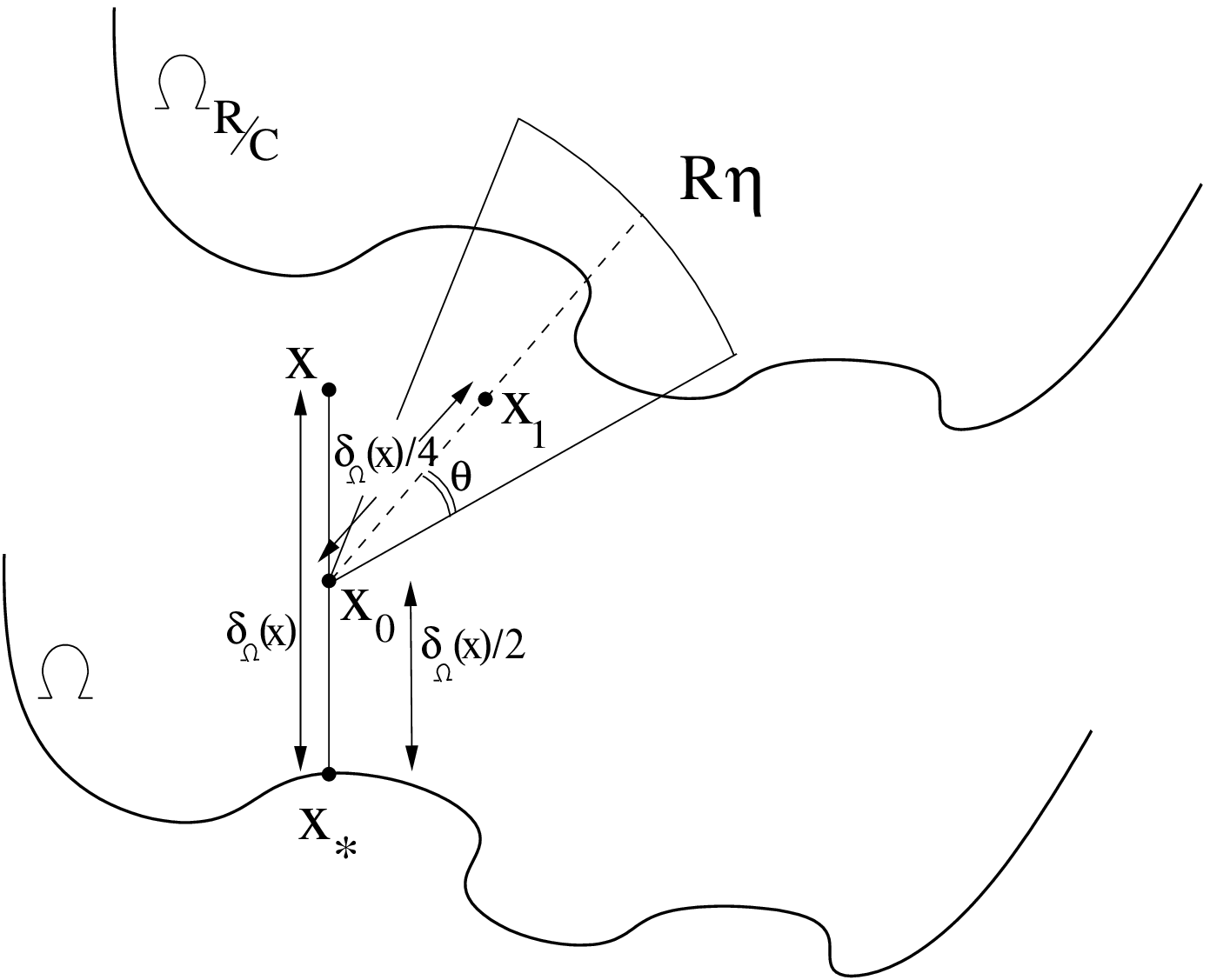}
\end{center}

To continue, introduce $x_1:=U(\delta_{\Omega}(x)z/4)$ which, 
given that $\delta_{\Omega}(x)\leq R$, belongs to the cone 
$U(\Gamma_{{\mathfrak{G}}}\cap B(0,2R))$. Then, on the one hand, Harnack's 
inequality gives 
\begin{eqnarray}\label{Gba-8}
G(x,y)\approx G(x_0,y)\approx G(x_1,y),
\end{eqnarray}
with universal comparability constants, while on the other hand, 
\eqref{Gba-1} applied to the function $w:=G(\cdot,y)$ yields 
\begin{eqnarray}\label{Gba-9}
G(x_1,y) &\geq & C(n,{\mathfrak{G}})R^{-\alpha_{{\mathfrak{G}}}}
|x_1-x_0|^{\alpha_{{\mathfrak{G}}}}G(RU(z),y)
\nonumber\\[4pt]
&\geq & C(n,{\mathfrak{G}})R^{-\alpha_{{\mathfrak{G}}}}
\delta_{\Omega}(x)^{\alpha_{{\mathfrak{G}}}}
\Bigl(\frac{R}{{\rm diam}\,(\Omega)}\Bigr)^m R^{2-n},
\end{eqnarray}
where the last inequality utilises \eqref{Gba-6} and the fact that 
$RU(z)\in\Omega_{R/C}$. Now \eqref{MS-G2} follows in the case we 
are currently considering from \eqref{Gba-8} and \eqref{Gba-9} 
(here we also use the fact that $0<R/{\rm diam}\,(\Omega)<1$ and that 
$\alpha_{{\mathfrak{G}}}>0$). 

The final arguments in the proof of \eqref{MS-G2} are contained in:

\vskip 0.08in
\noindent{\tt Step~4.} 
When $y\in\Omega_{4R}$ and $x\in(\Omega\setminus\Omega_{R})\cap B(y,2R)$ 
we have $\delta_{\Omega}(y)/2\geq 2R\geq |x-y|$, so \eqref{TF-1X} gives
$G(x,y)\geq C_n|x-y|^{2-n}\geq C_nR^{2-n}$. This is good enough to 
justify \eqref{MS-G2} in this case. Granted this and the cases treated
in Steps~2-3, it follows that \eqref{MS-G2} has been proved whenever
$y\in\Omega_{4R}$ and $x\in\Omega$. After relabeling, we may therefore 
conclude that \eqref{MS-G2} holds as stated. 
\end{proof}

The estimate in Proposition~\ref{MS-G} plays a basic role in our next 
theorem, which is the main result in this section. 
 
\begin{theorem}\label{TL-AS}
Assume that $\Omega\subseteq{\mathbb{R}}^n$ is a bounded NTA domain,  
with an Ahlfors regular boundary, and which satisfies an inner 
cone condition with smooth profile ${\mathfrak{G}}\subseteq S^{n-1}$. 
As usual, we denote by $\alpha_{{\mathfrak{G}}}$ the index 
associated with the subdomain ${\mathfrak{G}}$ of $S^{n-1}$ 
as in Definition~\ref{defalpha}. 

Next, let $R\in(0,{\rm diam}\,(\Omega)/4)$ be such that \eqref{RF66.X} 
holds and suppose that $0<\beta<1/\alpha_{{\mathfrak{G}}}$. 
Also, recall that $\Omega_R$ has been introduced in \eqref{RF65}. 
Then, if $n\geq 3$, the solution $u$ of \eqref{PKc-1} satisfies the 
dilation invariant estimate
\begin{eqnarray}\label{ASc-25}
\int_{\Omega}u(x)^{-\beta}\,dx &\leq & C_{\Omega}(n,{\mathfrak{G}},\beta)
\Bigl(\frac{{\rm diam}\,(\Omega)}{R}\Bigr)^{(m+\alpha_{\mathfrak{G}})\beta}
\Bigl(\frac{{\mathcal{H}}^{n-1}(\partial\Omega)}{R^{n-1}}\Bigr)
^{\alpha_{\mathfrak{G}}\beta}\times
\nonumber\\[4pt]
&&\times
\Bigl(\frac{|\Omega|}{R^n}\Bigr)^{2\beta/n-\alpha_{\mathfrak{G}}\beta}
\Bigl(\frac{R^n}{|\Omega_R|}\Bigr)^{\beta}|\Omega|^{1-2\beta/n}
\nonumber\\[4pt]
&\leq & C_{\Omega}(n,{\mathfrak{G}},\beta)
\Bigl(\frac{{\rm diam}\,(\Omega)}{R}\Bigr)^{n+m\beta}
\Bigl(\frac{R^n}{|\Omega_R|}\Bigr)^{1+(n-2)\beta/n}|\Omega|^{1-2\beta/n}
\end{eqnarray}
where $m>0$ depends only on the NTA constants of $\Omega$, and 
$C_{\Omega}(n,{\mathfrak{G}},\beta)>0$ is a finite constant which depends 
only on the Ahlfors character of $\partial\Omega$, the dimension $n$, 
the profile ${\mathfrak{G}}$ and the parameter $\beta$. 

Corresponding to $n=2$, assume that $\Omega\subseteq{\mathbb{R}}^2$ 
is a bounded NTA domain with an Ahlfors regular boundary,
satisfies the inner cone condition \eqref{RF66} with half-aperture 
$\theta\in(0,\pi/2)$ and height $R\in(0,{\rm diam}\,(\Omega)/4)$.
Then if $0<\beta<2\theta/\pi$, the solution $u$ of \eqref{PKc-1} 
satisfies the version of \eqref{ASc-25} written for $n=2$, i.e., 
the dilation invariant estimate 
\begin{eqnarray}\label{ASc-25.X}
\int_{\Omega}u(x)^{-\beta}\,dx &\leq & C_{\Omega}(\theta,\beta)
\Bigl(\frac{{\rm diam}\,(\Omega)}{R}\Bigr)^{(m+\alpha_{\mathfrak{G}})\beta}
\Bigl(\frac{{\mathcal{H}}^1(\partial\Omega)}{R}\Bigr)^{\alpha_{\mathfrak{G}}\beta}
\times
\nonumber\\[4pt]
&&\times\Bigl(\frac{|\Omega|}{R^2}\Bigr)^{\beta-\alpha_{\mathfrak{G}}\beta}
\Bigl(\frac{R^2}{|\Omega_R|}\Bigr)^{\beta}
|\Omega|^{1-\beta}
\end{eqnarray}
where, as before, $m>0$ depends only on the NTA constants of $\Omega$, 
and $C_{\Omega}(\theta,\beta)>0$ is a finite constant which depends 
only on the Ahlfors character of $\partial\Omega$, the angle $\theta$, 
and the parameter $\beta$. 
\end{theorem}

\begin{proof}
Suppose that $n\geq 3$. The representation in \eqref{PKc-2}, together 
with the nonnegativity of the Green function and estimate \eqref{MS-G2} give 
\begin{eqnarray}\label{TF-2SX}
u(x) &=& \int_\Omega G(x,y)\,dy\geq\int_{\Omega_R}G(x,y)\,dy
\nonumber\\[4pt]
&\geq & c(n,{\mathfrak{G}})
\Bigl(\frac{\delta_{\Omega}(x)}{{\rm diam}\,(\Omega)}\Bigr)
^{\alpha_{{\mathfrak{G}}}}\Bigl(\frac{R}{{\rm diam}\,(\Omega)}
\Bigr)^{m} R^{2-n}|\Omega_R|,\qquad\forall\,x\in\Omega.
\end{eqnarray}
With \eqref{TF-2SX} in hand, \eqref{ASc-25} follows from Lemma~\ref{bvv55}
(recall that $0<\beta<1/\alpha_{\mathfrak{G}}$), after some simple algebra 
(and using the fact that $|\Omega_R|\leq |\Omega|$). The case $n=2$ is similar.
More specifically, the same type of argument as above yields the bound 
\begin{eqnarray}\label{ASjj}
\int_{\Omega}u(x)^{-\beta}\,dx &\leq & C(\theta,\beta)
\Bigl(\frac{{\rm diam}\,(\Omega)}{R}\Bigr)^{(m+\alpha_{\mathfrak{G}})\beta}
\Bigl(\frac{{\mathcal{H}}^1(\partial\Omega)}{R}\Bigr)^{\alpha_{\mathfrak{G}}\beta}
\times
\nonumber\\[4pt]
&&\times\Bigl(\frac{|\Omega|}{R^2}\Bigr)^{\beta-\alpha_{\mathfrak{G}}\beta}
\Bigl(\frac{R^2}{|\Omega_R|}\Bigr)^{\beta}
|\Omega|^{1-\beta}
\Bigl(\log\Bigl(\frac{{\rm diam}\,(\Omega)}{R}\Bigr)\Bigr)^{-\beta}
\end{eqnarray}
and, given that ${\rm diam}\,(\Omega)/R>4$, the logarithmic factor 
can be bounded by $(\log\,4)^{-\beta}$. This gives \eqref{ASc-25.X}. 
\end{proof}

We continue by recording the following corollary. 

\begin{corollary}\label{anLL}
If $\Omega\subseteq{\mathbb{R}}^n$, $n\geq 2$, is a 
bounded Lipschitz domain and if $\alpha_{\Omega}$ is the critical exponent 
associated with $\Omega$ as in Definition~\ref{Vbab}, then the finiteness
condition \eqref{PKc-4} holds granted that 
\begin{eqnarray}\label{ASc-10}
0<\beta<\frac{1}{\alpha_{\Omega}}.
\end{eqnarray}
In particular, \eqref{PKc-4} holds for any $\beta\in(0,1)$ 
in the case when $\Omega$ is a bounded $C^1$ domain. 
\end{corollary}

\begin{proof}
The claim in the first part of the statement is an immediate consequence 
of our previous theorem, whereas \eqref{ASc-25} and \eqref{RFDE} readily 
yield the claim in the second part of the statement. 
\end{proof}

The principle emerging from Theorem~\ref{TL-AS} is that, for a  
bounded NTA domain $\Omega\subseteq{\mathbb{R}}^n$ with an Ahlfors 
regular boundary, the ratio 
\begin{eqnarray}\label{thg.5}
\Bigl(\int_{\Omega}u(x)^{-\beta}\,dx\Bigr)
\Big/|\Omega|^{1-2\beta/n}
\end{eqnarray}
can be controlled in terms of the proportion of the size of the cone 
(involved in cone condition \eqref{RF66.X}) relative to the size of 
the domain $\Omega$ itself (assuming that $\beta\in(0,1)$ relates favourably 
to the spherical profile of the cone). 

An example of this principle at work in a concrete case of interest 
is as follows. For each $\kappa\in(0,1)$ and $N\geq 3$, denote by 
${\mathcal{P}}(\kappa,N)$ the set of polygons with $N$ sides inscribed 
in the unit circle $S^1$ and having the property that the ratio of any 
two sides belongs to the interval $(\kappa,\kappa^{-1})$. Theorem~\ref{TL-AS} 
then gives that, once $\kappa\in(0,1)$ and $\beta\in(0,1)$ have been fixed, 
there exist $N_{\beta,\kappa}\in{\mathbb{N}}$ and $C_{\beta,\kappa}>0$ 
with the property that 
\begin{eqnarray}\label{thg.6}
\int_{\Omega}u(x)^{-\beta}\,dx\leq C_{\beta,\kappa},\qquad
\mbox{whenever $\Omega\in{\mathcal{P}}(\kappa,N)$ with 
$N\geq N_{\beta,\kappa}$}.
\end{eqnarray}
In particular, if $\Omega_N$ denotes the regular polygon with $N$ sides
($N\in{\mathbb{N}}$, $N\geq 3$) inscribed in $S^1$ and $u_N$ is the 
solution of \eqref{PKc-1} for $\Omega=\Omega_N$, then \eqref{thg.6} 
gives that for every fixed $\beta\in(0,1)$ we have 
\begin{eqnarray}\label{thg.7}
\int_{\Omega_N}u_N(x)^{-\beta}\,dx={\mathcal{O}}(1),\qquad
\mbox{as $N\to\infty$}.
\end{eqnarray}

Improvements of \eqref{thg.6}-\eqref{thg.7} 
(vis-\`{a}-vis the range of $\beta$'s, the shape of the polygon and 
the nature of the estimate for the $\beta$-integral) are presented 
in \S\,\ref{sect:4}. For the time being, we wish to point out that (for
any $0 < \beta < 1$)
\begin{eqnarray}\label{thg}
\sup_{\Omega}\,\Bigl(\int_{\Omega}u(x)^{-\beta}\,dx\Bigr)
\Big/|\Omega|^{1-2\beta/n}=+\infty,
\end{eqnarray}
when the supremum is taken over all bounded convex sets in ${\mathbb{R}}^n$. 
With this goal in mind, for a fixed, small $\varepsilon>0$, consider 
the thin rectangular domain 
\begin{eqnarray}\label{jhab.1}
\Omega:=\bigl\{x=(x',x_n)\in{\mathbb{R}}^{n-1}\times{\mathbb{R}}:\,
x'\in(0,1)^{n-1},\,\,|x_n|<\varepsilon\bigr\}\subseteq{\mathbb{R}}^n,
\end{eqnarray}
and set
\begin{eqnarray}\label{jhab.2}
v(x',x_n):={\textstyle\frac12}(\varepsilon^2-x_n^2),\qquad
\forall\,x=(x',x_n)\in\Omega.  
\end{eqnarray}
Then $v\in C^0(\overline{\Omega})$, 
$-\Delta v=1$ on $\Omega$, and $v\geq 0$ on $\partial\Omega$. 
Therefore, if $u$ solves \eqref{PKxc-1} for $\Omega$ as in \eqref{jhab.1},
we have $u(x)\leq v(x)$ for every $x\in\Omega$, 
on account of the Maximum Principle. As a result, for every $\beta\in(0,1)$ 
we may estimate 
\begin{eqnarray}\label{jhab}
\int_{\Omega} u(x)^{-\beta}\,dx &\geq &\int_{\Omega}v(x)^{-\beta}\,dx
=2^{\beta}\int_{-\varepsilon}^{\varepsilon}(\varepsilon^2-x_n^2)^{-\beta}\,dx_n
\nonumber\\[4pt] 
&=& 
2^{1+\beta}\Bigl(\int_{0}^{1}(1-t^2)^{-\beta}\,dt\Bigr)
\,\varepsilon^{1-2\beta}=C_\beta\,\varepsilon^{1-2\beta},
\nonumber\\[4pt] 
&=& C_\beta\,\varepsilon^{-2\beta(n-1)/n}\,|\Omega|^{1-2\beta/n},
\end{eqnarray}
from which \eqref{thg} readily follows. 

A natural end-point version of the estimate \eqref{PKc-4} is the `weak-type' inequality
\begin{equation} \label{weaktype}
|\{ x \in \Omega \, | \, u(x) < \lambda \}| \leq C(\Omega,\tilde{\beta})\lambda^{\tilde{\beta}} < + \infty \quad \mbox{for all $\lambda > 0$.}
\end{equation}
The two conditions \eqref{PKc-4} and \eqref{weaktype} are closely related, in that if \eqref{PKc-4} holds for some $\beta > 0$ then \eqref{weaktype} holds for $0 < \tilde{\beta} \leq \beta$, and if \eqref{weaktype} holds for some $\tilde{\beta} > 0$ then \eqref{PKc-4} holds for $0 < \beta < \tilde{\beta}$. These two statements follow easily from Chebyshev's inequality and the equality
\begin{equation} \label{Done}
\int_{\Omega} u(x)^{-\beta} dx = \beta \int_0^\infty \lambda^{-(\beta + 1)}|\{ x \in \Omega \, | \, u(x) < \lambda \}| d\lambda,
\end{equation}
respectively.

In the case of a ball, the calculation in Remark \ref{Rf-1.H} shows that \eqref{weaktype} holds if and only if $\tilde{\beta} \leq 1$ and so, while \eqref{PKc-4} fails when $\beta = 1$ (see \eqref{PKc-1.E}), \eqref{weaktype} holds when $\tilde{\beta} = 1$.

More general examples may be established via the same methods we have employed above. For example, an immediate consequence of \eqref{baD-ii} is
\begin{eqnarray}\label{Wtwo}
|\{ x \in \Omega \, | \, \delta_\Omega(x)^{\tilde{\beta}} < \lambda \}| \leq C\lambda^{1/\tilde{\beta}}\,{\mathcal{H}}^{n-1}(\partial\Omega),
\end{eqnarray}
for each $\tilde{\beta} \geq 0$ and all $\lambda > 0$. This can be used to prove the following theorem.
\begin{theorem} \label{Wthree}
(i) If $\Omega\subseteq{\mathbb{R}}^n$ is a bounded domain whose boundary 
is Ahlfors regular, then \eqref{weaktype} holds provided that $0 < \tilde{\beta} \leq 1/2$.

(ii) If $\Omega\subseteq{\mathbb{R}}^n$, $n\geq 2$, is a 
bounded Lipschitz domain and if $\alpha_{\Omega}$ is the critical exponent 
associated with $\Omega$ as in Definition~\ref{Vbab}, then \eqref{weaktype} holds provided that 
\begin{eqnarray}\label{Wfour}
0<\beta\leq\frac{1}{\alpha_{\Omega}}.
\end{eqnarray}
\end{theorem}
Part~(i) of Theorem~\ref{Wthree} contains the appropriate end-point version of Proposition~\ref{PKc-1.F}~(ii) corresponding to $\beta = 1/2$ and is proved using \eqref{Wtwo} and \eqref{TF-2.D}. Proposition~\ref{PKc-1.F}~(iii) and the above discussion show this is sharp. Part~(ii) corresponds to the end-point $\beta = 1/{\alpha_{\Omega}}$ of Corollary~\ref{anLL} and is proved again using \eqref{Wtwo} and, this time, \eqref{TF-2SX}.

We end the current section by recording a special case of 
Theorem~\ref{TL-AS} of independent interest. This requires that 
we first make the following definition.

\begin{definition}\label{nab.67}
We say that an open set $\Omega\subseteq{\mathbb{R}}^n$ satisfies 
an inner ball condition with radius $r_0\in(0,+\infty)$ provided
\begin{eqnarray}\label{hana}
\forall\,x\in\Omega\,\,\,\exists\,y\in\Omega\,\,\mbox{ such that }\,\,
x\in B(y,r_0)\subseteq\Omega.
\end{eqnarray}
\end{definition}
In other words, an open set $\Omega\subseteq{\mathbb{R}}^n$ satisfies 
an inner ball condition with radius $r_0$ provided $\Omega$ can be 
written as the union of all balls of radius $r_0$ contained in $\Omega$. 

Obviously, an open set $\Omega\subseteq{\mathbb{R}}^n$ satisfying 
an inner ball condition with radius $r_0$ also satisfies an axially 
symmetric inner cone condition with any half-aperture $\theta\in(0,\pi/2)$
and any height $\leq 2r_0\cos\theta$. This observation and 
Theorem~\ref{TL-AS} then readily yield the following corollary.

\begin{corollary}\label{Tgab}
Assume that $\Omega\subseteq{\mathbb{R}}^2$ is a bounded NTA domain 
with an Ahlfors regular boundary, which satisfies an inner 
ball condition with radius $r_0$. Fix $\beta\in(0,1)$ and select 
$\theta\in(\beta\pi/2,\pi/2)$. Then the solution $u$ of \eqref{PKc-1} 
satisfies 
\begin{eqnarray}\label{Ah-nb}
\int_{\Omega}u(x)^{-\beta}\,dx &\leq & C_{\Omega}(\theta,\beta)
\Bigl(\frac{{\rm diam}\,(\Omega)}{r_0}\Bigr)^{m\beta}
\Bigl(\frac{{\rm diam}\,(\Omega)^2}{|\Omega_{r_0}|}\Bigr)|\Omega|^{1-\beta}
\nonumber\\[4pt]
&\leq & C_{\Omega}(\theta,\beta)
\Bigl(\frac{{\rm diam}\,(\Omega)}{r_0}\Bigr)^{m\beta+2}|\Omega|^{1-\beta}, 
\end{eqnarray}
where $m>0$ depends only on the NTA constants of $\Omega$, 
and $C_{\Omega}(\theta,\beta)>0$ is a finite constant which depends 
only on the Ahlfors character of $\partial\Omega$, the angle $\theta$, 
and the parameter $\beta$. 

In fact, a result similar in spirit holds in 
the case when $\Omega\subseteq{\mathbb{R}}^n$ with $n\geq 3$ as well. 
\end{corollary}

\section{The case of a polygon in the plane}
\label{sect:4}
\setcounter{equation}{0}

In this section we focus on the finiteness of the $\beta$-integral 
(cf.~\eqref{PKc-4}) in the situation when $\Omega$ is a polygonal 
domain in ${\mathbb{R}}^2$. Some preparations are necessary. 
Given $\theta\in(0,\pi)$, consider the infinite sector 
\begin{eqnarray}\label{PKc-5}
S_{\theta}:=\{z\in{\mathbb{C}}={\mathbb{R}}^2:\,
|{\rm arg}\,(z)|<\theta\}
\end{eqnarray}
and, for each $r>0$, consider its truncated version 
\begin{eqnarray}\label{PKc-6}
S_{\theta,r}=\{z\in S_{\theta}:\,|z|<r\}.
\end{eqnarray}
Hence, in polar coordinates $x=(\rho\cos\omega,\rho\sin\omega)
\in{\mathbb{R}}^2$ with $(\rho,\omega)\in(0,\infty)\times (-\pi,\pi)$,
\begin{eqnarray}\label{PKc-8}
S_{\theta,r}=\{(\rho,\omega):\,0<\rho<r\,\,\mbox{ and }\,\,
-\theta<\omega<\theta\}.
\end{eqnarray}

One basic technical result in this section is contained in the 
next proposition below. To be able to formulate it, we will need 
the Gamma and Beta functions which, for the convenience of the reader, 
we now briefly recall. As is well-known, they are respectively given by 
\begin{eqnarray}\label{TFS-22}
\Gamma(z):=\int_0^\infty t^{z-1}e^{-t}\,dt,\qquad z\in{\mathbb{C}},\,\,\,
{\rm Re}\,z>0,
\end{eqnarray}
and 
\begin{eqnarray}\label{TFS-23}
{\mathfrak{B}}(z_1;z_2):=\int_0^1 t^{z_1-1}(1-t)^{z_2-1}\,dt, 
\qquad z_j\in{\mathbb{C}},\,\,\,{\rm Re}\,z_j>0,\,\,\,j=1,2.
\end{eqnarray}
with both integrals convergent under the specified conditions.
It will also be useful to recall that an alternative formula for the
Beta function is 
\begin{eqnarray}\label{TFS-24}
{\mathfrak{B}}(z_1;z_2)=2\int_0^{\pi/2}(\sin\alpha)^{2z_1-1}
(\cos\alpha)^{2z_2-1}\,d\alpha, 
\qquad z_j\in{\mathbb{C}},\,\,\,{\rm Re}\,z_j>0,\,\,\,j=1,2.
\end{eqnarray}
Here is the proposition alluded to above.

\begin{proposition}\label{TL-3}
Suppose that $\Omega$ is a bounded domain in ${\mathbb{R}}^2$ for which there 
exist $r>0$ and $\theta\in(0,\pi)$ such that 
\begin{eqnarray}\label{PKc-18}
\Omega\cap B(0,r)=S_{\theta,r}.
\end{eqnarray}
Also, let $u$ be the function defined by \eqref{PKc-1}. 
Then for every $\beta\in(0,1)$ there holds 
\begin{eqnarray}\label{PKc-25}
\int_{S_{\theta,r}}u(x)^{-\beta}\,dx\leq C(\theta,\beta)\,r^{2(1-\beta)},
\end{eqnarray}
where $C(\theta,\beta)>0$ is the finite constant described as 
\begin{eqnarray}\label{PKA-5T}
C(\theta,\beta):=
\left\{
\begin{array}{l}
\Big(\Bigl(\frac{\pi}{2\theta}\Bigr)^2-4\Bigr)^{\beta}
\frac{(2\theta)^2}{\pi(\pi-4\theta)}
\mathfrak{B}\Bigl(\frac{\theta(1-\beta)}{\pi/4-\theta};1-\beta\Bigr)
{\mathfrak{B}}\Bigl(\frac{1}{2};\frac{1-\beta}{2}\Bigr)
\,\,\mbox{ if }\,\,\theta\in\bigl(0,{\textstyle\frac{\pi}{4}}\bigr),
\\[12pt]
2^{3\beta-2}\Bigl(\frac{1}{1-\beta}\Bigr)^{1-\beta}\Gamma(1-\beta)
{\mathfrak{B}}\Bigl(\frac{1}{2};\frac{1-\beta}{2}\Bigr)
\hskip 1.02in
\mbox{ if }\,\,\theta={\textstyle\frac{\pi}{4}},
\\[12pt]
\Big(4-\Bigl(\frac{\pi}{2\theta}\Bigr)^2\Bigr)^{\beta}
\frac{(2\theta)^2}{\pi(4\theta-\pi)}
\mathfrak{B}\Bigl(\frac{\theta-\beta\pi/2}{\theta-\pi/4};1-\beta\Bigr)
{\mathfrak{B}}\Bigl(\frac{1}{2};\frac{1-\beta}{2}\Bigr)
\,\,\mbox{ if }\,\,\theta\in\bigl({\textstyle\frac{\pi}{4}},\pi\bigr].
\end{array}
\right.
\end{eqnarray}
As a consequence, for every $\beta\in(0,1)$, 
\begin{eqnarray}\label{PKc-25.D}
\int_{S_{\theta,r}}u(x)^{-\beta}\,dx\leq C_{\beta}\,
\theta^{1-2\beta}\,r^{2(1-\beta)}.
\end{eqnarray}
The bound in \eqref{PKc-25.D} is in the nature of best possible, 
in the sense that if $\Omega:=S_{\theta,r}$ for some $\theta\in(0,\pi)$ 
and $r>0$, and if $u$ is the solution of \eqref{PKc-1} for this domain 
then, in fact, for every $\beta\in(0,1)$, 
\begin{eqnarray}\label{PKc-25.D2}
\int_{S_{\theta,r}}u(x)^{-\beta}\,dx\approx 
\theta^{1-2\beta}\,r^{2(1-\beta)},\quad\mbox{ uniformly for }\,\,
\theta\in(0,\pi)\,\,\mbox{ and }\,\,r>0,
\end{eqnarray}
with comparability constants which depend exclusively on $\beta$.
\end{proposition}

To get a better feel for the constant defined in \eqref{PKA-5T}, 
a few comments are in order. Since, as is well-known, we have 
the following asymptotic formula 
\begin{eqnarray}\label{TGda-DS}
{\mathfrak{B}}(x,y)\sim\Gamma(y)\,x^{-y}\quad
\mbox{when $x>0$ is large, for each fixed $y>0$},
\end{eqnarray}
we deduce from this and \eqref{PKA-5T} that, 
for each $\beta\in(0,1)$ fixed, 
\begin{eqnarray}\label{HCSa}
C(\theta,\beta)\sim
\pi^{-1}(1-\beta)^{\beta-1}\theta^{1-\beta}(\pi+4\theta)^{\beta}\,
\Gamma(1-\beta){\mathfrak{B}}\Bigl(\frac{1}{2};\frac{1-\beta}{2}\Bigr)
\,\,\mbox{ for $\theta$ close to ${\textstyle\frac{\pi}{4}}$}.
\end{eqnarray}
In particular, this shows that the functions 
$\theta\mapsto C(\theta,\beta)$ from \eqref{PKA-5T} are continuous 
at $\pi/4$. In this vein, let us also remark here that since 
\begin{eqnarray}\label{T.i-BN}
\lim_{x\to 0^{+}}\Bigl(x{\mathfrak{B}}(x,y)\Bigr)=1\qquad
\mbox{for each fixed $y>0$},
\end{eqnarray}
it follows from \eqref{PKA-5T} that, for each $\beta\in(0,1)$ fixed, 
\begin{eqnarray}\label{Kan}
C(\theta,\beta)\sim\theta^{1-2\beta},
\qquad\mbox{for $\theta$ close to zero}.
\end{eqnarray} 
Altogether, the above analysis shows that for each fixed $\beta\in(0,1)$, 
the quantity $C(\theta,\beta)$ depends continuously on 
$\theta\in(0,\pi]$ and satisfies 
\begin{eqnarray}\label{GBva}
C(\theta,\beta)\approx\theta^{1-2\beta},
\qquad\mbox{uniformly for $\theta\in(0,\pi]$}.
\end{eqnarray}

The proof of the above proposition requires further preparations. 
For each $\theta\in(0,\pi)$ and $r>0$ consider the barrier function 
$v_{\theta,r}:S_{\theta,r}\to{\mathbb{R}}$ which, in polar 
coordinates $(\rho,\omega)$, is given by 
\begin{eqnarray}\label{PKc-7}
v_{\theta,r}(\rho,\omega):=
\left\{
\begin{array}{l}
\Big(\Bigl(\frac{\pi}{2\theta}\Bigr)^2-4\Bigr)^{-1}
r^2\Bigl[\Bigl(\frac{\rho}{r}\Bigr)^2
-\Bigl(\frac{\rho}{r}\Bigr)^{\pi/(2\theta)}\Bigr]
\cos\,\Bigl(\frac{\pi\omega}{2\theta}\Bigr)
\,\,\mbox{ if }\,\,\theta\in\bigl(0,\pi\bigr)
\setminus\bigl\{{\textstyle{\frac{\pi}{4}}}\bigr\},
\\[16pt]
\frac{\rho^2}{4}\log\Big(\frac{r}{\rho}\Bigr)\cos\,(2\omega)
\,\,\mbox{ if }\,\,\theta={\textstyle{\frac{\pi}{4}}}.
\end{array}
\right.
\end{eqnarray}
It is reassuring to observe that, because of the differentiation 
quotient present in \eqref{PKc-7}, which can be highlighted by writing 
\begin{eqnarray}\label{JFc-2}
v_{\theta,r}(\rho,\omega) &=& 
\Big(\Bigl(\frac{\pi}{2\theta}\Bigr)^2-4\Bigr)^{-1}
r^2\Bigl[\Bigl(\frac{\rho}{r}\Bigr)^2
-\Bigl(\frac{\rho}{r}\Bigr)^{\pi/(2\theta)}\Bigr]
\cos\,\Bigl(\frac{\pi\omega}{2\theta}\Bigr)
\nonumber\\[4pt]
&=& -r^2\Big(\frac{\pi}{2\theta}+2\Bigr)^{-1}
\frac{\Bigl(\frac{\rho}{r}\Bigr)^2
-\Bigl(\frac{\rho}{r}\Bigr)^{\pi/(2\theta)}}{2-\frac{\pi}{2\theta}}
\cos\,\Bigl(\frac{\pi\omega}{2\theta}\Bigr),
\end{eqnarray}
the formula for $v_{\theta,r}$ corresponding to the special value 
$\theta=\pi/4$ (i.e., second line of \eqref{PKc-7}) is the natural
limit case of the formula for $v_{\theta,r}$ in the first 
line of \eqref{PKc-7} as $\theta\to\pi/4$. 

The above barrier function has been designed precisely as to satisfy, for 
any $\theta\in(0,\pi)$ and any $r>0$ 
\begin{eqnarray}\label{PKc-1.D}
\begin{array}{l}
v_{\theta,r}=0\,\,\mbox{ on }\,\,\partial S_{\theta,r},\qquad
v_{\theta,r}>0\,\,\mbox{ in }\,\,S_{\theta,r},\quad\mbox{and}\quad
\\[6pt]
-(\Delta v_{\theta,r})(\rho,\omega)=\cos\,(\pi\omega/(2\theta))
\,\,\mbox{ in }\,\,S_{\theta,r}.
\end{array}
\end{eqnarray}
The last property is verified by means of an elementary 
calculation based on the fact that, in polar coordinates 
in the plane, the Laplacian can be written as 
$\Delta=d^2/d\rho^2+\rho^{-1}d/d\rho+\rho^{-2}d^2/d\omega^2$. 
The normalisation constants in \eqref{PKc-7} have been 
selected so that the right-hand side in the second line of 
\eqref{PKc-1.D} is precisely a cosine (this will be of relevance 
shortly; cf.~\eqref{PKc-22} below). While checking the last 
formula in \eqref{PKc-1.D} it also helps to notice that the function 
\begin{eqnarray}\label{TFSA}
\widetilde{v}_{\theta}(\rho,\omega)
:=\rho^{\pi/(2\theta)}\cos\,\Bigl(\frac{\pi\omega}{2\theta}\Bigr)
={\rm Im}\,[z^{\pi/(2\theta)}],\qquad z=\rho\,e^{i\omega},
\end{eqnarray}
is harmonic for every $\theta\in(0,\pi)$ and every $r>0$. 
It is instructive to note that $\widetilde{v}_{\theta}$ from \eqref{PKc-7} is 
the first singular function arising from the Mellin symbol of the Dirichlet-Laplacian
on the interval $(-\theta,\theta)$, while the function $v_{\theta,r}$ from 
\eqref{TFSA} is closely related to the term of degree $2$ in the corner asymptotics 
of $u$ (note that this asymptotics contains a log term only if $\theta=\frac{\pi}{4}$).
The preference of $v_{\theta,r}$ over $\widetilde{v}_{\theta}$ is then justified
by observing that, for small values of $\theta$ (more precisely, for 
$\theta\in(0,\pi/4)$), the corner asymptotics of $u$ is dominated by its term 
of degree $2$, and not by its term of degree $\frac{\pi}{2\theta}$.
This aspect\footnote{We owe this insight to one of the referees.}
plays a crucial role in our subsequent analysis. 

With the above notation and conventions we have: 

\begin{lemma}\label{TL-1}
For any $\beta\in(0,1)$, any $r>0$ and any $\theta\in (0,\pi)$,
we have 
\begin{eqnarray}\label{PKc-9}
\int_{S_{\theta,r}}v_{\theta,r}(x)^{-\beta}\,dx
=C(\theta,\beta)\,r^{2(1-\beta)}
\end{eqnarray}
where $C(\theta,\beta)>0$ is the finite constant given in 
\eqref{PKA-5T}.
\end{lemma}
\begin{proof}
Fix $\beta\in(0,1)$, let $r>0$ be arbitrary and first assume that
$\theta\in(0,\pi/4)$. We have 
\begin{eqnarray}\label{DSw-1}
\int_{S_{\theta,r}}v_{\theta,r}(x)^{-\beta}\,dx
&=&\Big(\Bigl(\frac{\pi}{2\theta}\Bigr)^2-4\Bigr)^{\beta}r^{-2\beta}
\int_{S_{\theta,r}}\Bigl[\Bigl(\frac{\rho}{r}\Bigr)^2
-\Bigl(\frac{\rho}{r}\Bigr)^{\pi/(2\theta)}\Bigr]^{-\beta}
\Bigl[\cos\,\Bigl(\frac{\pi\omega}{2\theta}\Bigr)\Bigr]^{-\beta}dx
\nonumber\\[4pt]
&=& \Big(\Bigl(\frac{\pi}{2\theta}\Bigr)^2-4\Bigr)^{\beta}r^{-2\beta}
I\cdot II,
\end{eqnarray}
where 
\begin{eqnarray}\label{DSw-2}
I:=\int_0^r\Bigl[\Bigl(\frac{\rho}{r}\Bigr)^2
-\Bigl(\frac{\rho}{r}\Bigr)^{\pi/(2\theta)}\Bigr]^{-\beta}
\rho\,d\rho\quad\mbox{ and }\quad
II:=\int_{-\theta}^{\theta}
\Bigl[\cos\,\Bigl(\frac{\pi\omega}{2\theta}\Bigr)\Bigr]^{-\beta}d\omega.
\end{eqnarray}
Making two changes of variables, first introducing $t:=\rho/r$, 
then substituting $s$ for $t^{\varepsilon}$, where 
$\varepsilon:=\pi/(2\theta)-2>0$, yields 
\begin{eqnarray}\label{DSw-3}
I &=&\varepsilon^{-1}r^2\int_0^1 
s^{2(1-\beta)/\varepsilon-1}(1-s)^{-\beta}\,ds=\frac{r^2}{\varepsilon}
{\mathfrak{B}}\Bigl(\frac{2(1-\beta)}{\varepsilon};1-\beta\Bigr)
\nonumber\\[4pt]
&=& r^2\Bigl(\frac{\pi}{2\theta}-2\Bigr)^{-1}
\mathfrak{B}\Bigl(\frac{\theta(1-\beta)}{\pi/4-\theta};1-\beta\Bigr),
\end{eqnarray}
after unravelling notation. On the other hand, making the change 
of variables $\alpha:=\pi\omega/(2\theta)$ and using the parity 
of the cosine function permits one to write (after a reference 
to \eqref{TFS-24})
\begin{eqnarray}\label{DSw-4}
II=\frac{4\theta}{\pi}\int_{0}^{\pi/2}(\cos\alpha)^{-\beta}\,d\alpha
=\frac{2\theta}{\pi}{\mathfrak{B}}\Bigl(\frac{1}{2};\frac{1-\beta}{2}\Bigr).
\end{eqnarray}
Collectively, \eqref{DSw-1}-\eqref{DSw-4} establish the validity 
of \eqref{PKc-9} in the case when $\theta\in(0,\pi/4)$ and with 
a constant $C(\theta,\beta)$ as in the first line in the right-hand 
side of \eqref{PKA-5T}. The case when $\theta\in(\pi/4,\pi)$ is 
treated in a most analogous manner and we omit it.

Finally, corresponding to $\theta=\pi/4$, we have 
\begin{eqnarray}\label{DSw-5}
\int_{S_{\pi/4,r}}v_{\pi/4,r}(x)^{-\beta}\,dx
&=& 4^{\beta}\int_{S_{\pi/4,r}}
\rho^{-2\beta}\Bigl[\log\Bigl(\frac{r}{\rho}\Bigr)\Bigr]^{-\beta}
\Bigl[\cos\,(2\omega)\Bigr]^{-\beta}dx
\nonumber\\[4pt]
&=& 4^{\beta}r^{2-2\beta}III\cdot IIV,
\end{eqnarray}
where, after the changes of variables $t=\rho/r$ and $\alpha=2\omega$, 
\begin{eqnarray}\label{DSw-6}
III:=\int_0^1 t^{1-2\beta}
\Bigl[\log\Bigl(\frac{1}{t}\Bigr)\Bigr]^{-\beta}\,dt
\quad\mbox{ and }\quad
IV:=2^{-1}\int_{-\pi/2}^{\pi/2}(\cos\,\alpha)^{-\beta}d\alpha.
\end{eqnarray}
One more change of variables, substituting ${\rm exp}(-s/(2-2\beta))$
for $t$ in $III$, transforms this term into 
\begin{eqnarray}\label{DSw-7}
III=\Bigl(\frac{1}{2-2\beta}\Bigr)^{1-\beta}
\int_0^\infty e^{-s}s^{-\beta}\,ds=2^{\beta-1}
\Bigl(\frac{1}{1-\beta}\Bigr)^{1-\beta}\Gamma(1-\beta),
\end{eqnarray}
whereas, much as before, 
\begin{eqnarray}\label{DSw-8}
IV=2^{-1}{\mathfrak{B}}\Bigl(\frac{1}{2};\frac{1-\beta}{2}\Bigr).
\end{eqnarray}
In concert, \eqref{DSw-5}-\eqref{DSw-8} justify \eqref{PKc-9} in the 
case when $\theta=\pi/4$, with a constant $C(\theta,\beta)$ as 
in the middle line in the right-hand side of \eqref{PKA-5T}. 
This completes the proof of the lemma. 
\end{proof}

After this preamble, here is the end-game in the following proof.

\vskip 0.08in
\noindent{\it Proof of Proposition~\ref{TL-3}.}
Let $r>0$, $\theta\in(0,\pi)$ be as in the statement of the proposition. 
With $v_{\theta,r}$ as in \eqref{PKc-7}, it follows that the function 
$w:=u-v_{\theta,r}:S_{\theta,r}\to{\mathbb{R}}$ satisfies
(recall that we are assuming that $\Delta u=-1$ in 
$\Omega\supseteq S_{\theta,r}$)
\begin{eqnarray}\label{PKc-22}
(\Delta w)(\rho,\omega)=-1+\cos\,(\pi\omega/(2\theta))\leq 0\quad
\,\,\mbox{ for each }\,\,\rho\,e^{i\omega}\in S_{\theta,r},
\end{eqnarray}
i.e., $w$ is superharmonic in $S_{\theta,r}$. In addition, 
$w$ is continuous in $\overline{S_{\theta,r}}$ and 
\begin{eqnarray}\label{PKc-23}
w\bigl|_{\partial S_{\theta,r}}=u\bigl|_{\partial S_{\theta,r}}\geq 0
\end{eqnarray}
given that, by design, $u$ is nonnegative in $\overline{\Omega}$. 
Hence, the Maximum Principle applies and yields $w\geq 0$ in  
$S_{\theta,r}$ or, in other words, 
\begin{eqnarray}\label{PKc-23DS}
u\geq v_{\theta,r}\,\,\mbox{ in }\,\, S_{\theta,r}.
\end{eqnarray}
The estimate \eqref{PKc-25} now readily follows by combining 
\eqref{PKc-23DS} with the result proved in Lemma~\ref{TL-1}.
With this in hand, \eqref{PKc-25.D} is a direct consequence 
of \eqref{PKc-25} and \eqref{Kan} (cf.~also the claim following 
this last equation). 

Finally, there remains to prove the equivalence in \eqref{PKc-25.D2}. 
Of course, the left-pointing inequality is contained in \eqref{PKc-25.D},
so we only need to check the right-pointing inequality 
in \eqref{PKc-25.D2}. To this end, suppose in what follows that 
$\Omega=S_{\theta,r}$ for some $\theta\in(0,\pi)$, $r>0$, and that 
$u=u_{\theta,r}$ solves \eqref{PKc-1} for this particular domain. 
To continue, fix $\theta_0\in(0,\pi/2)$ and assume 
first that $\theta\in(\theta_0,\pi)$. Then
\begin{eqnarray}\label{THfUU}
\int_{S_{\theta,r}}u_{\theta,r}(x)^{-\beta}\,dx
=r^{2(1-\beta)}\int_{S_{\theta,1}}u_{\theta,1}(x)^{-\beta}\,dx
\geq C_\beta\,\theta_0\,r^{2(1-\beta)},
\end{eqnarray}
because the way the first integral scales in the parameter $r$, and
\eqref{PKc-3}. Since in the situation we are currently considering 
$\theta^{1-2\beta}$ behaves like a constant, the desired conclusion 
follows in this case. We are left with considering the case 
when $\theta\in(0,\theta_0)$. In such a scenario, set 
\begin{eqnarray}\label{PKc-5L}
\widetilde{S}_{\theta,r}:=\{(x,y)\in{\mathbb{R}}^2:\,
0<y<x\,\tan\,\theta\,\,\mbox{ and }\,\,0<x<r\}.
\end{eqnarray}
and note that 
\begin{eqnarray}\label{PKc-5Q}
\widetilde{S}_{\theta,r\cos\theta}\subseteq S_{\theta,r}\subseteq
\widetilde{S}_{\theta,r}.
\end{eqnarray}
Next, consider the following barrier function, designed to befit 
the triangular region $\widetilde{S}_{\theta,r}$:
\begin{eqnarray}\label{JanB}
\widetilde{v}_{\theta,r}:\widetilde{S}_{\theta,r}
\longrightarrow{\mathbb{R}},
\quad\widetilde{v}_{\theta,r}(x,y):={\textstyle\frac12}\,y\,
\bigl(x\,\tan\,\theta-y\bigr),
\quad\mbox{for all }\,(x,y)=x+iy\in\widetilde{S}_{\theta,r}.
\end{eqnarray}
Hence, $\widetilde{v}_{\theta,r}$ is continuous in 
$\overline{\widetilde{S}_{\theta,r}}$ and satisfies 
$\Delta\widetilde{v}_{\theta,r}=-1$ in $\widetilde{S}_{\theta,r}$ 
as well as $\widetilde{v}_{\theta,r}\geq 0$ on 
$\partial\widetilde{S}_{\theta,r}$. Consequently, the function 
$\widetilde{v}_{\theta,r}-u$ is harmonic in $S_{\theta,r}(=\Omega)$, 
continuous on its closure, and $\widetilde{v}_{\theta,r}-u=
\widetilde{v}_{\theta,r}\geq 0$ on $\partial S_{\theta,r}(=\partial\Omega)$.
Therefore, by the Maximum Principle, $u\leq\widetilde{v}_{\theta,r}$ 
in $S_{\theta,r}$ which, in turn, gives that 
\begin{eqnarray}\label{THfd}
\int_{S_{\theta,r}}u(x)^{-\beta}\,dx &\geq & 
\int_{S_{\theta,r}}\widetilde{v}_{\theta,r}(x)^{-\beta}\,dx
\geq 
\int_{\widetilde{S}_{\theta,r\cos\theta}}
\widetilde{v}_{\theta,r}(x)^{-\beta}\,dx
\nonumber\\[4pt]
&=& 2^{-\beta}\int^{r\cos\theta}_0\Bigl(\int_0^{x\,\tan\,\theta}
y^{-\beta}(x\,\tan\,\theta-y)^{-\beta}\,dy\Bigr)dx
\nonumber\\[4pt]
&=& 2^{-\beta}(\tan\,\theta)^{1-2\beta}\Bigl(\int_0^{r\cos\theta}
x^{1-2\beta}\,dx\Bigr)\Bigl(\int_0^1 t^{-\beta}(1-t)^{-\beta}\,dt\Bigr),
\nonumber\\[4pt]
&=& C_\beta\,(\tan\,\theta)^{1-2\beta}(\cos\,\theta)^{2-2\beta}
r^{2-2\beta}=C_\beta\,(\cos\,\theta)(\sin\,\theta)^{1-2\beta}
r^{2-2\beta}
\nonumber\\[4pt]
&\approx & \theta^{1-2\beta}r^{2-2\beta},
\quad\mbox{ uniformly for }\,\,\theta\in(0,\theta_0),
\end{eqnarray}
(taking into account \eqref{PKc-5Q} in the second inequality and 
after making the change of variables $y=t\,x\,\tan\,\theta$ in the 
inner integral in the second line). This concludes the justification 
of \eqref{PKc-25.D2} and finishes the proof of the proposition.
\hfill$\Box$
\vskip 0.08in

We are now in a position to formulate the first main result in this 
section. This shows that, whenever $\Omega$ is a polygon, \eqref{PKc-25} 
holds for every $\beta<1$ which is remarkable since, as opposed to the 
situation discussed in Corollary~\ref{anLL}, this time $\partial\Omega$ 
is far from being regular.

\begin{theorem}\label{Te-3A}
Assume that $\Omega$ is a polygon in ${\mathbb{R}}^2$ and that 
$u$ is the solution of \eqref{PKc-1}. Then, for every 
$\beta\in(0,1)$, 
\begin{eqnarray}\label{PDsa-3}
\int_{\Omega}u(x)^{-\beta}\,dx\leq C(\Omega,\beta)<+\infty.
\end{eqnarray}
\end{theorem}

\begin{proof}
Let $\{P_1,...,P_N\}$ be the vertices of the polygon $\Omega$ and, 
for each $i\in\{1,...,N\}$, denote by $\theta_i\in(0,\pi)$ the half-measure
of the angle corresponding to $P_i$, and by $L'_i$ and $L''_i$ the 
lengths of the two sides of $\Omega$ emerging from $P_i$. 
Also, for each $i\in\{1,...,N\}$, introduce 
\begin{eqnarray}\label{GBab}
r_i:=\min\,\Bigl\{L'_i\,,\,L''_i\,,\,(L'_i)^{1/(2-2\beta)}
\theta_i^{(2\beta-1)/(2-2\beta)}\,,\,
(L''_i)^{1/(2-2\beta)}\theta_i^{(2\beta-1)/(2-2\beta)}\Bigr\}.
\end{eqnarray}
Parenthetically, we wish to note that 
\begin{eqnarray}\label{GBac}
\beta=1/2\Longrightarrow r_i=\min\,\bigl\{L'_i,L''_i\bigr\}.
\end{eqnarray}

Writing a formula similar to \eqref{PKc-25.D} at each vertex 
(note that the problem \eqref{PKc-1} transforms naturally under 
rigid motions of the plane) and summing up all contributions obtained 
from integrating $u^{-\beta}$ near each vertex gives
\begin{eqnarray}\label{PKc-26G}
\sum_{i=1}^N\int_{\Omega\cap B(P_i,r_i)}u(x)^{-\beta}\,dx
&\leq & C_\beta\,\sum_{i=1}^N \theta_i^{1-2\beta}r_i^{2(1-\beta)}
\nonumber\\[4pt]
&\leq & C_\beta\,\sum_{i=1}^N\min\,\{L'_i,L''_i\}
\leq C_\beta\,{\mathcal{H}}^1(\partial\Omega),
\end{eqnarray}
with a finite constant $C_\beta>0$ which depends only on $\beta$. 

Having estimated the contribution from the vertices, construct now 
a $C^1$ domain $\Omega_\ast\subseteq\Omega$ by rounding off 
each vertex $P_i$ with a suitably small circular arc contained within 
$B(P_i,r_i/2)$. This ensures that $\Omega_\beta$ has the property that 
\begin{eqnarray}\label{BVC}
\Omega_\ast\cup\Bigl(\bigcup_{1\leq i\leq N}
\bigl(\Omega\cap B(P_i,r_i)\bigr)\Bigr)=\Omega.
\end{eqnarray}
Next, consider $u_\ast$ such that 
\begin{eqnarray}\label{PKc-1XX}
\left\{
\begin{array}{l}
\Delta u_\ast=-1\,\,\mbox{ in }\,\,\Omega_\ast,
\\[4pt]
u_\ast=0\,\,\mbox{ on }\,\,\partial\Omega_\ast,
\\[4pt]
u_\ast\in C^0(\overline{\Omega_\ast}).
\end{array}
\right.
\end{eqnarray}
Since $u-u_\ast$ is a continuous function in $\overline{\Omega_\ast}$
which satisfies $\Delta(u-u_\ast)=0$ in $\Omega_\beta$ and 
$u-u_\ast=u\geq 0$ on $\partial\Omega_\ast\subseteq\overline{\Omega}$, 
it follows from the Maximum Principle that $u\geq u_\ast$ in $\Omega_\ast$. 
Thus, for every $\beta\in(0,1)$, 
\begin{eqnarray}\label{FD6tf}
\int_{\Omega_\ast}u(x)^{-\beta}\,dx\leq 
\int_{\Omega_\ast}u_\ast(x)^{-\beta}\,dx\leq C(\Omega_\ast,\beta)<+\infty,
\end{eqnarray}
by virtue of Corollary~\ref{anLL}. Now, \eqref{PDsa-3} follows readily 
from \eqref{PKc-26G}, \eqref{BVC} and \eqref{FD6tf}. 
\end{proof}

In the class of convex polygons in ${\mathbb{R}}^2$ it is possible
to further clarify the nature of the constant $C(\Omega,\beta)$ in 
\eqref{PDsa-3}. This is done in Theorem~\ref{Te-3A.u}, stated later
in this section. As a preamble, a key technical result used in the 
proof of this theorem is isolated in the proposition below. 

\begin{proposition}\label{TL-3.ii}
Assume that $\Omega$ is a bounded domain in ${\mathbb{R}}^2$ for which there 
exist $r>0$ and $\theta\in(0,\pi/2)$ with the property that 
\begin{eqnarray}\label{PKc-18.ii}
\Omega\cap B(0,r)=S_{\theta,r}\quad\mbox{ and }\quad 
B\bigl((r(\cos\theta)^{-1},0),r\tan\theta\bigr)\subseteq\Omega.
\end{eqnarray}
As usual, let $u$ be the function defined by \eqref{PKc-1}. 
Then 
\begin{eqnarray}\label{PKc-25.ii}
0<\beta<\min\Bigl\{1,{\textstyle{\frac{4\theta}{\pi}}}\Bigr\}\Longrightarrow
\int_{S_{\theta,r}}u(x)^{-\beta}\,dx\leq C_\beta\,
\left(\frac{\theta^{2(1-\beta)}}{4\theta-\pi\beta}\right)
r^{2(1-\beta)}(\cos\theta)^\beta,
\end{eqnarray}
where $C_\beta>0$ is the finite constant which depends only on $\beta$.

Furthermore, retaining \eqref{PKc-18.ii} it follows that for every 
$\beta\in(0,1)$ there holds
\begin{eqnarray}\label{PKc-25.iii}
\int_{S_{\theta,r}}u(x)^{-\beta}\,dx\leq 
C_\beta\,\theta^{1-2\beta}\,r^{2(1-\beta)}(\cos\theta)^\beta.
\end{eqnarray}
\end{proposition}

\begin{proof}
Consider the function defined by 
\begin{eqnarray}\label{Yh.i-1}
v(x):={\textstyle\frac14}\Bigl[r^2\tan^2\theta
-\bigl(x_1-r(\cos\theta)^{-1}\bigr)^2-x_2^2\Bigr],\qquad x=(x_1,x_2).
\end{eqnarray}
In polar coordinates $x=(x_1,x_2)=(\rho\cos\omega,\rho\sin\omega)$ this
takes the form 
\begin{eqnarray}\label{Yh.i-2}
v(\rho,\omega)={\textstyle\frac14}\Bigl[2r\rho\cos\omega(\cos\theta)^{-1}
-r^2-\rho^2\Bigr].
\end{eqnarray}
Let us also define the harmonic function 
\begin{eqnarray}\label{Yh.i-3}
w(\rho,\omega):=\rho^{\pi/(2\theta)}
\cos\,\Bigl(\frac{\pi\omega}{2\theta}\Bigr),
\qquad \rho>0,\,\,|\omega|<\theta.
\end{eqnarray}
Then, $w$ vanishes on the straight sides of $\partial S_{\theta,r}$ 
(i.e., for $\omega=\pm\theta$), while on the rounded portion of the boundary 
of $\partial S_{\theta,r}$ (i.e., the arc described by $\rho=r$ and 
$|\omega|<\theta$) we have 
\begin{eqnarray}\label{Yh.i-4}
\frac{w(r,\omega)}{v(r,\omega)}
=2(\cos\theta)\,r^{\pi/(2\theta)-2}
\frac{\cos\,\Bigl(\frac{\pi\omega}{2\theta}\Bigr)}{\cos\omega-\cos\theta},
\qquad |\omega|<\theta.
\end{eqnarray}
Note that 
\begin{eqnarray}\label{Yh.i-5}
\sup_{|\omega|<\theta}\left(
\frac{\cos\,\Bigl(\frac{\pi\omega}{2\theta}\Bigr)}{\cos\omega-\cos\theta}
\right)
=\sup_{0<\omega<\theta}\left(
\frac{\cos\,\Bigl(\frac{\pi\omega}{2\theta}\Bigr)}{\cos\omega-\cos\theta}
\right)
\end{eqnarray}
and, for each $\omega\in(0,\theta)$, 
\begin{eqnarray}\label{Yh.i-6}
\frac{\cos\,\Bigl(\frac{\pi\omega}{2\theta}\Bigr)}{\cos\omega-\cos\theta}
&=&\frac{\sin\,\Bigl(\frac{\pi}{2}-\frac{\pi\omega}{2\theta}\Bigr)}
{\cos\omega-\cos\theta}
=\frac{\sin\,\Bigl(\frac{\pi(\theta-\omega)}{2\theta}\Bigr)}
{\cos\omega-\cos\theta}
\nonumber\\[4pt]
&=& 
\frac{\pi}{2\theta}\left(
\frac{\sin\,\Bigl(\frac{\pi(\theta-\omega)}{2\theta}\Bigr)}
{\frac{\pi(\theta-\omega)}{2\theta}}\right)
\frac{\theta-\omega}{\cos\omega-\cos\theta}.
\end{eqnarray}
Since $\pi(\theta-\omega)/(2\theta)\in(0,\pi/2)$ whenever 
$\omega\in(0,\theta)$, it follows that the fraction in parentheses 
in the right-most expression in \eqref{Yh.i-6} is $\leq 1$ for each 
$\omega\in(0,\theta)$. Also, elementary calculus shows that there 
exists a universal constant $c>0$ such that 
$(\theta-\omega)/(\cos\omega-\cos\theta)\leq c/\theta$ if 
$0<\omega<\theta$ (recall that $\theta\in (0,\pi/2)$). Consequently,  
\begin{eqnarray}\label{Yh.i-7}
\frac{w(r,\omega)}{v(r,\omega)}
\leq c\,\theta^{-2}(\cos\theta)\,r^{\pi/(2\theta)-2},
\qquad\mbox{for every }\,\,\omega\in(-\theta,\theta).
\end{eqnarray}

Next, given that by design 
\begin{eqnarray}\label{Yh.i-8}
v=0\,\,\mbox{ on }\,\,\partial
B\bigl((r(\cos\theta)^{-1},0),r\tan\theta\bigr),\quad
\Delta v=-1\,\,\mbox{ in }\,\,
B\bigl((r(\cos\theta)^{-1},0),r\tan\theta\bigr),
\end{eqnarray}
and that, by assumption, 
$B\bigl((r(\cos\theta)^{-1},0),r\tan\theta\bigr)\subseteq\Omega$, 
the Maximum Principle ensures that 
\begin{eqnarray}\label{Yh.i-9}
v\leq u\,\,\mbox{ in }\,\,B\bigl((r(\cos\theta)^{-1},0),r\tan\theta\bigr).
\end{eqnarray}
In particular, 
\begin{eqnarray}\label{Yh.i-10}
v(r,\omega)\leq u(r,\omega)\,\,\mbox{ for all }\,\,\omega\in(-\theta,\theta).
\end{eqnarray}
From this and \eqref{Yh.i-7} we may therefore conclude that 
\begin{eqnarray}\label{Yh.i-11}
c^{-1}\,\theta^{2}(\cos\theta)^{-1}\,r^{2-\pi/(2\theta)}w(r,\omega)
\leq u(r,\omega)\qquad\mbox{for every }\,\,\omega\in(-\theta,\theta).
\end{eqnarray}
Granted this, as well as the properties of $w$ recorded just after 
\eqref{Yh.i-3}, the Maximum Principle applies again and yields that 
\begin{eqnarray}\label{Yh.i-12}
c^{-1}\,\theta^{2}(\cos\theta)^{-1}\,r^{2-\pi/(2\theta)}w(x)
\leq u(x)\qquad\mbox{for every }\,\,x\in S_{\theta,r}.
\end{eqnarray}
Hence, for every $\beta\in(0,4\theta/\pi)$, a familiar (by now) computation 
gives that 
\begin{eqnarray}\label{Yh.i-13}
\int_{S_{\theta,r}}u(x)^{-\beta}\,dx
&\leq & c^{\beta}\theta^{-2\beta}(\cos\theta)^{\beta}\,
r^{(\pi/(2\theta)-2)\beta}\int_{S_{\theta,r}}w(x)^{-\beta}\,dx
\nonumber\\[4pt]
&=& c^{\beta}\theta^{-2\beta}(\cos\theta)^{\beta}\,r^{(\pi/(2\theta)-2)\beta}
\,{\mathfrak{B}}\Bigl(\frac{1}{2};\frac{1-\beta}{2}\Bigr)
\left(\frac{4\theta^2}{\pi(4\theta-\pi\beta)}\right)r^{2-\pi\beta/(2\theta)}
\nonumber\\[4pt]
&=& C_\beta \left(\frac{\theta^{2(1-\beta)}}{4\theta-\pi\beta}\right)
r^{2(1-\beta)}(\cos\theta)^\beta,
\end{eqnarray}
proving \eqref{PKc-25.ii}. 

Finally, when $\theta\in(\pi/3,\pi/2)$, it follows from \eqref{PKc-25.ii}
that for every $\beta\in(0,1)$ we have 
\begin{eqnarray}\label{P-25.sz}
\int_{S_{\theta,r}}u(x)^{-\beta}\,dx\leq 
C_\beta\,r^{2(1-\beta)}(\cos\theta)^\beta, 
\end{eqnarray}
which further implies \eqref{PKc-25.iii} in the case we are considering.
When $\theta\in(0,\pi/3)$, then \eqref{PKc-25.iii} is a direct consequence 
of \eqref{PKc-25.D}. 
\end{proof}

Continuing the buildup to Theorem~\ref{Te-3A.u}, we now make 
several definitions and comment on their significance and how 
they interrelate. 

\begin{definition}\label{gaTf}
The eccentricity of an open, bounded convex set $\Omega$ in ${\mathbb{R}}^n$ 
is defined as 
\begin{eqnarray}\label{jaTg}
{\rm ecc}\,(\Omega):=\frac{\inf\,\{R_1>0:\,\exists\,x\in{\mathbb{R}}^n
\,\,\mbox{ such that }\,\,\Omega\subseteq B(x,R_1)\}}
{\sup\,\{R_2>0:\,\exists\,x\in{\mathbb{R}}^n
\,\,\mbox{ such that }\,\,B(x,R_2)\subseteq\Omega\}}.
\end{eqnarray}
\end{definition}

It follows that 
\begin{eqnarray}\label{Rtagb}
\begin{array}{l}
{\rm ecc}\,(\Omega)\,\,\mbox{ controls both the NTA constants of $\Omega$}
\\[4pt]
\mbox{as well as the Ahlfors character of $\partial\Omega$},
\end{array}
\end{eqnarray}
uniformly in the class of open, bounded and convex subsets 
$\Omega$ of ${\mathbb{R}}^n$. Furthermore, there exists a dimensional 
constant $c_n$ with the property that for every open, bounded convex set 
$\Omega\subseteq{\mathbb{R}}^n$ we have 
\begin{eqnarray}\label{Rtag-2}
{\rm diam}\,(\Omega)\leq c_n\,{\rm ecc}\,(\Omega)\,|\Omega|^{1/n}.
\end{eqnarray}

\begin{definition}\label{gann}
Let $\Omega$ be a convex polygon in ${\mathbb{R}}^2$. Call $R>0$ an admissible 
radius for $\Omega$ provided for each side of $\Omega$ there exists a 
ball of radius $R$ contained in $\Omega$ which is tangent to that side. 
Then define the maximal admissible radius of $\Omega$ as 
\begin{eqnarray}\label{jann}
R_{\Omega}:=\sup\,\{R>0:\,R\,\,
\mbox{ is an admissible radius for $\Omega$}\},
\end{eqnarray}
and set 
\begin{eqnarray}\label{jann.2}
\Omega^{\#}:=\bigcup\limits_{x\in\Omega,\,\delta_{\Omega}(x)>R_{\Omega}}
B(x,R_{\Omega}).
\end{eqnarray}
\end{definition}

Straight from definitions it can be seen that 
\begin{eqnarray}\label{hav.F}
\Omega\subseteq{\mathbb{R}}^2\,\,\mbox{ convex polygon}\,\Longrightarrow
\Omega^{\#}\,\,\mbox{ satisfies an inner ball condition of radius }\,\,
R_{\Omega}.
\end{eqnarray}
Also, it is not too difficult to show that 
\begin{eqnarray}\label{hav.F2}
\begin{array}{l}
\Omega\subseteq{\mathbb{R}}^2\,\,\mbox{ convex polygon}\,\Longrightarrow
\Omega^{\#}\,\,\mbox{ is a $C^{1,1}$ convex domain}
\\[4pt]
\hskip 1.92in
\mbox{ satisfying ${\rm ecc}\,(\Omega^{\#})\leq{\rm ecc}\,(\Omega)$}.
\end{array}
\end{eqnarray}
Shortly we will also need the readily verified claim that 
\begin{eqnarray}\label{hav.F3}
\begin{array}{l}
\Omega\subseteq{\mathbb{R}}^2\,\,\mbox{ convex polygon}\,\Longrightarrow
\mbox{ the angles of $\Omega$ are $\geq 2\theta_{\ast}$, where } 
\\[4pt]
\hskip 1.99in
\mbox{$\theta_{\ast}\in(0,\pi/2)$ depends only on ${\rm ecc}\,(\Omega)$}.
\end{array}
\end{eqnarray}

After this prelude, we are now prepared to state and prove a refined 
version of \eqref{PDsa-3} in the class of convex polygons in ${\mathbb{R}}^2$. 

\begin{theorem}\label{Te-3A.u}
Assume that $\Omega$ is a convex polygon in ${\mathbb{R}}^2$. 
Then for every $\beta\in(0,1)$ the solution $u$ of \eqref{PKc-1} satisfies 
\begin{eqnarray}\label{Ah-nb.i}
\int_{\Omega}u(x)^{-\beta}\,dx \leq C({\rm ecc}\,(\Omega),\beta)
\Bigl(\frac{{\rm diam}\,(\Omega)}{R_{\Omega}}\Bigr)^{m\beta+2}
|\Omega|^{1-\beta}, 
\end{eqnarray}
where $m>0$ depends only on ${\rm ecc}\,(\Omega)$.
\end{theorem}

\begin{proof}
Denote by $\{P_1,...,P_N\}$ the vertices of the polygon $\Omega$ and, for 
each $i\in\{1,...,N\}$, let $\theta_i\in(0,\pi/2)$ be the half-measure 
of the angle corresponding to $P_i$. Set 
\begin{eqnarray}\label{kam-xx}
r_i:=R_{\Omega}\cos\theta_i,\qquad 1\leq i\leq N, 
\end{eqnarray}
and fix $\beta\in(0,1)$. Also, let $\theta_{\ast}\in(0,\pi/2)$
be as in \eqref{hav.F3}. Then, thanks to \eqref{PKc-25.iii}, 
for each $i\in\{1,...,N\}$ we have
\begin{eqnarray}\label{kam-1}
\int_{\Omega\cap B(P_i,r_i)}u(x)^{-\beta}\,dx
\leq C(\theta_{\ast},\beta)\,r_i^{2(1-\beta)}(\cos\theta_i)^\beta
=C(\theta_{\ast},\beta)\,R_{\Omega}^{-\beta}r_i^{2-\beta}.
\end{eqnarray}
Consequently, 
\begin{eqnarray}\label{kam-2}
\sum_{i=1}^N\int_{\Omega\cap B(P_i,r_i)}u(x)^{-\beta}\,dx
&\leq & C(\theta_{\ast},\beta)\,R_{\Omega}^{-\beta}\sum_{i=1}^Nr_i^{2-\beta}
\nonumber\\[4pt]
&\leq & C(\theta_{\ast},\beta)\,R_{\Omega}^{-\beta}\Bigl(\sum_{i=1}^Nr_i
\Bigr)^{2-\beta}, 
\end{eqnarray}
since $2-\beta>1$. Given that $\sum_{i=1}^Nr_i$ is controlled by the 
perimeter of $\Omega$, this yields 
\begin{eqnarray}\label{kam-3}
\sum_{i=1}^N\int_{\Omega\cap B(P_i,r_i)}u(x)^{-\beta}\,dx
&\leq & C(\theta_{\ast},\beta)\,R_{\Omega}^{-\beta}
\bigl[{\mathcal{H}}^1(\partial\Omega)\bigr]^{2-\beta}
\nonumber\\[4pt]
&\leq & C(\theta_{\ast},\beta)\,R_{\Omega}^{-\beta}
({\rm diam}\,(\Omega))^{2-\beta}
\nonumber\\[4pt]
&=& C(\theta_{\ast},\beta)
\Bigl(\frac{{\rm diam}\,(\Omega)}{R_{\Omega}}\Bigr)^{\beta}
({\rm diam}\,(\Omega))^{2(1-\beta)}.
\end{eqnarray}
In light of \eqref{Ah-nb.i} (and keeping \eqref{hav.F3} in mind), 
this bound suits our purposes. 

To continue, we note that 
\begin{eqnarray}\label{kam-4}
\Omega^{\#}\cup\Bigl(\bigcup_{i=1}^N(\Omega\cap B(P_i,r_i))\Bigr)=\Omega, 
\end{eqnarray}
and observe that, thanks to \eqref{hav.F} and \eqref{Ah-nb}, 
\begin{eqnarray}\label{kam-5}
\int_{\Omega^{\#}}u(x)^{-\beta}\,dx &\leq & C_{\Omega^{\#}}(\beta)
\Bigl(\frac{{\rm diam}\,(\Omega^{\#})}{R_{\Omega}}\Bigr)^{m\beta+2}
|\Omega^{\#}|^{1-\beta}
\nonumber\\[4pt]
&\leq & C_{\Omega^{\#}}(\beta)
\Bigl(\frac{{\rm diam}\,(\Omega)}{R_{\Omega}}\Bigr)^{m\beta+2}
({\rm diam}\,(\Omega))^{2(1-\beta)},
\end{eqnarray}
where $m>0$ depends only on the NTA constants of $\Omega^{\#}$, 
and $C_{\Omega^{\#}}(\beta)>0$ is a finite constant which depends 
only on the Ahlfors character of $\partial\Omega^{\#}$ and $\beta$. 
Hence, by \eqref{Rtagb} and \eqref{hav.F2}, $C_{\Omega^{\#}}(\beta)$ 
can be controlled in terms of ${\rm ecc}\,(\Omega)$ and $\beta$. 
Estimate \eqref{Ah-nb.i} now follows from this observation, 
\eqref{kam-3}, \eqref{kam-4}, \eqref{kam-5} and \eqref{Rtag-2}.
\end{proof}

\begin{remark}\label{Rem-F2}
In regard to \eqref{Ah-nb.i}, it should be pointed out that, 
in the class of convex polygons in ${\mathbb{R}}^2$, the maximal 
admissible radius cannot be controlled in terms of the diameter and 
the eccentricity. A simple example is as follows. Let $\Omega$ 
be the triangle whose vertices have coordinates $(-1,0)$, $(1,0)$, 
$(0,1)$ and, for each $j\geq 2$, consider the convex quadrilateral
$\Omega_j:=\{(x,y)\in\Omega:\,y<1-1/j\}$. It is then clear that 
while ${\rm ecc}\,(\Omega_j)$ and ${\rm diam}\,(\Omega_j)$ stay bounded,
$R_{\Omega_j}\to 0$ as $j\to\infty$.  
\end{remark}

We conclude this section by giving an asymptotic formula for the 
$\beta$-integral of a regular polygon, as the number of vertices increases. 
This augments earlier estimates in \eqref{thg.6}-\eqref{thg.7}.  

\begin{proposition}\label{taTH}
For each $N\in{\mathbb{N}}$, $N\geq 3$, let $\Omega_N\subseteq{\mathbb{R}}^2$ 
denote the regular polygon with $N$ sides, circumscribed by $B(0,1)$. 
Denote by $u_N$ the solution $u$ of \eqref{PKc-1} when $\Omega=\Omega_N$. 
Then for each $\beta\in(0,1)$ the following asymptotic formula holds
\begin{eqnarray}\label{HGba-1}
\int_{\Omega_N}u_N(x)^{-\beta}\,dx=\frac{4^\beta\pi}{1-\beta}
+{\mathcal{O}}(N^{\beta-1})\qquad\mbox{as }\,\,N\to\infty.
\end{eqnarray}
\end{proposition}

\begin{proof}
We specialise part of the proof of Theorem~\ref{Te-3A.u} to the 
present case. In the current setting, using notation introduced on that 
occasion, we have: 
\begin{eqnarray}\label{HGba-2}
R_{\Omega_N} = 1 - o(1),\quad
\theta_i=\frac{\pi}{2}-\frac{\pi}{N},\quad
r_i=R_{\Omega_N}\cos\theta_i=R_{\Omega_N}\sin(\pi/N),\quad 1\leq i\leq N. 
\end{eqnarray}
Also, as before, we let $P_1,...,P_N$ be the vertices of $\Omega_N$. 
Hence, the first inequality in \eqref{kam-2} gives 
\begin{eqnarray}\label{HGba-3}
\sum_{i=1}^N\int_{\Omega_N\cap B(P_i,r_i)}u_N(x)^{-\beta}\,dx
\leq C_\beta R_{\Omega_N}^{2(1-\beta)}\,\sum_{i=1}^N\bigl(\sin(\pi/N)\bigr)^{2-\beta}
\leq C_\beta\,N^{\beta-1}.
\end{eqnarray}
Since the Maximum Principle and \eqref{PKc-1.A} imply 
$u(x)\geq\frac14 (1-|x|^2)$ for all $x\in B(0,1)$, we therefore 
obtain the asymptotic estimate 
\begin{eqnarray}\label{HGba-4}
\int_{\Omega_N}u_N(x)^{-\beta}\,dx &\leq &  
\sum_{i=1}^N\int_{\Omega_N\cap B(P_i,r_i)}u_N(x)^{-\beta}\,dx
+\int_{B(0,1)}u_N(x)^{-\beta}\,dx 
\nonumber\\[4pt]
&\leq & 4^{\beta}\int_{B(0,1)}\frac{dx}{(1-|x|^2)^{\beta}}
+{\mathcal{O}}(N^{\beta-1})
\nonumber\\[4pt]
&=& \frac{4^\beta\pi}{1-\beta}+{\mathcal{O}}(N^{\beta-1}).
\end{eqnarray}
On the other hand, \eqref{PKc-1.A} and the Maximum Principle give
\begin{eqnarray}\label{HGba-5}
u_N(x)\leq{\textstyle\frac14}\bigl((\cos(\pi/N))^{-2}-|x|^2)\quad
\mbox{ for all }\,\,x\in\Omega_N, 
\end{eqnarray}
which then forces 
\begin{eqnarray}\label{HGba-6}
\int_{\Omega_N}u_N(x)^{-\beta}\,dx &\geq & 
4^{\beta}\int_{B(0,1)}\frac{dx}{\bigl((\cos(\pi/N))^{-2}-|x|^2
\bigr)^{\beta}}
=4^{\beta}\pi\int_0^1\frac{dt}{\bigl((\cos(\pi/N))^{-2}-t\bigr)^{\beta}}
\nonumber\\[4pt]
&=& \frac{4^{\beta}\pi}{1-\beta}\Bigl((\cos(\pi/N))^{-2(1-\beta)}
-\bigl((\cos(\pi/N))^{-2}-1\bigr)^{1-\beta}\Bigr)
\nonumber\\[4pt]
&=& \frac{4^{\beta}\pi}{1-\beta}(\cos(\pi/N))^{-2(1-\beta)}
\Bigl(1-\bigl(\sin(\pi/N)\bigr)^{2(1-\beta)}\Bigr)
\nonumber\\[4pt]
&=& \frac{4^{\beta}\pi}{1-\beta}
\Bigl(1-{\mathcal{O}}(N^{2(\beta-1)})\Bigr). 
\end{eqnarray}
Now, \eqref{HGba-1} follows from \eqref{HGba-4} and \eqref{HGba-6}.
\end{proof}

\section{Piecewise smooth domains with conical singularities
in ${\mathbb{R}}^n$}
\label{sect:5}
\setcounter{equation}{0}

In this section we shall work in the general $n$-dimensional case, 
$n\geq 2$. Throughout, we retain notation introduced in \S\,\ref{sect:3}.
Given $r>0$ and an open connected subset ${\mathfrak{G}}$ of $S^{n-1}$ 
with a sufficiently regular boundary (relative to $S^{n-1}$), define 
the truncated cone
\begin{eqnarray}\label{Mc-T1}
S_{{\mathfrak{G}},r}:=B(0,r)\cap\Gamma_{\mathfrak{G}}=
\{(\rho,\omega):\,0<\rho<r,\,\,\omega\in {\mathfrak{G}}\},  
\end{eqnarray}
where $(\rho,\omega)\in(0,\infty)\times S^{n-1}$ are 
the standard polar coordinates in ${\mathbb{R}}^n$. 
Associated with this truncated cone, consider the barrier function 
$v_{{\mathfrak{G}},r}:S_{{\mathfrak{G}},r}\to{\mathbb{R}}$ which, 
in polar coordinates is given by
\begin{eqnarray}\label{PKc-7H.G}
v_{{\mathfrak{G}},r}(\rho,\omega):=
\left\{
\begin{array}{l}
\frac{r^2}{\Lambda_{\mathfrak{G}}-2n}\Bigl[\Bigl(\frac{\rho}{r}\Bigr)^2
-\Bigl(\frac{\rho}{r}\Bigr)^{\alpha_{\mathfrak{G}}}\Bigr]
\phi_{{\mathfrak{G}}}(\omega)\,\,\,\,
\mbox{ if $\Lambda_{\mathfrak{G}}\not=2n$},
\\[16pt]
\frac{\rho^2}{n+2}
\log\Big(\frac{r}{\rho}\Bigr)\phi_{\mathfrak{G}}\,(\omega)
\hskip 0.81in
\mbox{ if $\Lambda_{\mathfrak{G}}=2n$},
\end{array}
\right.
\end{eqnarray}
for each $\omega\in {\mathfrak{G}}$ and $\rho\in(0,r)$. 
Given that, by \eqref{critic.G}, we have 
\begin{eqnarray}\label{BGfa}
\Lambda_{\mathfrak{G}}-2n=(\alpha_{\mathfrak{G}}-2)(\alpha_{\mathfrak{G}}+n),
\end{eqnarray}
it is worth noting that 
\begin{eqnarray}\label{Gva-Ji}
\Lambda_{{\mathfrak{G}}}=2n\Longleftrightarrow\alpha_{\mathfrak{G}}=2.
\end{eqnarray}
In particular, the formula for $v_{{\mathfrak{G}},r}(\rho,\omega)$ in 
the second line of \eqref{PKc-7H.G} is the limiting case of the formula 
for $v_{{\mathfrak{G}},r}(\rho,\omega)$ in the first line of \eqref{PKc-7H.G}
as $\Lambda_{\mathfrak{G}}$ becomes $2n$. Much as before, in axially 
symmetric case, i.e., for ${\mathfrak{G}}=S^{n-1}\cap\Gamma_\theta$ for some 
$\theta\in(0,\pi)$, we agree to abbreviate $S_{S^{n-1}\cap\Gamma_\theta,r}$ 
and $v_{S^{n-1}\cap\Gamma_\theta,r}$ by $S_{\theta,r}$ and $v_{\theta,r}$, 
respectively. In this scenario, we therefore have 
\begin{eqnarray}\label{PKc-7H}
v_{\theta,r}(\rho,\omega)=
\left\{
\begin{array}{l}
\frac{r^2}{\Lambda_\theta-2n}\Bigl[\Bigl(\frac{\rho}{r}\Bigr)^2
-\Bigl(\frac{\rho}{r}\Bigr)^{\alpha_\theta}\Bigr]\phi_{\theta}(\omega),
\hskip 0.51in
\mbox{ if $\theta\not=\theta_n$},
\\[16pt]
\frac{\rho^2}{n+2}
\log\Big(\frac{r}{\rho}\Bigr)\phi_{\theta_n}\,(\omega),
\,\,\mbox{ corresponding to $\theta=\theta_n$},
\end{array}
\right.
\end{eqnarray}
where $\theta_n\in(0,\pi)$ is the unique angle for which
$\alpha_{\theta_n}=2$. 
Note that since the assignment $\theta\mapsto\alpha_{\theta}$ 
is strictly decreasing (cf.~\eqref{alphamap}) and since $\alpha_{\pi/2}=1$
and $\alpha_{\theta}\nearrow+\infty$ as $\theta\searrow 0$ 
(cf.~\eqref{alphamap2}), there exists precisely one angle  
$\theta_n\in(0,\pi/2)$ for which $\alpha_{\theta_n}=2$. In fact, from
\eqref{limn2F} we know that 
\begin{eqnarray}\label{Gva-J}
\theta_n={\rm arccos}\,(1/\sqrt{n}),\qquad n\geq 2, 
\end{eqnarray}
so that, in particular,
\begin{eqnarray}\label{BGT/a}
\theta_n=\left\{
\begin{array}{l}
\frac{\pi}{4}\,\,\mbox{ when }\,\,n=2,
\\[6pt]
\frac{\pi}{3}\,\,\mbox{ when }\,\,n=4,
\end{array}
\right.
\qquad\mbox{ and }\,\,\,\theta_n\nearrow\frac{\pi}{2}
\,\,\mbox{ as }\,\,\,n\to\infty. 
\end{eqnarray}
This discussion shows that $\Lambda_\theta\not=2n$
if $\theta\not=\theta_n$ and, hence, $v_{\theta,r}(\rho,\omega)$ 
is well-defined for every $\theta\in(0,\pi)$ and  
$v_{\theta_n,r}(\rho,\omega)$ is the limit of 
$v_{\theta,r}(\rho,\omega)$ as $\theta\to\theta_n$.

\begin{lemma}\label{TL-1H}
Assume that ${\mathfrak{G}}$ is an open, connected subset of $S^{n-1}$ whose
relative boundary is a submanifold of class $C^{1,\alpha}$, 
for some $\alpha\in(0,1)$, and of codimension one in $S^{n-1}$. 
Then for every $\beta\in(0,1)$ and $r>0$, 
the barrier function $v_{{\mathfrak{G}},r}$ from \eqref{PKc-7H.G} satisfies
\begin{eqnarray}\label{PKc-9H}
\int_{S_{{\mathfrak{G}},r}}v_{{\mathfrak{G}},r}(x)^{-\beta}\,dx
=c_n({\mathfrak{G}},\beta)\,r^{n-2\beta}
\Bigl(\int_{{\mathfrak{G}}}\phi_{{\mathfrak{G}}}(\omega)^{-\beta}
\,d\omega\Bigr)<+\infty,
\end{eqnarray}
where
\begin{eqnarray}\label{P-9J}
c_n({\mathfrak{G}},\beta):=\left\{
\begin{array}{l}
\frac{(\Lambda_{\mathfrak{G}}-2n)^{\beta}}{\alpha_{\mathfrak{G}}-2}
\mathfrak{B}\Bigl(\frac{n-2\beta}{\alpha_{\mathfrak{G}}-2};1-\beta\Bigr),
\hskip 0.36in
\mbox{ if }\,\,2<\alpha_{\mathfrak{G}}<+\infty,
\\[12pt]
(n-2\beta)^{\beta-1}(n+2)^\beta\Gamma(1-\beta),
\,\,\mbox{ if }\,\,\alpha_{\mathfrak{G}}=2,
\\[12pt]
\frac{(2n-\Lambda_{\mathfrak{G}})^{\beta}}{2-\alpha_{\mathfrak{G}}}
\mathfrak{B}\Bigl(\frac{n-2\beta}{2-\alpha_{\mathfrak{G}}};1-\beta\Bigr)
\hskip 0.36in
\mbox{ if }\,\,0<\alpha_{\mathfrak{G}}<2.
\end{array}
\right.
\end{eqnarray}
\end{lemma}
\begin{proof}
The proof largely parallels that of Lemma~\ref{TL-1}. We include 
it primarily to indicate how the right-hand side of \eqref{PKc-9H} 
shapes up. Fix $\beta\in(0,1)$, let $r>0$ be arbitrary and first 
assume that $\alpha_{\mathfrak{G}}\in(2,\infty)$. 
This forces $\Lambda_{\mathfrak{G}}>2n$ and we have 
\begin{eqnarray}\label{DSw-1H}
\int_{S_{{\mathfrak{G}},r}}v_{{\mathfrak{G}},r}(x)^{-\beta}\,dx
&=& r^{-2\beta}(\Lambda_{\mathfrak{G}}-2n)^{\beta}
\int_{S_{{\mathfrak{G}},r}}\Bigl[\Bigl(\frac{\rho}{r}\Bigr)^2
-\Bigl(\frac{\rho}{r}\Bigr)^{\alpha_{\mathfrak{G}}}\Bigr]^{-\beta}
\phi_{{\mathfrak{G}}}(\omega)^{-\beta}\,dx
\nonumber\\[4pt]
&=& r^{-2\beta}(\Lambda_{\mathfrak{G}}-2n)^{\beta}I\cdot II,
\end{eqnarray}
where 
\begin{eqnarray}\label{DSw-2H}
I:=\int_0^r\Bigl[\Bigl(\frac{\rho}{r}\Bigr)^2
-\Bigl(\frac{\rho}{r}\Bigr)^{\alpha_\theta}\Bigr]^{-\beta}
\rho^{n-1}\,d\rho\quad\mbox{ and }\quad
II:=\int_{{\mathfrak{G}}}\phi_{{\mathfrak{G}}}(\omega)^{-\beta}\,d\omega.
\end{eqnarray}
As in the past, we make two changes of variables, first letting
$t:=\rho/r$, then replacing $t^{\varepsilon}$ by $s$ where, this time,  
we set $\varepsilon:=\alpha_{\mathfrak{G}}-2>0$. This yields 
\begin{eqnarray}\label{DSw-3H}
I &=&\varepsilon^{-1}r^n\int_0^1 
s^{(n-2\beta)/\varepsilon-1}(1-s)^{-\beta}\,ds=\frac{r^n}{\varepsilon}
{\mathfrak{B}}\Bigl(\frac{n-2\beta}{\varepsilon};1-\beta\Bigr)
\nonumber\\[4pt]
&=& \frac{r^n}{\alpha_{\mathfrak{G}}-2}
\mathfrak{B}\Bigl(\frac{n-2\beta}{\alpha_{\mathfrak{G}}-2};1-\beta\Bigr),
\end{eqnarray}
Thus, \eqref{DSw-1H}-\eqref{DSw-3H} prove \eqref{PKc-9H} in the case 
when $\alpha_{\mathfrak{G}}\in(2,\infty)$ and with a constant 
$c_n({\mathfrak{G}},\beta)$ as in the first line in the right-hand side 
of \eqref{P-9J}. The case when $\alpha_{\mathfrak{G}}\in(0,2)$ is treated 
similarly and we omit it. Moving on, in the case corresponding to 
$\alpha_{\mathfrak{G}}=2$ we write 
\begin{eqnarray}\label{DSw-5H}
\int_{S_{{\mathfrak{G}},r}}v_{{\mathfrak{G}},r}(x)^{-\beta}\,dx
&=& (n+2)^{\beta}\int_{S_{{\mathfrak{G}},r}}
\rho^{-2\beta}\Bigl[\log\Bigl(\frac{r}{\rho}\Bigr)\Bigr]^{-\beta}
\phi_{{\mathfrak{G}}}(\omega)^{-\beta}\,dx
\nonumber\\[4pt]
&=& (n+2)^{\beta}r^{2-2\beta}III\cdot IV,
\end{eqnarray}
where (after a natural change of variables) 
\begin{eqnarray}\label{DSw-6H}
III:=\int_0^1 t^{n-1-2\beta}
\Bigl[\log\Bigl(\frac{1}{t}\Bigr)\Bigr]^{-\beta}\,dt
\quad\mbox{ and }\quad
IV:=\int_{{\mathfrak{G}}}\phi_{{\mathfrak{G}}}(\omega)^{-\beta}\,d\omega.
\end{eqnarray}
Substituting ${\rm exp}(-s/(n-2\beta))$ for $t$ in $III$, further
transforms this term into 
\begin{eqnarray}\label{DSw-7H}
III=(n-2\beta)^{\beta-1}\int_0^\infty e^{-s}s^{-\beta}\,ds
=(n-2\beta)^{\beta-1}\Gamma(1-\beta).
\end{eqnarray}
Together, \eqref{DSw-5H}-\eqref{DSw-7H} justify \eqref{PKc-9H} 
in the case when $\alpha_{\mathfrak{G}}=2$, with a constant 
$c_n({\mathfrak{G}},\beta)$ as in the middle line in the right-hand 
side of \eqref{P-9J}. 

The last thing left to justify, in order to complete the proof 
of the lemma, is the finiteness condition in \eqref{PKc-9H}. 
This, however, is a direct consequence of \eqref{GFsA} 
(cf.~also Lemma~\ref{bvv55}).
\end{proof}

The main result in this section is the following higher-dimensional
analogue of Theorem~\ref{Te-3A}. Essentially, this asserts that
\eqref{PKc-4} holds for every $\beta\in(0,1)$ in the class of bounded
piecewise $C^1$ domains in ${\mathbb{R}}^n$ with conical singularities. 
While this constitutes a subclass of the larger class of bounded 
Lipschitz domains, it is worth recalling that Theorem~\ref{TL-AS}
establishes \eqref{PKc-4} only for a smaller range of values for 
the parameter $\beta$ (described in \eqref{ASc-10}). This is surprising
since the conclusion in (the first part of) Corollary~\ref{anLL} progressively 
weakens precisely when the Lipschitz constant of a domain becomes 
large (cf.~\eqref{RFDE.2}). Thus, for the type of 
domains considered here, the approach developed in this section yields a better
control of the $\beta$-integral of the solution 
of the Saint Venant problem \eqref{PKc-1} than 
the earlier methods based on direct pointwise estimates (from below) on the 
solution $u$ of \eqref{PKc-1} in terms of powers of the distance to the boundary. 

\begin{theorem}\label{Te-3AH}
Assume that $\Omega$ is a bounded open set in 
${\mathbb{R}}^n$, $n\geq 2$, whose boundary is of class $C^1$ 
with the exception of finitely many points $P_1,...,P_N\in\partial\Omega$,
and such that for each $i\in\{1,...,N\}$ there exist an open, 
connected subset ${\mathfrak{G}}_i$ of $S^{n-1}$ whose relative boundary is a 
submanifold of class $C^{1,\alpha}$, $\alpha\in(0,1)$, and of codimension 
one in $S^{n-1}$, along with a number $r_i>0$ with the property that 
\begin{eqnarray}\label{Gbab}
\mbox{$\Omega\cap B(P_i,r_i)$ and $S_{{\mathfrak{G}}_i,r_i}$ coincide, 
modulo a rigid transformation of ${\mathbb{R}}^n$}. 
\end{eqnarray}
Let $u$ be the solution of \eqref{PKc-1}. Then for 
every $\beta\in(0,1)$ there holds 
\begin{eqnarray}\label{PDsa-3H}
\int_{\Omega}u(x)^{-\beta}\,dx\leq C(\Omega,\beta)<+\infty.
\end{eqnarray}
\end{theorem}

\begin{proof}
Pick $P\in\partial\Omega$ with the property that there exist an open, 
connected subset ${\mathfrak{G}}$ of $S^{n-1}$ whose relative boundary 
is a $C^{1,\alpha}$ submanifold of codimension one in $S^{n-1}$ (where 
$\alpha\in(0,1)$) and $r>0$ such 
that $\Omega\cap B(P,r)$ and $S_{{\mathfrak{G}},r}$ are congruent.
Without loss of generality, assume that $P$ is the origin 
in ${\mathbb{R}}^n$ and that, in fact, $\Omega\cap B(0,r)=S_{{\mathfrak{G}},r}$. 
Bring in the barrier function $v_{{\mathfrak{G}},r}$ 
from \eqref{PKc-7H.G} and note that, by design,  
\begin{eqnarray}\label{PKc-1.DH}
\begin{array}{l}
v_{{\mathfrak{G}},r}=0\,\,\mbox{ on }\,\,\partial S_{{\mathfrak{G}},r},
\qquad v_{{\mathfrak{G}},r}>0\,\,\mbox{ in }\,\,S_{{\mathfrak{G}},r},
\\[6pt]
\mbox{and }\,\,(\Delta v_{{\mathfrak{G}},r})(\rho,\omega)
=-\phi_{{\mathfrak{G}}}(\omega)\,\,\mbox{ in }\,\,S_{{\mathfrak{G}},r}.
\end{array}
\end{eqnarray}
Indeed, these are direct consequences of \eqref{PKc-7H.G}, \eqref{La-LB}, 
\eqref{vTR1}, \eqref{critic.G}, and \eqref{Mc-T2A.G}. The normalisation 
constants in \eqref{PKc-7H} have been selected so that the 
right-hand side in the second line of \eqref{PKc-1.DH} is always 
a number belonging to the interval $[-1,0]$ (for this, \eqref{Mc-T2A.G} 
is crucial). As such, $\Delta(u-v_{{\mathfrak{G}},r})(\rho,\omega)
=-1+\phi_{{\mathfrak{G}}}(\omega)\leq 0$ for any 
$\rho\omega\in S_{{\mathfrak{G}},r}$. Since $u-v_{{\mathfrak{G}},r}$ 
is continuous in $\overline{S_{{\mathfrak{G}},r}}$ and is equal to 
(the nonnegative function) $u$ on $\partial S_{{\mathfrak{G}},r}$, 
the Maximum Principle gives that   
\begin{eqnarray}\label{PKc-2H}
u\geq v_{{\mathfrak{G}},r}\,\,\mbox{ in }\,\, S_{{\mathfrak{G}},r}.
\end{eqnarray}
In turn, this and \eqref{PKc-9H} permit us to estimate 
\begin{eqnarray}\label{PKii}
\int_{\Omega\cap B(P,r)}u(x)^{-\beta}\,dx
\leq c_n({\mathfrak{G}},\beta)\,r^{n-2\beta}\Bigl(\int_{{\mathfrak{G}}}
\phi_{{\mathfrak{G}}}(\omega)^{-\beta}\,d\omega\Bigr)<+\infty,
\end{eqnarray}
where the constant $c_n({\mathfrak{G}},\beta)$ is as in \eqref{P-9J}.
Once this local estimate near a conical point $P\in\partial\Omega$ 
has been established, the remainder of the proof follows along the
lines of the proof of Theorem~\ref{Te-3A}. 
\end{proof}

\section{Results for other classes of nonsmooth domains}
\label{sect:6}
\setcounter{equation}{0}

In this section we study the nature of $\beta$-integrals associated with 
other important classes of non-smooth domains, starting with

\subsection{The case of polyhedral domains}

Consider the case when $\Omega\subseteq{\mathbb{R}}^3$ is a 
polyhedral domain\footnote{Throughout, a polyhedral domain is understood 
to have finitely many faces, edges and vertices.}. 
Pick a vertex $x_0\in\partial\Omega$ and, for a sufficiently small 
$\varepsilon>0$, set 
\begin{eqnarray}\label{Knab}
{\mathfrak{G}}:=\bigl\{(x-x_0)/|x-x_0|:\,
x\in B(x_0,\varepsilon)\cap\Omega\bigr\}\subseteq S^2.
\end{eqnarray}
Hence, the spherical polygon ${\mathfrak{G}}$ is the profile of the cone 
which agrees with $\partial\Omega$ is a small neighbourhood of $x_0$.
In this setting, a good portion of our earlier analysis carries through 
verbatim. In particular, we may consider the eigenvalue problem 
\eqref{vTR1.G} which continues to have a solution which satisfies 
\eqref{Mc-T2A.G}. The key feature which is lost in the present setting
(in which ${\mathfrak{G}}$ no longer has a smooth boundary in $S^2$) 
is the equivalence \eqref{GFsA}. Recall that this played a basic role 
in the finiteness condition in \eqref{PKc-9H}. We nonetheless have the 
following result. 

\begin{theorem}\label{Te-3Aj}
Suppose that $u$ is the solution of the Saint Venant problem \eqref{PKc-1} 
in the case when $\Omega$ is a polyhedral domain in ${\mathbb{R}}^3$. Then 
\begin{eqnarray}\label{PDsa-3d}
\int_{\Omega}u(x)^{-\beta}\,dx\leq C(\Omega,\beta)<+\infty
\end{eqnarray}
for every $\beta\in(0,1)$. 
\end{theorem}

The remainder of this section is devoted to presenting a proof of this result. 
Dealing with the first eigenfunction $\phi_{\mathfrak{G}}$ in the case when 
${\mathfrak{G}}$ is a spherical polyhedral domain requires the following 
asymptotic representation of $\phi_{\mathfrak{G}}$ near a corner point of 
the spherical polygon ${\mathfrak{G}}$ (in the spirit of work in \cite{Ko})
\begin{eqnarray}\label{LS-e3}
\phi_{\mathfrak{G}}(\psi,\varphi)
=C_{\mathfrak{G}}\,(\sin\varphi)^{\pi/(2\theta)}\cos(\pi\psi/(2\theta))
+{\mathcal{O}}\bigl((\sin\varphi)^{\pi/(2\theta)+\varepsilon}\bigr),
\end{eqnarray}
for some $\varepsilon>0$, where $C_{\mathfrak{G}}$ is a constant depending 
on the global shape of ${\mathfrak{G}}$ and $\psi$, $\varphi$ are local 
polar coordinates near a corner vertex $O\in\partial_{S^2}{\mathfrak{G}}$, 
i.e., $0<\varphi<\!<1$, $|\psi|<\theta$ where $\theta$ is the half-aperture 
of the spherical angle at $O$. In addition, the coefficient $C_{\mathfrak{G}}$ 
in \eqref{LS-e3} is given by the following formula, itself a special case 
of closely related results proved in \cite{MP},
\begin{eqnarray}\label{LS-e4}
C_{\mathfrak{G}}=\Lambda_{\mathfrak{G}}
\int_{{\mathfrak{G}}}\phi_{\mathfrak{G}}(\omega)\zeta(\omega)\,d\omega.
\end{eqnarray}
Above, $\zeta$ is a positive function in ${\mathfrak{G}}$, harmonic  
(in the sense of Laplace-Beltrami) in ${\mathfrak{G}}$ and vanishing on 
$\partial_{S^2}{\mathfrak{G}}\setminus\{O\}$, and which exhibits 
a prescribed singularity at the vertex $O$, namely 
\begin{eqnarray}\label{GHbb}
\zeta(\psi,\varphi)\sim(2/\pi)(\sin\varphi)^{-\pi/(2\theta)}
\cos(\pi\psi/(2\theta)),\quad\mbox{uniformly for $|\psi|<\theta$, as }
\,\,\varphi\searrow 0.
\end{eqnarray}
Together with \eqref{LS-e4} and the fact that $\phi_{\mathfrak{G}}>0$, 
this analysis shows that $C_{\mathfrak{G}}>0$. 

Unfortunately, formula \eqref{LS-e3} is not sufficiently refined 
in order to allow us to estimate 
\begin{eqnarray}\label{LS-e1X.2}
\int_{{\mathfrak{G}}}\phi_{\mathfrak{G}}(\omega)^{-\beta}\,d\omega.
\end{eqnarray}
An asymptotic expansion of $\phi_{\mathfrak{G}}$ near a corner 
$O\in\partial_{S^2}{\mathfrak{G}}$ which better suits our purposes 
is contained in the lemma below. 

\begin{lemma}\label{LM-2.a}
The remainder in the asymptotic formula 
\begin{eqnarray}\label{TG-M.1}
\phi_{\mathfrak{G}}(\psi,\varphi)
=C_{\mathfrak{G}}(\sin\varphi)^{\pi/(2\theta)}\cos\bigl(\pi\psi/(2\theta)\bigr)
+A(\psi,\varphi), 
\end{eqnarray}
where $0<\varphi<\!<1$ and $|\psi|<\theta$, obeys the estimate 
\begin{eqnarray}\label{TG-M.2}
|A(\psi,\varphi)|\leq C(\sin\varphi)^{\pi/(2\theta)+\varepsilon}
\cos(\pi\psi/(2\theta)),
\end{eqnarray}
for some $\varepsilon>0$. 
\end{lemma}

\begin{proof}
Without loss of generality assume that the direction of the edge inducing 
the spherical angle with opening $2\theta$ is along the $x_3$-axis and set 
\begin{eqnarray}\label{TG-M.3}
\begin{array}{c}
x=(x',x_3),\quad 
x'=(x_1,x_2)=(\rho\sin\varphi\cos\psi,\rho\sin\varphi\sin\psi),
\\[4pt]
x_3=\rho\cos\varphi,\quad\rho=\sqrt{x_1^2+x_2^2+x_3^2}.
\end{array}
\end{eqnarray}
Let $\alpha_{\mathfrak{G}}>0$ solve \eqref{critic.G} so that, 
in particular, the function described in polar coordinates by 
$\rho^{\alpha_{\mathfrak{G}}}\phi_{\mathfrak{G}}(\psi,\varphi)$ 
is harmonic. Consequently, 
\begin{eqnarray}\label{TG-M.4}
0=\Delta_x\bigl(\rho^{\alpha_{\mathfrak{G}}}
\phi_{\mathfrak{G}}(\psi,\varphi)\bigr)
&=& C_{\mathfrak{G}}\Delta_{x'}(\rho^{\alpha_{\mathfrak{G}}-\pi/(2\theta)}
(\rho\sin\varphi)^{\pi/(2\theta)}\cos(\pi\psi/(2\theta))\bigr)
\nonumber\\[4pt]
&& +\Delta_x\bigl(\rho^{\alpha_{\mathfrak{G}}}A(\psi,\varphi)\bigr).
\end{eqnarray}
Given that
\begin{eqnarray}\label{TG-M.5}
\Delta_{x'}((\rho\sin\varphi)^{\pi/(2\theta)}\cos(\pi\psi/(2\theta))\bigr)=0,
\end{eqnarray}
we can express the last term in the first line of \eqref{TG-M.4} as 
\begin{eqnarray}\label{TG-M.6}
&& C_{\mathfrak{G}}\Delta_{x'}
(\rho^{\alpha_{\mathfrak{G}}-\pi/(2\theta)}\bigr)\cdot
(\rho\sin\varphi)^{\pi/(2\theta)}\cos(\pi\psi/(2\theta))
\nonumber\\[4pt]
&& +2C_{\mathfrak{G}}\nabla_{x'}
(\rho^{\alpha_{\mathfrak{G}}-\pi/(2\theta)}\bigr)\cdot
\nabla_{x'}\bigl((\rho\sin\varphi)^{\pi/(2\theta)}\cos(\pi\psi/(2\theta))\bigr).
\end{eqnarray}
Fix $\delta\in(0,1)$ small enough and restrict $\rho$ to the interval 
$[1-\delta,1]$. On this range, we have $|\nabla_{x'}\rho|\leq C\sin\varphi$ 
and $|\Delta_{x'}\rho|\leq C$, where the constant $C$ depends only on $\delta$.
As a result, the absolute value of the expression in \eqref{TG-M.6} does
not exceed 
\begin{eqnarray}\label{TG-M.7}
C(\sin\varphi)^{\pi/(2\theta)}.
\end{eqnarray}
Using this and \eqref{TG-M.4} we may then write 
\begin{eqnarray}\label{TG-M.8}
\Upsilon:=\Delta_x\bigl(\rho^{\alpha_{\mathfrak{G}}}A(\psi,\varphi)\bigr)
={\mathcal{O}}\Bigl((\sin\varphi)^{\pi/(2\theta)}\Bigr).
\end{eqnarray}
Next, since $\rho^{\alpha_{\mathfrak{G}}}A(\psi,\varphi)$ vanishes on 
the sides of the dihedral angle $\psi=\pm\theta$, it follows from the 
classical local regularity result of Agmon, Douglas and Nirenberg 
(cf.~\cite{ADN}) that
\begin{eqnarray}\label{TG-M.9}
\sum_{0\leq|\gamma|\leq 2}\delta^{|\gamma|-2}
\|\partial_x^\gamma(\rho^{\alpha_{\mathfrak{G}}}A)\|
_{L^p({\mathcal{E}}_{\delta})}
\leq C\Bigl(\|\Upsilon\|_{L^p(\widetilde{\mathcal{E}}_{\delta})}
+\delta^{3/p-2}\max_{\widetilde{\mathcal{E}}_{\delta}}
|\rho^{\alpha_{\mathfrak{G}}}A|\Bigr),
\end{eqnarray}
where $p\in(1,\infty)$ is fixed and we have set 
\begin{eqnarray}\label{TG-M.10}
\begin{array}{l}
{\mathcal{E}}_{\delta}:=\bigl\{x\in{\mathbb{R}}^3:\,
1-\delta\leq\rho\leq 1,\,\,|\psi|\leq\theta,\,\,
\delta\leq\sin\varphi\leq 2\delta\bigr\},
\\[4pt]
\widetilde{\mathcal{E}}_{\delta}:=\bigl\{x\in{\mathbb{R}}^3:\,
1-2\delta\leq\rho\leq 1+\delta,\,\,|\psi|\leq\theta,\,\,
\delta/2\leq\sin\varphi\leq 4\delta\bigr\}.
\end{array}
\end{eqnarray}
On account of this and the Sobolev embedding theorem we therefore obtain 
for $p>3$
\begin{eqnarray}\label{TG-M.11}
\delta^{3/p-1}\max_{{\mathcal{E}}_{\delta}}|\nabla_{x}
(\rho^{\alpha_{\mathfrak{G}}}A)|
\leq C\Bigl(\|\Upsilon\|_{L^p(\widetilde{\mathcal{E}}_{\delta})}
+\delta^{3/p-2}\max_{\widetilde{\mathcal{E}}_{\delta}}
|\rho^{\alpha_{\mathfrak{G}}}A|\Bigr).
\end{eqnarray}
On the other hand, recall that (cf.~\eqref{TG-M.8} and \eqref{LS-e3})
\begin{eqnarray}\label{TG-M.12}
\Upsilon={\mathcal{O}}\Bigl((\sin\varphi)^{\pi/(2\theta)}\Bigr)
\quad\mbox{ and }\quad 
A={\mathcal{O}}\Bigl((\sin\varphi)^{\pi/(2\theta)+\varepsilon}\Bigr).
\end{eqnarray}
In concert with \eqref{TG-M.11} this yields 
\begin{eqnarray}\label{TG-M.13}
\delta^{3/p-1}\max_{{\mathcal{E}}_{\delta}}|\nabla_{x}(\rho^{\alpha_{\mathfrak{G}}}A)|
\leq C\Bigl(\delta^{3/p+\pi/(2\theta)}+\delta^{3/p-2+\pi/(2\theta)+\varepsilon}\Bigr)
\end{eqnarray}
so that, ultimately, 
\begin{eqnarray}\label{TG-M.14}
\max_{{\mathcal{E}}_{\delta}}|\nabla_{x}(\rho^{\alpha_{\mathfrak{G}}}A)|
\leq C\delta^{\pi/(2\theta)+\varepsilon-1}.
\end{eqnarray}
Using this and the fact that $\rho^{\alpha_{\mathfrak{G}}}A$ vanishes on the sides 
of the dihedral angle $\psi=\pm\theta$ we then obtain 
\begin{eqnarray}\label{TG-M.15}
\rho^{\alpha_{\mathfrak{G}}}|A(\psi,\varphi)| &\leq & 
C\cos(\pi\psi/(2\theta))\cdot\delta\cdot\max_{{\mathcal{E}}_{\delta}}
|\nabla_{x}(\rho^{\alpha_{\mathfrak{G}}}A)|
\nonumber\\[4pt]
&\leq & C\delta^{\pi/(2\theta)+\varepsilon}\cos(\pi\psi/(2\theta)).
\end{eqnarray}
This proves that \eqref{TG-M.2} holds for small values of $\varphi$. 
\end{proof}

\begin{corollary}\label{CM-1.v}
Retaining notation introduced above we have 
\begin{eqnarray}\label{TG-M.16}
\phi_{\mathfrak{G}}(\psi,\varphi)
=\Bigl(C_{\mathfrak{G}}+{\mathcal{O}}\bigl((\sin\varphi)^{\varepsilon}
\bigr)\Bigr)(\sin\varphi)^{\pi/(2\theta)}\cos(\pi\psi/(2\theta)),
\end{eqnarray}
as $\varphi\to 0$. In addition, 
\begin{eqnarray}\label{TG-M.17}
\phi_{\mathfrak{G}}(x)
\geq C\,{\rm dist}_{S^2}\bigl(x,\partial_{S^2}{\mathfrak{G}}\bigr),
\quad\mbox{uniformly for $x\in {\mathfrak{G}}$ away from the vertices 
of ${\mathfrak{G}}$}.
\end{eqnarray}
\end{corollary}

\begin{proof}
Formula \eqref{TG-M.16} is simply a re-writing 
of \eqref{TG-M.1}-\eqref{TG-M.2}, whereas formula \eqref{TG-M.17} 
follows from a simple barrier argument and the Maximum Principle. 
\end{proof}

Let $O$ be a point on one of the edges such that its distance from the nearest 
vertex of our polyhedron is $r>0$. In such a scenario, 
${\mathfrak{G}}$ equals the diangle on the unit sphere
\begin{eqnarray}\label{ddd-1}
{\mathfrak{G}}=\{\omega=(\psi,\varphi):\,0<\varphi<\pi,\,|\psi|<\theta\}.
\end{eqnarray}
In this case, the normalised eigenfunction (in spherical polar coordinates) is 
\begin{eqnarray}\label{ddd-2}
\phi_{\mathfrak{G}}(\omega)
=(\sin\varphi)^{\pi/(2\theta)}\,\cos(\pi\psi/(2\theta)),
\qquad \omega=(\psi,\varphi).
\end{eqnarray}

After this preamble, we are ready to present the 

\vskip 0.08in
\noindent{\it Proof of Theorem~\ref{Te-3Aj}.}
Let ${\mathfrak{G}}\subseteq S^2$ be the spherical polygon from \eqref{Knab}
and consider the truncated cone $S_{{\mathfrak{G}},1}$ 
(cf.~\eqref{Mc-T1}) with edges meeting at a generic vertex $O\in{\mathbb{R}}^3$
of $\Omega$. Without loss of generality, assume that $O$ is the origin 
in ${\mathbb{R}}^3$. Throughout, $\alpha_{\mathfrak{G}}$ retains its earlier 
significance, and we denote the openings of the dihedral angles of 
$S_{{\mathfrak{G}},1}$ by $2\theta_j$, $1\leq j\leq N$. 
We wish to show that $u^{-\beta}$ is integrable near $O$, and divide 
the subsequent analysis into several cases, starting with:

\vskip 0.08in
\noindent{\it Case (i):} Assume that the vertex $O$ is such that 
$\alpha_{\mathfrak{G}}>2$ (our analysis also applies to ``fictitious vertices'', 
i.e., for points of edges with $2\theta_j<\pi/2$). Let $\Phi_{\mathfrak{G}}$ 
be the (unique) variational solution of the Dirichlet problem
\begin{eqnarray}\label{ccdd-1}
-\Delta_{S^2}\,\Phi_{\mathfrak{G}}(\omega)-6\,\Phi_{\mathfrak{G}}(\omega)=1
\quad\mbox{on}\quad{\mathfrak{G}}, 
\qquad\Phi_{\mathfrak{G}}\bigl|_{\partial_{S^2}{\mathfrak{G}}}=0.
\end{eqnarray}
Since $6<\alpha_{\mathfrak{G}}(\alpha_{\mathfrak{G}}+1)$ in the current case, 
it follows from the Maximum Principle that $\Phi_{\mathfrak{G}}>0$ 
on ${\mathfrak{G}}$. Furthermore, there exists $\delta>0$ small with the 
property that the following asymptotic representations hold for 
$\Phi_{\mathfrak{G}}$ near the $j$-th angle vertex on 
$\partial_{S^2}{\mathfrak{G}}$ (cf.~\cite{Ko} and \cite{MP} for closely 
related results):
\begin{eqnarray}\label{ccdd-2}
\Phi_{\mathfrak{G}}(\omega)=
\left\{
\begin{array}{l}
\frac{\varphi_j^2}{4}(1+{\mathcal{O}}(\varphi_j^\delta))
\Bigl(\frac{\cos(2\psi_j)}{\cos(2\theta_j)}-1\Bigr),
\qquad\,\,\mbox{if }\theta_j<\pi/4,
\\[8pt]
\varphi_j^2(1+{\mathcal{O}}(\varphi_j^\delta))
\bigl(\log\frac{1}{\varphi_j}\bigr)\cos(2\psi_j),
\quad\,\,\mbox{if }\theta_j=\pi/4,
\\[8pt]
C_j\varphi_j^{\pi/2\theta_j}(1+{\mathcal{O}}(\varphi_j^\delta))
\cos\bigl(\frac{\pi\psi_j}{2\theta_j}\bigr),\qquad\mbox{if }\theta_j>\pi/4,
\end{array}
\right.
\end{eqnarray}
where $C_j>0$ and $(\varphi_j,\psi_j)$ are the polar coordinates of the point 
$\omega\in S^2$, near the $j$-th angle vertex of $\partial_{S^2}{\mathfrak{G}}$.
In turn, this yields 
\begin{eqnarray}\label{ccdd-3}
\int_{\mathfrak{G}}\frac{d\omega}{\Phi_{\mathfrak{G}}(\omega)^\beta}<+\infty
\quad\mbox{whenever}\quad\beta<1.
\end{eqnarray}
To proceed, recall that, in general, for any function $w$ in ${\mathbb{R}}^3$ 
we have (with $\rho:=|x|$)
\begin{eqnarray}\label{ccdd-8b}
\Delta w=\rho^{-2}(\rho^2 w_\rho)_\rho+\rho^{-2}\Delta_{S^2}\,w.
\end{eqnarray}
Hence, if we now introduce the function
\begin{eqnarray}\label{ccdd-4}
v(x):=\bigl(|x|^2-|x|^{\alpha_{\mathfrak{G}}}\bigr)\Phi_{\mathfrak{G}}
\bigl(\frac{x}{|x|}\bigr)\quad\mbox{on}\quad S_{{\mathfrak{G}},1},
\end{eqnarray}
it follows that 
\begin{eqnarray}\label{ccdd-5}
&&-\Delta v(x)=1+\bigl(\alpha_{\mathfrak{G}}(\alpha_{\mathfrak{G}}+1)-6\bigr)
\bigl(\frac{x}{|x|}\bigr)^{\alpha_{\mathfrak{G}}-2}
\Phi_{\mathfrak{G}}\bigl(\frac{x}{|x|}\bigr)-|x|^{\alpha_{\mathfrak{G}}-2}
\,\,\mbox{ on }S_{{\mathfrak{G}},1}
\\[4pt]
&& v\bigr|_{\partial S_{{\mathfrak{G}},1}}=0.
\end{eqnarray}
Since, by assumption, $\alpha_{\mathfrak{G}}>2$, the right-hand side 
of \eqref{ccdd-5} bounded by some positive finite constant $c_0$. 
Based on this and the Maximum Principle we may then conclude that 
for a sufficiently small $r>0$ there holds
\begin{eqnarray}\label{ccdd-6}
u(x)\geq\frac{1}{c_0}r^2v\bigl(\frac{x}{r}\bigr)\quad
\mbox{on}\quad S_{{\mathfrak{G}},r}.
\end{eqnarray}
Now, the fact that $u^{-\beta}\in L^1(\Omega\cap B(O,r))$ follows 
from \eqref{ccdd-6} and \eqref{ccdd-3}, by observing that 
\begin{eqnarray}\label{ccdd-7}
\int_0^1\frac{\rho^2\,d\rho}{(\rho^2-\rho^{\alpha_{\mathfrak{G}}})^\beta}
=\frac{1}{\alpha_{\mathfrak{G}}-2}
\mathfrak{B}\Bigl(\frac{3-2\beta}{\alpha_{\mathfrak{G}}-2};1-\beta\Bigr)<+\infty.
\end{eqnarray}

\vskip 0.08in
\noindent{\it Case (ii):} Assume that the vertex $O$ is such that 
$\alpha_{\mathfrak{G}}=2$ (our subsequent analysis also applies to 
``fictitious'' vertices, i.e., points of edges with $2\theta_j=\pi/2$).
In this case, it follows from \eqref{critic.G} that the first eigenvalue
of $-\Delta_{S^2}$ on ${\mathfrak{G}}$ is $\Lambda_{\mathfrak{G}}=6$, and 
we recall the eigenfunction $\phi_{\mathfrak{G}}$ from  
\eqref{vTR1.G} (with $n=3$). Also, fix some small $r>0$ and set
\begin{eqnarray}\label{ccdd-8}
v_r(x):=|x|^2\Bigl(\log\frac{r}{|x|}\Bigr)
\phi_{\mathfrak{G}}\bigl(\frac{x}{|x|}\bigr)
\quad\mbox{on}\quad S_{{\mathfrak{G}},r}.
\end{eqnarray}
Hence, by \eqref{ccdd-8b}, 
\begin{eqnarray}\label{ccdd-9}
-\Delta v_r(x)=-\Bigl(\log\frac{r}{|x|}\Bigr)(\Delta_{S^2}\,\phi_{\mathfrak{G}}
+6\phi_{\mathfrak{G}})+5\phi_{\mathfrak{G}}=5\phi_{\mathfrak{G}}.
\end{eqnarray}
Since $v_r\bigl|_{\partial S_{{\mathfrak{G}},r}}=0$ and 
$u\bigl|_{\partial S_{{\mathfrak{G}},r}}\geq 0$, and since 
\begin{eqnarray}\label{ccdd-10}
-\Delta v_r(x)\leq 5\max_{{\mathfrak{G}}}\phi_{\mathfrak{G}}
=-\Delta(5\max_{{\mathfrak{G}}}\phi_{\mathfrak{G}}\cdot u(x)),
\end{eqnarray}
it follows that 
\begin{eqnarray}\label{ccdd-11}
u(x)\geq (5\max_{{\mathfrak{G}}}\phi_{\mathfrak{G}})^{-1}v_r(x)
\quad\mbox{ on }\,\,S_{{\mathfrak{G}},r}. 
\end{eqnarray}
As a consequence of this, \eqref{ccdd-8} and Corollary~\ref{TG-M.17}, 
we therefore obtain 
\begin{eqnarray}\label{TG-uuN}
u(x)\geq C|x|^2(\sin\varphi_j)^{\pi/(2\theta_j)}\cos(\pi\psi_j/(2\theta_j)),
\qquad\mbox{where }\,\,x/|x|=(\psi_j,\varphi_j),
\end{eqnarray}
for $x$ near $O$. In turn, the nature of the expression in 
the right-hand side of \eqref{TG-uuN} guarantees the integrability of 
$u^{-\beta}$ for each $\beta\in(0,1)$ over small conical neighbourhoods 
of every edge of a dihedral angle of half-opening $\theta_j\geq\pi/4$, since 
generally speaking 
\begin{eqnarray}\label{ddYFS}
\int_0^1\frac{\rho^2\,d\rho}{(\rho^2)^\beta}
\int_0^{\pi}\frac{\sin\varphi_j\,d\varphi_j}
{(\sin\varphi_j)^{(\pi/(2\theta_j))\beta}}
\int_{|\psi_j|<\theta_j}\frac{d\psi_j}{[\cos(\pi\psi_j/(2\theta_j)]^{\beta}}
<+\infty
\end{eqnarray}
whenever $\beta<\min\{1,\frac{4\theta_j}{\pi}\}$.

The treatment of the case when $\theta_j<\pi/4$ requires a different 
approach which we now describe. Let us consider a conic neighbourhood, 
centred at $O$, of the $j$-th edge of a dihedral angle of 
opening $2\theta_j$ with $\theta_j<\pi/4$, i.e., 
\begin{eqnarray}\label{TYh.1}
U_j:=\{x=(\rho,\psi_j,\varphi_j):\,\rho>0,\,\,|\psi_j|\leq\theta_j,\,\,
0<\varphi_j<\epsilon\},
\end{eqnarray}
where $\epsilon>0$ is a small number. We choose a point $P$ on this edge, 
set $r:=|P|/3$ and apply \eqref{ccdd-6} with the role of $v$ played by the
function 
\begin{eqnarray}\label{TYh.2}
v^{(P)}(x):=\bigl(|x-P|^2-|x-P|^{\pi/(2\theta_j)}\bigr)
\Phi^{(j)}\bigl({\textstyle\frac{x-P}{|x-P|}}\bigr), 
\end{eqnarray}
where $\Phi^{(j)}$ is our old $\Phi_{\mathfrak{G}}$ (cf.~\eqref{ccdd-2}) 
constructed for the point $P$. On account of the first asymptotic 
formula in \eqref{ccdd-2}, this gives (much as for \eqref{ccdd-6}) 
that on $\{x\in U_j\cap\Omega:\,|x|=r\}$ we have 
\begin{eqnarray}\label{TYh.3}
u(x) &\geq & c_o^{-1}r^2\left(\Bigl(\frac{|x-P|}{r}\Bigr)^2-
\Bigl(\frac{|x-P|}{r}\Bigr)^{\pi/(2\theta_j)}\right)
\Phi^{(j)}\bigl({\textstyle\frac{x-P}{|x-P|}}\bigr)
\nonumber\\[4pt]
&\geq & C|x-P|^2\varphi_j^2\Bigl(\frac{\cos(2\psi_j)}{\cos(2\theta_j)}-1\Bigr)
\nonumber\\[4pt]
&\geq & Cr^2\varphi_j^2\Bigl(\frac{\cos(2\psi_j)}{\cos(2\theta_j)}-1\Bigr).
\end{eqnarray}
To continue, let $\Gamma_{\mathfrak{G}}$ solve the boundary value problem 
\begin{eqnarray}\label{TYh.4}
&& -\Delta_{S^2}\Gamma_{\mathfrak{G}}-6\Gamma_{\mathfrak{G}}
=1-\Bigl(\int_{\mathfrak{G}}\phi_{\mathfrak{G}}(\omega)\,d\omega\Bigr)^{-1}
\phi_{\mathfrak{G}}\quad\mbox{ in }\,\,{\mathfrak{G}},
\nonumber\\[4pt]
&& \qquad\quad\,\,\,\,
\Gamma_{\mathfrak{G}}\Bigl|_{\partial_{S^2}{\mathfrak{G}}}=0,
\end{eqnarray}
Much as in \cite{Ko}, it follows that there exists $\delta>0$ 
with the property that 
\begin{eqnarray}\label{TYh.5}
\Gamma_{\mathfrak{G}}(\omega)=4^{-1}\varphi_j^2
\bigl(1+{\mathcal{O}}(\varphi_j^\delta)\bigr)
\Bigl(\frac{\cos(2\psi_j)}{\cos(2\theta_j)}-1\Bigr).
\end{eqnarray}
Also, clearly, 
\begin{eqnarray}\label{TYh.6}
\Delta\bigl(|x|^2\Gamma_{\mathfrak{G}}(x/|x|)\bigr)
=-(\Delta_{S^2}\Gamma_{\mathfrak{G}})(x/|x|)-6\Gamma_{\mathfrak{G}}(x/|x|)\leq 1.
\end{eqnarray}
To continue, we consider a smooth, nonnegative cutoff function $\eta_j(\omega)$
with the property that $\eta_j=1$ for $\varphi_j\in (0,\epsilon/2)$ 
and $\eta_j=0$ for $\varphi_j\in (\epsilon,\pi)$. Then 
\begin{eqnarray}\label{TYh.7}
-\Delta\bigl(\eta_j(x/|x|)|x|^2\Gamma_{\mathfrak{G}}(x/|x|)\bigr)\leq c_1
\quad\mbox{ on }\quad S_{{\mathfrak{G}},r}.
\end{eqnarray}
Furthermore, by \eqref{TYh.5}, on $\{x\in U_j\cap\Omega:\,|x|=r\}$ we have 
\begin{eqnarray}\label{TYh.8}
\eta_j(\omega)|x|^2\Gamma_{\mathfrak{G}}(\omega)\leq cr^2\varphi_j^2
\Bigl(\frac{\cos(2\psi_j)}{\cos(2\theta_j)}-1\Bigr).
\end{eqnarray}
Note that the function in the left-hand side of \eqref{TYh.8} vanishes 
on the conical side of $\partial S_{{\mathfrak{G}},r}$. The same is, obviously, 
true for $u$. Consequently, on account of this, \eqref{TYh.3} and 
\eqref{TYh.7}-\eqref{TYh.8}, we obtain 
\begin{eqnarray}\label{TYh.9}
u(x)\geq c_2\eta_j(\omega)|x|^2\Gamma_{\mathfrak{G}}(\omega)
\quad\mbox{ on }\quad\{x\in U_j\cap\Omega:\,|x|=r\}.
\end{eqnarray}
Hence, by the Maximum Principle, we ultimately have 
\begin{eqnarray}\label{TYh.10}
u(x)\geq c_2\eta_j\bigl(\frac{x}{|x|}\bigr)|x|^2\Gamma_{\mathfrak{G}}
\bigl(\frac{x}{|x|}\bigr)\quad\mbox{ for all }\quad x\in S_{{\mathfrak{G}},r}.
\end{eqnarray}
This, along with \eqref{TYh.5} now shows that $u^{-\beta}$ is integrable 
for all $\beta\in(0,1)$ over a small conical neighbourhood of the $j$-th edge, 
completing the treatment of the situation described in Case~(ii). 

\vskip 0.08in
\noindent{\it Case (iii):} Assume that the vertex $O$ is such that 
$\alpha_{\mathfrak{G}}<2$ (our analysis also applies to a point on the 
edge with $2\theta_j>\pi/2$). Hence, in this situation, 
$6>\alpha_{\mathfrak{G}}(\alpha_{\mathfrak{G}}+1)$ and, therefore, by a slight 
variant of results proved in \cite{Ko} and \cite{MP}, there exists a 
positive constant $C_0$ with the property that, granted that $|x|$ is small, 
\begin{eqnarray}\label{ccdd-12}
u(x)=C_0|x|^{\alpha_{\mathfrak{G}}}\phi_{\mathfrak{G}}\bigl(\frac{x}{|x|}\bigr)
+{\mathcal{O}}(|x|^{\alpha_{\mathfrak{G}} +\delta}),\quad\mbox{ for some }\,\,\delta>0,
\end{eqnarray}
where $\phi_{\mathfrak{G}}$ is the eigenfunction of the Laplace-Beltrami 
operator $\Delta_{S^2}$ with Dirichlet boundary condition on ${\mathfrak{G}}$
corresponding to the eigenvalue $\alpha_{\mathfrak{G}}(\alpha_{\mathfrak{G}}+1)$ 
(compare with \eqref{vTR1.G} when $n=3$). Let $\eta(\omega)$ denote a smooth 
cut-off function with support in a small conical neighbourhood of the edges, 
such that $\eta=1$ in a conical neighbourhood of every edge. Furthermore, 
let $H(x)$ be a smooth cut-off function in $C^\infty_0(B(O,2))$ with small 
support near $O$, which is identically equal to $1$ on $B(O,1)$. 

As in Case~(i), we need the solution $\Phi_{\mathfrak{G}}$ of problem 
\eqref{ccdd-1}. We wish to stress that, in the current context, it is not 
known whether $\Phi_{\mathfrak{G}}$ is of definite sign, but we shall not 
make use of this property. The asymptotic formulas for 
$\Phi_{\mathfrak{G}}(\omega)$ given in Case~(i) for $\theta_j\leq\pi/4$ 
remain valid here as well. However, for $\theta_j>\pi/4$ we can only say that 
\begin{eqnarray}\label{ccdX}
\Phi_{\mathfrak{G}}(\omega)={\mathcal{O}}\Bigl(\varphi_j^{\pi/2\theta_j}
\cos\frac{\pi\psi_j}{\theta_j}\Bigr).
\end{eqnarray}
For some small $r>0$, let us now introduce the function
\begin{eqnarray}\label{ccdd-13}
w(x):=|x|^2\Phi_{\mathfrak{G}}\bigl(\frac{x}{|x|}\bigr)
\eta\bigl(\frac{x}{|x|}\bigr)H\bigl(\frac{x}{r}\bigr).
\end{eqnarray}
Since $-\Delta(|x|^2\Phi_{\mathfrak{G}}\bigl(\frac{x}{|x|}\bigr))=1$ on 
$S_{{\mathfrak{G}},\infty}$, we have $-\Delta w=1$ in the 
intersection of $B(O,r)$ and small conical neighbourhoods of the edges meeting 
at $O$. Moreover, $|\Delta w|\leq C$ on $S_{{\mathfrak{G}},r}$. 
Finally, introduce 
\begin{eqnarray}\label{ccdd-14}
W(x):=u(x)-w(x),\quad\forall\,x\in S_{{\mathfrak{G}},r}
\end{eqnarray}
and note that $W=0$ on the conical side of 
$\partial S_{{\mathfrak{G}},r}$. Also,  
\begin{eqnarray}\label{ccdd-15}
-\Delta W=1+\Delta w\quad\mbox{on}\quad S_{{\mathfrak{G}},r},
\end{eqnarray}
and the right-hand side is bounded, and vanishes on small conical 
neighbourhoods of the edges meeting at $O$. Consequently, by \cite{Ko}, 
\begin{eqnarray}\label{ccdd-16}
W(x)=C_1|x|^{\alpha_{\mathfrak{G}}}\phi_{\mathfrak{G}}\bigl(\frac{x}{|x|}\bigr)
+{\mathcal{O}}(|x|^{\alpha_{\mathfrak{G}}+\delta}).
\end{eqnarray}
Comparing this with \eqref{ccdd-12} and using \eqref{ccdX}, \eqref{ccdd-13}
and \eqref{ccdd-14}, we derive from $\alpha_{\mathfrak{G}}<2$ that 
actually $C_1=C_0$. Using the asymptotics \eqref{TG-M.16} of 
$\phi_{\mathfrak{G}}$ near angle vertices on $\partial_{S^2}{\mathfrak{G}}$
as well as the harmonicity of the remainder term in \eqref{ccdd-16}, 
we see that the remainder term in \eqref{ccdd-16} can be replaced by
${\mathcal{O}}(|x|^{\alpha_{\mathfrak{G}}+\delta})\phi_{\mathfrak{G}}(\frac{x}{|x|})$. 
(The argument is similar to the proof of Lemma~\ref{LM-2.a} and we omit it.)
Thus, 
\begin{eqnarray}\label{ccdd-17}
W(x)=C_0|x|^{\alpha_{\mathfrak{G}}}\bigl(1+{\mathcal{O}}(|x|^{\delta})\bigr)
\phi_{\mathfrak{G}}\bigl(\frac{x}{|x|}\bigr),
\end{eqnarray}
and we have by \eqref{ccdd-13} and \eqref{ccdd-14}
\begin{eqnarray}\label{ccdd-18}
u(x) &=& w(x)+W(x)
\nonumber\\[4pt]
&=&|x|^2\Phi_{\mathfrak{G}}\bigl(\frac{x}{|x|}\bigr)
\eta\bigl(\frac{x}{|x|}\bigr)H\bigl(\frac{x}{r}\bigr)
+C_0|x|^{\alpha_{\mathfrak{G}}}\bigl(1+{\mathcal{O}}(|x|^{\delta})\bigr)
\phi_{\mathfrak{G}}\bigl(\frac{x}{|x|}\bigr).
\end{eqnarray}
This means that for $|x|<r$ in a small angular neighbourhood of the 
edges meeting at $O$ and such that $\theta_j\leq\pi/4$ the following holds
\begin{eqnarray}\label{ccdd-19}
u(x)\geq c_1|x|^2\Phi_{\mathfrak{G}}\bigl(\frac{x}{|x|}\bigr)
\geq c_2(|x|\varphi_j)^2.
\end{eqnarray}
On the other hand, in a small angular neighbourhood of the 
edges with $\theta_j>\pi/4$, for $|x|<r$ we have that
\begin{eqnarray}\label{ccdd-20}
u(x) &\geq & C_0|x|^{\alpha_{\mathfrak{G}}}\bigl(1+{\mathcal{O}}(|x|^{\delta})\bigr)
\phi_{\mathfrak{G}}\bigl(\frac{x}{|x|}\bigr)
-C|x|^2\varphi_j^{\pi/(2\theta_j)}\cos\bigl(\frac{\pi\psi_j}{2\theta_j}\bigr)
\nonumber\\[4pt]
&\geq & C_1|x|^{\alpha_{\mathfrak{G}}}
\varphi_j^{\pi/(2\theta_j)}\cos\bigl(\frac{\pi\psi_j}{2\theta_j}\bigr),
\end{eqnarray}
where $C_1$ is a finite constant. Finally, when $x$ is at a fixed, 
positive angular distance to the edges meeting at $O$, we may conclude
in the same way as in \eqref{TYh.3} that 
\begin{eqnarray}\label{ccdd-21}
u(x)\geq C_2|x|^{\alpha_{\mathfrak{G}}-1}{\rm dist}\,(x,\partial\Omega).
\end{eqnarray}
Collectively, \eqref{ccdd-19}, \eqref{ccdd-20} and \eqref{ccdd-21} prove 
that $u^{-\beta}$ is integrable near $O$, for each number $\beta$ in the 
interval $(0,1)$.
\hfill$\Box$
\vskip 0.08in

\subsection{Piecewise $C^2$ domains with outward 
cuspidal vertices in ${\mathbb{R}}^n$}

Here we elaborate on the case of domains with exterior cusps. 
This class of domains consists of bounded open sets 
$\Omega\subseteq{\mathbb{R}}^n$ with a piecewise $C^2$ boundary, 
exhibiting finitely many exterior cusps. By definition, $x_0\in\partial\Omega$ 
is called an exterior cusp if, after a rigid transformation of the space 
which maps $x_0$ into the origin of ${\mathbb{R}}^n$, there exist two small 
numbers $\varepsilon,\eta>0$ along with a function
${\mathcal{F}}\in C^2([0,\eta])$ with ${\mathcal{F}}(0)=0$, ${\mathcal{F}}>0$ on $(0,\eta]$ and ${\mathcal{F}}'(0)=0$, and such that 
$\{x\in\Omega:\,x_n\leq\eta\}$ coincides with the cuspidal set
\begin{eqnarray}\label{ddd-4}
\{x=(x',x_n):\,0<x_n\leq\eta,\,|x'|<\varepsilon{\mathcal{F}}(x_n)\}.
\end{eqnarray}

\begin{theorem}\label{Te-3SZ}
Assume that $u$ is the solution of the Saint Venant problem \eqref{PKc-1} 
in the case when $\Omega\subseteq{\mathbb{R}}^n$, $n\geq 2$, is 
a domain with exterior cusps. If $n\geq 3$ then for every $\beta\in(0,1)$ there
holds 
\begin{eqnarray}\label{PDS.11}
\int_{\Omega}u(x)^{-\beta}\,dx\leq C(\Omega,\beta)<+\infty
\end{eqnarray}
The same is true in the two dimensional setting 
provided $0<\beta\leq 1/2$. Finally, in the case when $n=2$ and 
$\beta\in(1/2,1)$, then \eqref{PDS.11} holds if and only the finiteness 
condition 
\begin{eqnarray}\label{ddXCT}
\int_0^\eta{\mathcal{F}}(\tau)^{1-2\beta}\,d\tau<+\infty
\end{eqnarray}
holds for every boundary cusp. 
\end{theorem}

\begin{proof}
Assume that $0\in\partial\Omega$ is a cusp, and ${\mathcal{F}}$ is 
as in the preamble to this subsection. Without loss of generality assume that 
$\eta=1$ (hence ${\mathcal{F}}$ is defined on $[0,1]$). 
The function $v(x):=\varepsilon^2{\mathcal{F}}^2(x_n)-|x'|^2$ satisfies
\begin{eqnarray}\label{ddd-5}
-\Delta v(x)=2(n-1)-\varepsilon^2({\mathcal{F}}^2)''(x_n)
\end{eqnarray}
and its trace on $\partial\Omega$ is nonnegative and vanishes when $x_n<1$. 

As usual, $u$ stands for the unique solution of
\begin{eqnarray}\label{ddd-6}
-\Delta u=1\,\,\mbox{ in }\,\Omega,\quad u=0\,\,\mbox{ on }\,\partial\Omega.
\end{eqnarray}
Since $-\Delta v>1$ on $\{x\in\Omega:\,x_n<1\}$ 
(assuming $\varepsilon$ small), we have 
\begin{eqnarray}\label{ddd-7}
u\leq v\,\,\mbox{ on }\,\,\{x\in\Omega:\,x_n<1\}.
\end{eqnarray}
We are next going to obtain an opposite estimate. First, by the smallness of
$\varepsilon$, we have
\begin{eqnarray}\label{ddd-8}
-\Delta v\leq 2n=-2n\Delta u.
\end{eqnarray}
Next, by the Giraud-Hopf Lemma (cf. Lemma~\ref{quanthopf} in the Appendix), we obtain
\begin{eqnarray}\label{ddd-9}
-\partial_\nu u\geq C>0\,\,\mbox{ on }\,\,
\{x\in\partial\Omega:\,x_n=1\},
\end{eqnarray}
where $\nu$ is the outward unit normal to $\Omega$. Therefore, for some $c_1>0$,
we have (cf. Lemma~\ref{quanthopf} in the Appendix)
\begin{eqnarray}\label{ddd-10}
u(x)\geq c_1\,{\rm dist}\,(x,\partial\Omega)
\,\,\mbox{ on }\,\,\{x\in\Omega:\,x_n=1\}.
\end{eqnarray}
The estimate
\begin{eqnarray}\label{ddd-11}
v(x)\leq c_2\,{\rm dist}\,(x,\partial\Omega)\,\,\mbox{ on }\,\,
\{x\in\Omega:\,x_n=1\}
\end{eqnarray}
for some $c_2>0$ follows directly from the definition of $v$. Hence, taking
$c:=\max\,\{2n,c_2/c_1\}$, we have 
\begin{eqnarray}\label{ddd-12.U}
v(x)\leq cu(x)\,\,\mbox{ on }\,\,\{x\in\Omega:\,x_n=1\}.
\end{eqnarray}
Recall that the traces of $u$ and $v$ on $\partial\Omega$ vanishes when $x_n<1$,
while $v-cu\leq 0$ on $\{x\in\Omega:\,x_n=1\}$ by \eqref{ddd-12.U}. 
Given that $\Delta(v-cu)\geq 0$ on $\{x\in\Omega:\,x_n<1\}$ by \eqref{ddd-8}
and the choice of $c$, the Maximum Principle then gives
\begin{eqnarray}\label{ddd-13}
v(x)\leq c\,u(x)\,\,\mbox{ on }\,\,\{x\in\Omega:\,x_n\leq1\}.
\end{eqnarray}
If the origin is the only singularity of $\partial\Omega$, the condition 
$u^{-\beta}\in L^1(\Omega)$ is equivalent to
\begin{eqnarray}\label{ddd-14}
\int_{\{x\in\Omega:\,x_n<1\}}\frac{dx}{v(x)^{\beta}}<+\infty,
\end{eqnarray}
which is the same as 
\begin{eqnarray}\label{ddd-15}
\int_0^1dx_n\int_{|x|<\varepsilon{\mathcal{F}}(x_n)}
\frac{dx'}{v(x)^{\beta}}<+\infty.
\end{eqnarray}
This can be written in the form
\begin{eqnarray}\label{ddd-16}
+\infty&>&\int_0^1dx_n\int_0^{\varepsilon{\mathcal{F}}(x_n)}
\frac{\rho^{n-2}d\rho}
{(\varepsilon^2{\mathcal{F}}^2(x_n)-\rho^2)^{\beta}}
\nonumber\\[4pt]
&=&\int_0^1(\varepsilon{\mathcal{F}}(x_n))^{n-1-2\beta}\,dx_n
\,\int_0^1\frac{\tau^{n-2}\,d\tau}{(1-\tau^2)^{\beta}}, 
\end{eqnarray}
which is satisfied if and only if $\beta<1$ for $n\geq 3$, and 
if and only if $\beta<1$ and 
\begin{eqnarray}\label{ddTT}
\int_0^1{\mathcal{F}}(\tau)^{1-2\beta}\,d\tau<+\infty
\end{eqnarray}
for $n=2$ (note that \eqref{ddTT} is always satisfied when 
$\beta\leq 1/2$; compare also with the example \eqref{JHca-1} in which case 
we take ${\mathcal{F}}(\tau):=\tau^{1/(2\beta-1)}$). 

Once the contribution from near each boundary cusp has been estimated, 
the end-game of the argument is similar to that of the proof of 
Theorem~\ref{Te-3AH}. 
\end{proof}

\section{Appendix}
\label{sect:A}
\setcounter{equation}{0}

This appendix is devoted to presenting a proof of the equivalence in \eqref{GFsA}.
The main ingredient, contained in Lemma~\ref{quanthopf} below,  
is a version of classical work by Giraud\footnote{A result related to 
our Lemma~\ref{quanthopf} is stated in \cite[Theorem~3,IV, p.\,7]{Miranda}
for $C^{1,\alpha}$ domains, though the proof given there actually requires a 
$C^2$ boundary. This result is attributed to Giraud who has indeed dealt in 
\cite[p.\,50]{Giraud} with $C^{1,\alpha}$ domains via an argument based on 
a change of variables (this explains the global nature of the smoothness 
assumption on the boundary of the domain in question). By way of contrast, 
our proof, which is based on a barrier construction, works for more general 
domains and, at the same time, appears conceptually simpler and significantly 
shorter than the one provided in \cite{Giraud}.} in domains of class $C^{1,\alpha}$,
$\alpha\in(0,1)$. Its statement has two parts, the first of which may be regarded 
as a quantitative version of the classical Hopf lemma, while the second part 
gives two-sided pointwise bounds for functions satisfying conditions which are
reminiscent of the properties of the eigenfunction corresponding to the first
eigenvalue for the Laplacian (cf. \eqref{vTR1.G} and \eqref{Mc-T2A.G}).
Significantly, the format of this result is designed in a manner which makes 
it applicable to localised versions of the Laplace-Beltrami operator 
on $C^{1,\alpha}$ domains on a given smooth surface (we shall comment on 
the actual localisation procedure in the last part of the appendix). 

To set the stage, we make a few definitions.
Given an index $\alpha\in(0,1]$ and two numbers $a,b>0$, consider the region 
in ${\mathbb{R}}^n$ given by 
\begin{eqnarray}\label{PAR-66.HH}
G_{a,b}^{\alpha}:=\{x=(x_1,\dots,x_n)\in{\mathbb{R}}^n:\,a|x|^{1+\alpha}<x_n<b\}
\end{eqnarray}
and call it a pseudo-ball (note that $G_{a,b}^{1}$ is a genuine solid spherical cap).

\begin{definition}\label{DEE-F1}
(i) Let $\Omega\subseteq{\mathbb{R}}^n$ be an open set, and assume that 
$x_0\in\partial\Omega$. Then $\Omega$ is said to satisfy an interior
pseudo-ball condition at $x_0$ provided there exists a rigid transformation 
$R:{\mathbb{R}}^n\to{\mathbb{R}}^n$ (i.e., a composition between a translation 
and a rotation in ${\mathbb{R}}^n$) with the property that $R(x_0)=0$ and
there exist an index $\alpha\in(0,1)$ and two numbers $a,b>0$ (collectively 
referred to as the pseudo-ball character of $\Omega$ at $x_0$)
such that $G_{a,b}^{\alpha}\subseteq R(\Omega)$. 

(ii) An open set $\Omega\subseteq{\mathbb{R}}^n$ is said to satisfy a
uniform interior pseudo-ball condition relative to a subset 
$\Sigma$ of $\partial\Omega$ provided $\Omega$ satisfies 
an interior pseudo-ball condition at every boundary point $x\in\Sigma$
with parameters $\alpha\in(0,1)$ and $a,b>0$ independent of $x$.
Collectively, $\alpha,a,b$ will be referred to as the uniform 
pseudo-ball character of $\Omega$ relative to $\Sigma$. 
\end{definition}

There is one important property of domains $\Omega\subseteq{\mathbb{R}}^n$ 
of class $C^1$ (cf. the last part in  Definition~\ref{lipdom}) satisfying 
a pseudo-ball condition at a boundary point $x_0$, namely 
\begin{eqnarray}\label{GVV.JJ}
\mbox{the outward unit normal at $x_0$ is $-R^{-1}e_n$},
\end{eqnarray}
where $e_n:=(0,\cdots,0,1)\in{\mathbb{R}}^n$ and 
$R:{\mathbb{R}}^n\to{\mathbb{R}}^n$ is the rigid transformation appearing in 
the first part of Definition~\ref{DEE-F1}. This is going to be of relevance 
shortly (cf. \eqref{PAR-66.2} below). For now, we note that one expeditious 
way of justifying the aforementioned property is by invoking the following 
result of geometric measure theoretic flavour (which appears as Proposition~2.9 
in \cite{HMT2}):

\begin{lemma}\label{Pop-CoN1}
Let $\Omega$ be a proper, nonempty open subset of ${\mathbb{R}}^n$, of locally 
finite perimeter. Fix a point $x_0$ belonging to $\partial^{\ast}\Omega$ 
(the reduced boundary of $\Omega$) with the property that there exists a 
(circular, open, truncated, one-component, not necessarily upright) cone 
$\Gamma$ in ${\mathbb{R}}^n$ with vertex at $0\in{\mathbb{R}}^n$ and having 
total aperture $\theta\in(0,\pi)$, for which 
\begin{eqnarray}\label{Cn-H1}
x_0+\Gamma\subseteq\Omega.
\end{eqnarray}
Denote by $\Gamma_{\ast}$ the (circular, open, infinite, one-component) 
cone in ${\mathbb{R}}^n$ with vertex at $0\in{\mathbb{R}}^n$, of total 
aperture $\pi-\theta$, having the same axis as $\Gamma$ and pointing in 
the opposite direction to $\Gamma$. Then, if $\nu(x_0)$ denotes the geometric 
measure theoretic outward unit normal to $\partial\Omega$ at $x_0$, it is that  
\begin{eqnarray}\label{Cn-H2}
\nu(x_0)\in\Gamma_{\ast}. 
\end{eqnarray}
\end{lemma}
Indeed, if $G^{\alpha}_{a,b}$ is the pseudo-ball associated with the 
point $x_0\in\partial\Omega$ as in the first part of Definition~\ref{DEE-F1}, 
it suffices to observe that the subset $R^{-1}(G^{\alpha}_{a,b})$ of $\Omega$ 
contains truncated circular cones of total aperture arbitrarily close to $\pi$ 
and whose axes are along $R^{-1}e_n$. From this \eqref{GVV.JJ} readily follows. 

We are now prepared to state and prove the main result in the appendix.
While other versions naturally present themselves (for example, a suitable 
variant of this result continues to hold for domains of class $C^{1,\omega}$ 
where $\omega$ is a modulus of continuity satisfying a Dini condition), 
the result proved here more than suffices for our purposes, and its
treatment has the advantage of being largely self-contained. 

\begin{lemma}\label{quanthopf}  
Suppose $\Omega\subseteq{\mathbb{R}}^n$ is a bounded domain of class $C^1$
and denote by $\nu$ the outward unit normal to $\Omega$. Assume that 
\begin{eqnarray}\label{Gvv-Rf.4}
L:=-\sum_{i,j=1}^n a^{ij}\partial_i\partial_j+\sum_{i=1}^n b^{i}\partial_i
\end{eqnarray}
is a nondivergence-form, second-order, differential operator whose 
coefficients satisfy 
\begin{eqnarray}\label{Gvv-Rf}
a^{ij},b^i\in C^0(\overline{\Omega}),\,\,1\leq i,j\leq n,
\,\mbox{ and }\,\sum_{i,j=1}^n a^{ij}(x)\xi_i\xi_j\geq c|\xi|^2
\end{eqnarray}
for every point $x\in\Omega$ and every vector $\xi\in{\mathbb{R}}^n$, 
where $c>0$ is a fixed constant. Finally, consider a real-valued function 
$\phi\in C^1(\overline{\Omega})\cap C^2(\Omega)$ satisfying 
\begin{eqnarray}\label{pos.KS}
L\phi\geq 0\,\,\mbox{ in }\,\,\Omega, 
\end{eqnarray}
and assume that $x_0\in\partial\Omega$ is a point with the property that $\Omega$ 
satisfies an interior pseudo-ball condition at $x_0$ and for which
\begin{eqnarray}\label{pos}
\phi(x_0)<\inf_{K}\phi\quad\mbox{for any compact subset $K$ of $\Omega$}.
\end{eqnarray}
Then there exist a compact subset $K_0$ of $\Omega$ and a constant $\kappa>0$ 
which depends only on 
\begin{eqnarray}\label{pos.2}
\begin{array}{c}
\displaystyle
c,\quad\,\phi(x_0),\quad\,\,
\inf_{K_0}\phi,\quad\,\,
\sum_{i,j=1}^n\sup\limits_{\overline{\Omega}}|a^{ij}|
+\sum_{i=1}^n\sup_{\overline{\Omega}}|b^{i}|,
\\[6pt]
\mbox{as well as the pseudo-ball character of $\Omega$ at $x_0$},
\end{array}
\end{eqnarray}
with the property that 
\begin{eqnarray}\label{JfRR-4}
-(\partial_\nu\phi)(x_0)\geq\kappa.
\end{eqnarray}

Furthermore, if $\Omega\subseteq{\mathbb{R}}^n$ is a bounded domain of class 
$C^1$, satisfying a uniform interior pseudo-ball condition relative to 
$\partial\Omega\cap B(x_{\#},R)$ for some point $x_{\#}\in\partial\Omega$ 
and number $R>0$, the operator $L$ is as before, and if 
$\phi\in C^1(\overline{\Omega})\cap C^2(\Omega)$ is a real-valued 
function satisfying 
\begin{eqnarray}\label{II-pos}
L\phi\geq 0\,\,\mbox{ in }\,\,\Omega,\quad
\phi>0\,\,\mbox{ in }\,\,\Omega,\,\,\mbox{ and }\,\,
\phi=0\,\,\mbox{ on }\,\,\partial\Omega\cap B(x_{\#},R),
\end{eqnarray}
then there exist a compact subset $K_0$ of $\Omega$ along with two constants 
$c',c''>0$ which depend only on $R$, $\|\nabla\phi\|_{L^\infty(\Omega)}$ 
the quantities listed in the first line of \eqref{pos.2} and the uniform
pseudo-ball character of $\Omega$ relative 
to $\partial\Omega\cap B(x_{\#},R)$ with the property that 
\begin{eqnarray}\label{GTvv-Yg}
c'\,{\rm dist}\,(x,\partial\Omega)\leq\phi(x)\leq
c''\,{\rm dist}\,(x,\partial\Omega),
\quad\mbox{ for every }\,\,x\in\Omega\cap B(x_{\#},R/2). 
\end{eqnarray}

In particular, with the same background assumptions on the operator $L$ and 
the function $\phi$ as above, all earlier conclusions hold in domains of class
$C^{1,\alpha}$ for some $\alpha\in(0,1)$.
\end{lemma}

\begin{proof}
Given that both the hypotheses and the conclusion in the statement of the lemma
are invariant under rotations and translations, there is no loss of generality
in assuming that $x_0$ is the origin in ${\mathbb{R}}^n$ and that the tangent 
plane to $\partial\Omega$ at $x_0$ is ${\mathbb{R}}^{n-1}\times\{0\}$.
Hence, in particular, $\nu(x_0)=(0,\dots,0,-1)$.
Also, since both $L$ and $\partial_\nu$ annihilate constants, we may assume 
that $\phi(x_0)=0$. 

To proceed, we note that owing to assumptions and the discussion pertaining to 
\eqref{GVV.JJ} there exist an exponent $\alpha\in(0,1)$ and two 
constants $a,b_{\ast}>0$ which depend exclusively on the pseudo-ball 
character of $\Omega$ at $x_0$ with the property that if $G^{\alpha}_{a,b}$ 
denotes the region introduced in \eqref{PAR-66.HH} then 
\begin{eqnarray}\label{PAR-66.2}
\overline{G^{\alpha}_{a,b}}
\setminus\{0\}\subseteq\Omega,\qquad\forall\,b\in(0,b_{\ast}].
\end{eqnarray}
Fix $b\in(0,b_{\ast}]$ and, for two constants $C_0,C_1>0$ to be specified later,
consider the barrier function 
\begin{eqnarray}\label{pos.4}
v(x):=x_n+C_0x_n^{1+\alpha}-C_1|x|^{1+\alpha},
\quad\mbox{for every }\,\,x=(x_1,...,x_n)\in\overline{G^{\alpha}_{a,b}}.
\end{eqnarray}
A direct computation then gives that, for each $x\in G^{\alpha}_{a,b}$,
\begin{eqnarray}\label{pos.5}
(Lv)(x)=I+II+III,
\end{eqnarray}
where 
\begin{eqnarray}\label{pos.5B}
I &:=& C_1(\alpha^2-1)\Bigl(\sum_{i,j=1}^na^{ij}(x)x_ix_j\Bigr)|x|^{\alpha-3}
+C_1(\alpha+1)\Bigl(\sum_{i=1}^na^{ii}(x)\Bigr)|x|^{\alpha-1},
\\[6pt]
II &:=& -C_0\alpha(\alpha+1)a^{nn}(x)x_n^{\alpha-1},
\label{pos.5B.2}
\\[6pt]
III &:=& b^n(x)+C_0(\alpha+1)b^n(x)x_n^\alpha
-C_1(\alpha+1)\Bigl(\sum_{i=1}^nb^{i}(x)x_i\Bigr)|x|^{\alpha-1}.
\label{pos.5B.3}
\end{eqnarray}
In concert with the uniform ellipticity condition for $L$
(which, in particular, entails $a^{nn}(x)\geq c$), 
formulas \eqref{pos.5}-\eqref{pos.5B.3} then allow us to estimate 
\begin{eqnarray}\label{pos.6}
(Lv)(x)\leq -Ax_n^{\alpha-1}+B|x|^{\alpha-1}+C,\qquad\forall\,x\in G^{\alpha}_{a,b},
\end{eqnarray}
where 
\begin{eqnarray}\label{pos.6.B}
&& A:=cC_0\alpha(\alpha+1),\qquad
B:=C_1(\alpha+1)\Bigl(c(\alpha-1)+\sum_{i=1}^n\sup_{\overline{\Omega}}|a^{ii}|\Bigr),\quad
\\[4pt]
&& C:=[C_0(\alpha+1)b^\alpha+1]\cdot\sup_{\overline{\Omega}}|b^{n}|
+C_1(\alpha+1)b^{\frac{\alpha}{\alpha+1}}
\Bigl(\sum_{i=1}^n\sup_{\overline{\Omega}}|b^{i}|\Bigr).
\label{pos.6.B.2}
\end{eqnarray}
Let us also observe that 
\begin{eqnarray}\label{NBB.1}
v(x)=(1-a^{-1}C_1)x_n+C_0x_n^{1+\alpha}
\quad\mbox{for every }\,\,x=(x_1,...,x_n)\in\partial G^{\alpha}_{a,b}\setminus
\{x\in{\mathbb{R}}^n:\,x_n=b\}.
\end{eqnarray}
Hence, by selecting first $C_1>a$, then $C_0$ sufficiently large and, finally, 
$b\in(0,b_{\ast}]$ sufficiently small, it follows from
\eqref{pos.6}-\eqref{pos.6.B.2} and \eqref{NBB.1} that matters may be arranged 
so that 
\begin{eqnarray}\label{pos.6.K}
Lv\leq 0\quad\mbox{ on }\,\,\,G^{\alpha}_{a,b},
\end{eqnarray}
and
\begin{eqnarray}\label{NBB.2}
v\leq 0\quad\mbox{on }\,\,\,\partial G^{\alpha}_{a,b}\setminus
\{x=(x_1,...,x_n)\in{\mathbb{R}}^n:\,x_n=b\}.
\end{eqnarray}
If we now consider the compact subset of $\Omega$ given by 
\begin{eqnarray}\label{pos.3b}
K_0:=\{x=(x_1,...,x_n)\in\overline{G^{\alpha}_{a,b}}:\,x_n=b\},
\end{eqnarray}
then, thanks to \eqref{NBB.2}, \eqref{pos}, and the earlier conventions 
made in the proof, we may choose $\varepsilon>0$, depending only on 
the constants listed in \eqref{pos.2} so that, on the one hand, 
\begin{eqnarray}\label{pos.7}
0\leq\phi(x)-\varepsilon v(x)\,\,\mbox{ for every $x\in\partial G^{\alpha}_{a,b}$}. 
\end{eqnarray}
On the other hand, from \eqref{pos.6.K} and \eqref{pos.KS} we obtain 
\begin{eqnarray}\label{pos.3aa}
L(\phi-\varepsilon v)\geq 0\,\,\,\mbox{ in }\,\,G^{\alpha}_{a,b}.
\end{eqnarray}
Having established \eqref{pos.7} and \eqref{pos.3aa}, the Weak Maximum 
Principle (cf., e.g., \cite[p.\,329]{Evans}) then gives 
\begin{eqnarray}\label{pos.8}
\phi-\varepsilon v\geq 0\,\,\mbox{ in $\overline{G^{\alpha}_{a,b}}$}. 
\end{eqnarray}
Given that both $\phi$ and $v$ vanish at $x_0=0\in\partial G^{\alpha}_{a,b}$, 
this proves that the function 
$\phi-\varepsilon v\in C^1(\overline{G^{\alpha}_{a,b}})$ 
has a global minimum at $x_0$. Hence, 
\begin{eqnarray}\label{pos.9}
(\partial_\nu\phi)(x_0)-\varepsilon(\partial_\nu v)(x_0)\leq 0,
\end{eqnarray}
which further entails 
\begin{eqnarray}\label{pos.10}
-(\partial_\nu\phi)(x_0)\geq -\varepsilon(\partial_\nu v)(x_0) 
=\varepsilon\frac{\partial v}{\partial x_n}(0)=\varepsilon>0.
\end{eqnarray}
Choosing $\kappa:=\varepsilon>0$ then finishes the proof of \eqref{JfRR-4}.

Consider next the claim made in the last part of the statement of the lemma. 
Given that $\phi$ is continuous in $\overline{\Omega}$, the last two 
conditions in \eqref{II-pos} ensure that \eqref{pos} is satisfied for 
any $x_0\in\partial\Omega\cap B(x_{\#},R)$. To proceed, we note that 
since $\partial\Omega$ is a compact $C^{1}$ surface, the following property holds:
\begin{eqnarray}\label{Gb65}
\begin{array}{l}
\mbox{for every $x\in\Omega$ there exists a point 
$x_{\ast}\in\partial\Omega$}
\\[4pt] 
\mbox{so that if $r:={\rm dist}\,(x,\partial\Omega)$ then
$x=x_{\ast}-r\nu(x_{\ast})$}.
\end{array}
\end{eqnarray}
Indeed, $B(x,r)\subseteq\Omega$ and if $x_{\ast}\in\partial\Omega$ is such that
${\rm dist}\,(x,\partial\Omega)=|x-x_{\ast}|$ then $x_\ast\in\partial B(x,r)
\cap\partial\Omega$. Then elementary geometrical analysis gives that
$\nu(x^*)$ is parallel to $x-x_{\ast}$, from which the desired conclusion follows. 
Next, fix an arbitrary point $x\in\Omega\cap B(x_{\#},R/2)$ and suppose
that $x_\ast\in\partial\Omega$ is associated with $x$ as in \eqref{Gb65}.
Also, assume that $R>0$ is sufficiently small to begin with. Then, after making 
a translation and a rotation, matters may be arranged so that (with $G^{\alpha}_{a,b}$ as before):
\begin{eqnarray}\label{Gb65.L}
\begin{array}{l}
\mbox{$x_\ast$ is the origin in ${\mathbb{R}}^n$, 
$\nu(x_\ast)=(0,\dots,0,-1)\in{\mathbb{R}}^n$},
\\[4pt]
\mbox{and $x=(0,\dots,0,r)\in G^{\alpha}_{a,b}$, 
where $r:={\rm dist}\,(x,\partial\Omega)$}.
\end{array}
\end{eqnarray}
In such a scenario, \eqref{pos.8} and \eqref{pos.4} then yield 
\begin{eqnarray}\label{pos.7BB}
\phi(x)\geq\varepsilon v(x)=\varepsilon(r+C_0 r^{1+\alpha}-C_1 r^{1+\alpha})
\geq c'r=c'\,{\rm dist}\,(x,\partial\Omega),
\end{eqnarray}
given that $R$ small forces $r$ small. This takes care of the lower bound 
for $\phi$ in \eqref{GTvv-Yg}. There remains to establish the upper bound 
for $\phi$ in \eqref{GTvv-Yg}. Keeping in mind \eqref{Gb65.L} and recalling 
that $\phi$ vanishes on $\partial\Omega\cap B(x_{\#},R)$, the Mean Value 
Theorem then allows us to estimate 
\begin{eqnarray}\label{Jbb-Tg}
\phi(x)=\phi(x)-\phi(x_{\ast})\leq r\cdot\sup_{t\in (0,r)}
\bigl|\nu(x_{\ast})\cdot(\nabla\phi)(x_{\ast}-t\nu(x_{\ast}))\bigr|
\leq c''\,{\rm dist}\,(x,\partial\Omega),
\end{eqnarray}
where $c'':=\|\nabla\phi\|_{L^\infty(\overline{\Omega})}$. 
This concludes the proof of the lemma.
\end{proof}

Moving on, we wish to note that the smoothness requirements for the 
function $\phi$ from Lemma~\ref{quanthopf} are automatically satisfied 
in the case in which $\Omega$ is a bounded domain of class
$C^{1,\alpha}$ for some $\alpha\in(0,1)$ and the function $\phi$ is a 
classical solution of the boundary value problem 
\begin{eqnarray}\label{Td-Yt.33}
(L-V)\phi=0\,\,\mbox{ in }\,\,\Omega,\qquad
\phi=0\,\,\mbox{ on }\,\,\partial\Omega,
\end{eqnarray}
where $L$ is a second-order, elliptic, nondivergence-form differential operator 
with smooth coefficients, and $V$ is a smooth scalar function, in $\overline{\Omega}$.
Indeed, in this context the function $\phi$ actually satisfies  
$\phi\in C^{1,\alpha}(\overline{\Omega})\cap C^\infty(\Omega)$. Such a 
regularity result is well-known; see, e.g., \cite[Th\'eor\`eme~I, p.\,42]{Giraud}
(cf. also \cite{ADN}, or the discussion in \cite[Chapter~7]{MazShap}). 

The last step in the proof of \eqref{GFsA} has to do with localisation. 
In this vein, recall that if $M$ is a $C^2$ boundaryless manifold of 
(real) dimension $n$ equipped with a $C^1$ Riemannian metric tensor 
$g=\sum\limits_{j,k=1}^n g_{jk}\,dx_j\otimes dx_k$ then, in local coordinates, 
the Laplace-Beltrami operator $\Delta_M$ on $M$ is given by 
\begin{eqnarray}\label{eq2.3}
\Delta_M:=\frac{1}{\sqrt{g}}\sum\limits_{j,k=1}^n\partial_j
\bigl(g^{j,k}\sqrt{g}\,\partial_k\,\cdot\,\bigr)
\end{eqnarray}
where $(g^{jk})_{1\leq j,k\leq n}$ is the inverse of the matrix 
$(g_{jk})_{1\leq j,k\leq n}$ and $g:={\rm det}\,(g_{jk})_{1\leq j,k\leq n}$. 
The relevant issue for us here is that such a differential operator falls
within the class of operators considered in Lemma~\ref{quanthopf}. 
In particular, the results in this lemma apply to the local version of 
the Laplace-Beltrami operator on the unit sphere. 

As a consequence of the above analysis, the equivalence in \eqref{GFsA} follows.

\vskip 0.10in
\noindent --------------------------------------
\vskip 0.20in
\begin{minipage}[t]{7.5cm}
\noindent {\tt Anthony Carbery}

\noindent School of Mathematics and 

\noindent Maxwell Inst. for Math. Sciences 

\noindent University of Edinburgh

\noindent King's Buildings, Mayfield Road 

\noindent Edinburgh, EH9 3JZ, UK

\noindent {\tt e-mail}: {\it A.Carbery\@@ed.ac.uk}

\vskip 0.15in

\noindent {\tt Vladimir Maz'ya}

\noindent Department of Math. Sciences

\noindent University of Liverpool

\noindent Liverpool L69 3BX, UK

and

\noindent Department of Mathematics

\noindent Link\"oping University

\noindent Link\"oping SE-581 83, Sweden

\noindent {\tt e-mail}: {\it vlmaz\@@mai.liu.se}

\end{minipage}
\hfill
\begin{minipage}[t]{7.5cm}

\noindent {\tt Marius Mitrea}

\noindent Department of Mathematics

\noindent University of Missouri at Columbia

\noindent Columbia, MO 65211, USA

\noindent {\tt e-mail}: {\it mitream\@@missouri.edu}

\vskip 0.15in

\noindent {\tt David J. Rule}

\noindent Department of Mathematics and 

\noindent Maxwell Inst. for Math. Sciences 

\noindent Heriot-Watt University

\noindent Riccarton 

\noindent Edinburgh, EH14 4AS, UK

\noindent {\tt e-mail}: {\it rule\@@uchicago.edu}

\end{minipage}

\end{document}